%% file: main.tex
\documentclass{article}

\usepackage{arxiv}

\usepackage[utf8]{inputenc}
\usepackage[T1]{fontenc}
\usepackage[english]{babel}
\usepackage{amsmath,amssymb,amsthm}
\usepackage{graphicx}
\usepackage{array}
\usepackage{booktabs}
\usepackage{enumitem}
\usepackage{microtype}
\usepackage[hidelinks]{hyperref}
\graphicspath{{figures/}}

\theoremstyle{plain}
\newtheorem{theorem}{Theorem}[section]
\newtheorem{lemma}[theorem]{Lemma}
\newtheorem{proposition}[theorem]{Proposition}
\newtheorem{corollary}[theorem]{Corollary}

\theoremstyle{definition}
\newtheorem{definition}[theorem]{Definition}
\newtheorem{convention}[theorem]{Convention}
\newtheorem{example}[theorem]{Example}

\theoremstyle{remark}
\newtheorem{remark}[theorem]{Remark}

\theoremstyle{plain}
\newtheorem{mainthm}{Theorem}

\newtheorem*{assumptionH}{Assumption (H)}

\newenvironment{displaystatement}
  {\begin{list}{}{\setlength{\leftmargin}{1.2em}\setlength{\rightmargin}{0em}%
                  \setlength{\listparindent}{0pt}\setlength{\itemindent}{0pt}}%
   \item\relax}
  {\end{list}}

\newcommand{\ord}{\operatorname{ord}}

\numberwithin{equation}{section}

\newcommand{\tabtext}[1]{\footnotesize #1}

\def\mscname{{\bfseries \emph{2020 Mathematics Subject Classification}}}
\def\msc#1{\par\addvspace\medskipamount{\rightskip=0pt plus1cm
  \noindent\mscname\enspace\ignorespaces#1\par}}

\setlist[enumerate]{topsep=2pt,itemsep=1pt,leftmargin=2em}
\setlist[itemize]{topsep=2pt,itemsep=1pt,leftmargin=1.6em}

\title{Unsigned Frenet Data of Closed Space Curves:\\ Exact Fibres, Generic Rigidity, and Conditional Stability}

\author{%
  JiYe Liu \\
  School of Artificial Intelligence \\
  Tianjin University, Tianjin, China \\
  \texttt{jayliu@tju.edu.cn}
  \And
  Wenkai Wang \\
  School of Artificial Intelligence \\
  Tianjin University, Tianjin, China \\
  \texttt{wkwang@tju.edu.cn}
  \AND
  Qiang Tian \\
  School of Computer and \\
  Information Engineering \\
  Tianjin Normal University \\
  Tianjin 300384, China \\
  Zhongzhen Hengyu Intelligent Technology (Tianjin) Co., Ltd. \\
  Tianjin, CN \\
  \texttt{tianqiang@tjnu.edu.cn}
  \And
  Wenjun Wang\thanks{Corresponding author.} \\
  School of Artificial Intelligence \\
  Tianjin University, Tianjin, China \\
  \texttt{wjwang@tju.edu.cn}%
}

\date{}

\begin{document}

\maketitle

\input{sections/00-abstract}
\input{sections/01-introduction}
\input{sections/02-unsigned-datum}
\input{sections/03-branch-invariant}
\input{sections/04-rigidity}
\input{sections/05-flexibility}
\input{sections/06-genericity}
\input{sections/07-engine}
\input{sections/08-dichotomy}
\input{sections/09-relation}
\input{sections/10-questions}

\appendix
\input{sections/A-construction-chain}
\input{sections/B-proposition-57}
\input{sections/C-constants}
\input{sections/D-notation}

\input{sections/E-declarations}

\input{sections/bibliography}
\end{document}

%% file: sections/00-abstract.tex
\begin{abstract}
A closed positively curved space curve is determined by its curvature and \emph{signed} torsion up to
an orientation-preserving rigid motion; that sign is the only place the ambient orientation enters.
We ask what survives its loss, for closed embedded curves in $\mathbb R^3$ with $\kappa>0$ compared
pointwise in a common arclength label. The answer is governed by the branch invariant
$c(\tau)$, the number of components left by the infinite-order zero set of $\tau$: the smooth signed
lifts of $|\tau|$ number exactly $2^{c(\tau)}$, and reduce to $\{\tau,-\tau\}$ precisely when
$c(\tau)\le1$. Hence a given unsigned datum is carried by at most $2^{c(\tau)}$ classes modulo
$SE(3)$, and by a single $E(3)$-orbit when $c(\tau)\le1$. Both extremes
occur: for arbitrary knot types $K_1,\dots,K_m$ there is a datum with fibre \emph{exactly} $2^m$
classes modulo $SE(3)$, realising all connected sums of the $K_i$ and their mirrors; under a
chirality hypothesis these are $2^m$ knot types. Conversely, curves with only simple torsion zeros
are open and dense, hence residual, among parametrised $C^r$ embeddings ($r\ge4$), and each is
determined up to $E(3)$ by its datum. No uniform quantitative form of this rigidity exists; but on
each stratum $\Delta=\inf_s\sqrt{\tau^2+(\tau')^2}\ge\delta>0$ with uniform $C^5$ and curvature
bounds the orbit distance obeys a log-Lipschitz bound, whose optimal constants diverge as
$\delta\downarrow0$ on the strata containing a fixed exact ambiguous pair. The engine is a
one-dimensional inverse estimate for the signed square root, logarithmically optimal at that level.
\end{abstract}


\msc{Primary 53A04; Secondary 26A16, 26E10, 57K10.}

%% file: sections/01-introduction.tex
\section{Introduction}
\label{sec:introduction}

\subsection{Forgetting the sign of the torsion}
\label{sec:1.1}

A regular closed curve in $\mathbb R^3$ with nowhere vanishing curvature carries two classical scalar
functions of arclength: the curvature $\kappa>0$ and the torsion $\tau$. The fundamental theorem of
space curves states that the \emph{signed} pair $(\kappa,\tau)$ determines the curve up to an
orientation-preserving rigid motion, hence determines its knot type as well. Among these two
functions the torsion is distinguished: it is the only place where the ambient orientation enters.
A reflection of $\mathbb R^3$ leaves $\kappa$ and $|\tau|$ untouched and replaces $\tau$ by $-\tau$.
Whenever the Frenet data are recorded by a procedure that has no access to an orientation --- a
measurement, a discretisation, an invariant built from unoriented quantities --- what is available is
not $(\kappa,\tau)$ but the \textbf{unsigned Frenet datum}
\[
F_{(\kappa,|\tau|)}(\gamma)=\bigl(\kappa(\cdot),\,|\tau(\cdot)|\bigr).
\]
The question of this paper is what such a datum still determines.

There is a natural guess. Since a reflection is exactly what flips the sign of $\tau$, one expects
the loss to be a single global mirror ambiguity: two curves with the same $(\kappa,|\tau|)$ should
differ by an isometry of $\mathbb R^3$, and their knot types should therefore agree up to mirror
image. This guess is correct for a large and generic class of curves, and where it fails it fails in the
strongest way available: the set of curves sharing a datum can be made exponentially large, with
arbitrary prescribed knot types. Our purpose is to draw the line between
the two regimes, to show that the line is drawn by one explicitly defined branch invariant, and to make
the failure quantitative on both sides of it.

Throughout, curves are compared \textbf{through the same arclength label}: two curves of length $L$ are
compared through the functions they induce on one and the same circle $\mathbb R/L\mathbb Z$.
Without this convention the question degenerates, since a reparametrisation could be absorbed into
the comparison. All fibre and stability comparisons below are therefore made between parametrised
curves in a common arclength label; we pass to the reparametrisation quotient
$\mathcal E^r/\mathrm{Diff}^+(S^1)$ at exactly one place, the genericity statement of
Corollary~\ref{cor:6.11}, where the assertion is about the space of curves and not about a
comparison of two data. The labelling convention is a genuine hypothesis and it is used everywhere:
it is what makes the pointwise comparison of two data meaningful.

\subsection{The branch invariant and the main results}
\label{sec:1.2}

Let $\gamma$ be a closed embedded curve as above, with torsion $\tau$, and let
\[
Z_\infty(\tau):=\{s\in S^1:\ \tau^{(k)}(s)=0\ \text{for all }k\ge0\}
\]
be the set where the torsion vanishes to infinite order. Write $c(\tau)$ for the number of connected
components of $S^1\setminus Z_\infty(\tau)$; it is a \textbf{cardinal},
\[
c(\tau)\in\mathbb N_0\cup\{\aleph_0\},
\]
not in general an integer, and it is what governs the entire qualitative picture. Three counts are
attached to a datum $d=(\kappa,|\tau|)$ and are kept apart throughout (Definition~\ref{def:2.4}):
the set $\mathrm{Fib}_{SE}(d)$ of curves carrying $d$ in the common arclength label modulo
orientation-preserving congruence, the same set modulo the full Euclidean group,
$\mathrm{Fib}_E(d)$, and the set $\mathcal F(d)$ of knot types realised in the fibre. Always
$|\mathcal F(d)|\le|\mathrm{Fib}_{SE}(d)|$ and
$|\mathrm{Fib}_E(d)|\le|\mathrm{Fib}_{SE}(d)|\le2|\mathrm{Fib}_E(d)|$.

\begin{displaystatement}
\begin{mainthm}[the branch invariant governs the fibre; Theorems~\ref{thm:3.7},
\ref{thm:3.9} and \ref{thm:4.1}, Corollaries~\ref{cor:3.8} and \ref{cor:4.5}]
\label{thm:A}
Let $\gamma$ be as above, with datum $d=(\kappa,|\tau|)$.
\begin{enumerate}
\item \emph{(One-dimensional classification.)} The smooth signed lifts of the unsigned torsion are in
   bijection with the locally constant signs on the components of $S^1\setminus Z_\infty(\tau)$, so
   $|\mathcal L(\tau)|=2^{c(\tau)}$; at that level the criterion is \textbf{exact}, the lifts reducing
   to $\{\tau,-\tau\}$ \textbf{if and only if} $c(\tau)\le1$.
\item \emph{(Geometric consequence.)} $\bigl|\mathrm{Fib}_{SE}(d)\bigr|\le2^{c(\tau)}$, and therefore
   also $|\mathcal F(d)|\le2^{c(\tau)}$.
\item \emph{(Rigidity.)} \textbf{Whenever $c(\tau)\le1$}, the whole fibre lies in a single
   $E(3)$-orbit: every competitor equals $Q\gamma+a$ with $Q\in O(3)$, $a\in\mathbb R^3$, so
   $|\mathrm{Fib}_E(d)|=1$ and $|\mathrm{Fib}_{SE}(d)|\le2$, the two classes being exchanged by a
   reflection. This holds in particular whenever every zero of $\tau$ has finite order, and whenever
   $\gamma$ is real-analytic.
\end{enumerate}
\end{mainthm}
\end{displaystatement}

Two remarks on how to read this. First, the criterion is exact at the level of signed lifts but only
one implication holds at the level of curves: when $c(\tau)\ge2$ the extra signed lifts exist, but a
lift must still integrate to a \emph{closed embedded} curve to enter the fibre, and closure is a
non-trivial constraint on the Frenet system, characterised by Grinevich--Schmidt \cite{ref13}. It can
eliminate individual sign choices and push the count strictly below $2^{c(\tau)}$
(Remark~\ref{rem:5.19}); Theorem~\ref{thm:B} supplies data for which it eliminates none, so the
upper bound is sharp even though it is not always attained.

Second, the invariant is the right one. Neither the size of the zero set $Z(\tau)$ nor its
finiteness is the relevant quantity: a torsion with infinitely many zeros accumulating at a point
still has only the two lifts $\pm\tau$, while a torsion vanishing on two arcs has four
(Example~\ref{ex:3.10}). Only the topology of $Z_\infty(\tau)$ enters, and only through the number of components
$c(\tau)$.

The upper bound of Theorem~\ref{thm:A} is sharp, and sharp with arbitrary prescribed topology.

\begin{displaystatement}
\begin{mainthm}[exact fibre; Theorems~\ref{thm:5.8}, \ref{thm:5.11} and \ref{thm:5.16}]
\label{thm:B}
For every $m\ge1$ and arbitrary knot types $K_1,\dots,K_m$ there is a closed embedded $C^\infty$
curve $\gamma$ of positive curvature whose torsion satisfies $c(\tau_0)=m$, together with a family
$\{\gamma_\varepsilon\}_{\varepsilon\in\{\pm1\}^m}$ of closed embeddings sharing one and the same
pointwise unsigned datum $d_m$ along one arclength label, such that
$[\gamma_\varepsilon]=\#_{i=1}^mK_i^{\varepsilon_i}$, with $K^{+1}=K$ and $K^{-1}=\overline K$ the
mirror image. \textbf{With no hypothesis on the $K_i$}:
$\bigl|\mathrm{Fib}_{SE}(d_m)\bigr|=2^m$ \textbf{exactly}, with exactly one member of the family in
each class, and $\bigl|\mathrm{Fib}_{E}(d_m)\bigr|=2^{m-1}$; equivalently, every smooth signed lift of
$|\tau_0|$ is realised by a closed embedding, uniquely up to $SE(3)$. The $2^m$ classes realise all
the connected sums $\#_iK_i^{\varepsilon_i}$, but these knot types \emph{may repeat}. \textbf{Under the
chirality hypothesis (H)} --- the $K_i$ prime and chiral with $\{K_i,\overline{K_i}\}_{i\le m}$
pairwise distinct --- they do not: $|\mathcal F(d_m)|=2^m$, forming $2^{m-1}$ mirror classes.
\end{mainthm}
\end{displaystatement}

So the unsigned datum does not determine the knot type, not even up to mirror image, and the failure
is exponential in the number of prescribed factors. The family exhausts the fibre rather than sitting
inside a larger one; what makes that equality available is that the reference curve can be built with its infinite-order zero
set computed exactly, which is the content of Theorem~\ref{thm:5.8} and the technical heart of the
construction (\S\ref{sec:1.3} below).

Rigidity, meanwhile, is the rule rather than the exception.

\begin{displaystatement}
\begin{mainthm}[genericity; Theorem~\ref{thm:6.8}, Corollaries~\ref{cor:6.10} and \ref{cor:6.11}]
\label{thm:C}
Let $r\ge4$ and let $\mathcal E^r$ be the space of parametrised closed $C^r$ embeddings with
positive curvature. The curves whose torsion has only simple zeros form an \textbf{open and dense},
hence residual, subset $\mathcal G^r\subset\mathcal E^r$, and its image is open and dense in the
reparametrisation quotient $\mathcal E^r/\mathrm{Diff}^+(S^1)$. Curves in $\mathcal E^r$ are not
required to be unit speed; for every $\gamma\in\mathcal G^r$ it is the \emph{arclength
reparametrisation} $\widehat\gamma$ that is determined up to $E(3)$ by its unsigned datum, among all
competitors that are $C^4$, unit speed, positively curved and share its arclength label.
\end{mainthm}
\end{displaystatement}

Theorems~\ref{thm:A}--\ref{thm:C} settle the qualitative structure of the signed lifts, and give a general upper bound
for curve fibres that Theorem~\ref{thm:B} shows to be sharp. The remaining question is quantitative,
and it is the one an application would ask: if two curves have \emph{nearly} the same unsigned datum, are
they nearly congruent? Generic rigidity says nothing about this, and in fact no global answer
exists. But the obstruction can be localised, in the following sense. For a normalised curve
$\alpha$ put
\[
\Delta(\alpha):=\inf_{s}\sqrt{\tau_\alpha(s)^2+\tau_\alpha'(s)^2},
\]
a quantitative non-degeneracy: it is positive exactly when all torsion zeros are simple, and it
measures how simple. Two distances between curves defined on one and the same
$\mathbb R/L_0\mathbb Z$ appear: the \textbf{labelled} orbit distance
$d_{\mathrm{lab}}(\alpha,\beta)=\inf_{g\in E(3)}\|\alpha-g\beta\|_{C^0}$, which compares the
parametrisations in the given label, and the label-free Hausdorff orbit distance $D(\alpha,\beta)$
between the images modulo $E(3)$; always $D\le d_{\mathrm{lab}}$
(Lemma~\ref{lem:8.2}). Put $\rho(t):=t\bigl(1+\log_+\frac1t\bigr)$.

\begin{displaystatement}
\begin{mainthm}[global instability and conditional stability; Theorems~\ref{thm:8.8}, \ref{thm:8.17},
\ref{thm:8.20}, \ref{thm:8.22} and Proposition~\ref{prop:8.23}]
\label{thm:D}
\leavevmode
\begin{enumerate}
\item \emph{(No global modulus.)} There are sequences $\alpha_n,\beta_n$ of smooth curves with simple
   torsion zeros, fixed, distinct and non-mirror knot types, such that
   \[
   \begin{gathered}
   \bigl\|(\kappa_{\alpha_n},\tau_{\alpha_n}^2)-(\kappa_{\beta_n},\tau_{\beta_n}^2)\bigr\|_{C^{r-3}}
   \longrightarrow0,\\[2pt]
   \text{while}\quad D(\alpha_n,\beta_n)\ge\delta_0>0 .
   \end{gathered}
   \]
   Hence no modulus of continuity governs the inverse of the unsigned datum on all of
   $\mathcal G^r$, in either distance.
\item \emph{(Conditional stability on the strata.)} On the normalised strata
   $\mathcal K^5_{\delta,M}(L_0)=\{\|\alpha\|_{C^5}\le M,\ \kappa_\alpha\ge M^{-1},\
   \Delta(\alpha)\ge\delta\}$ one has, for all $\alpha,\beta$ in the strata,
   \[
   D(\alpha,\beta)\ \le\ d_{\mathrm{lab}}(\alpha,\beta)\ \le\ C_{\mathrm{geo}}(M,\delta,L_0)\ \rho\Bigl(\bigl\|(\kappa_\alpha,\tau_\alpha^2)-(\kappa_\beta,\tau_\beta^2)\bigr\|_{C^0}\Bigr),
   \]
   with an explicit constant satisfying $C_{\mathrm{geo}}\le1.6\times10^8L_0^2M^{36}\delta^{-4}$ for
   $M\ge1$ and $0<\delta\le1$.
\item \emph{(Local version.)} The same modulus holds on a relative $C^5$ neighbourhood, inside the
   normalised slice, of any fixed curve with simple torsion zeros (Theorem~\ref{thm:8.20}).
\item \emph{(The uniform constant must and does degenerate.)} For a fixed exact ambiguous pair as in
   Theorem~\ref{thm:B} with $m=2$ there is a finite $M_0$, depending on that pair, such that for
   every $M\ge M_0$ the optimal constant of part 2 diverges as $\delta\downarrow0$. Conversely,
   \emph{every} sequence of pairs obeying uniform $C^5$ and curvature bounds whose data converge
   while their orbit distances stay bounded below must satisfy $\Delta\to0$.
\end{enumerate}
\end{mainthm}
\end{displaystatement}

Parts 2 and 4 are what join the two halves. Within the uniformly controlled class, every fixed
stratum $\Delta\ge\delta>0$ is uniformly stable, no near-collision can avoid degenerating in
$\Delta$, and no uniform constant survives the limit $\delta\downarrow0$.

The mechanism behind Theorem~\ref{thm:D} is one-dimensional, and we isolate it because it is independent of
the geometry and, as far as we know, new (\S\ref{sec:relation}).

\begin{displaystatement}
\begin{mainthm}[inverse stability of the signed square root; Theorems~\ref{thm:7.1} and \ref{thm:7.2}]
\label{thm:E}
Fix $B,\delta,L_0>0$. For $f,g\in C^2(\mathbb R/L_0\mathbb Z)$ with $\|f\|_{C^2}+\|g\|_{C^2}\le B$
and $f^2+(f')^2\ge\delta^2$, $g^2+(g')^2\ge\delta^2$, there is a \textbf{global} sign
$\varepsilon\in\{\pm1\}$ with
\[
\|f-\varepsilon g\|_{L^1}\ \le\ C_A(B,\delta,L_0)\ \rho\bigl(\|f^2-g^2\|_{C^0}\bigr),
\]
and $C_A$ is explicit. The modulus $\rho$ cannot be replaced by the identity: an explicit family
with $B,\delta,L_0$ held fixed forces the ratio of the two sides to grow like $\log\frac1\eta$.
\end{mainthm}
\end{displaystatement}

Thus recovering a function from its square, modulo the global sign that no datum can see, is
log-Lipschitz from $C^0$ to $L^1$ under uniform non-degeneracy, and the logarithmic order is
optimal at this one-dimensional level.

The five statements differ in their objects and in the hypotheses they carry, and it is worth
tabulating them before the proofs begin.

\begin{table}[htbp]
\centering
\setlength{\tabcolsep}{3.5pt}
\renewcommand{\arraystretch}{1.15}
\tabtext{%
\begin{tabular}{@{}>{\raggedright\arraybackslash}p{0.05\textwidth}
                  >{\raggedright\arraybackslash}p{0.155\textwidth}
                  >{\raggedright\arraybackslash}p{0.165\textwidth}
                  >{\raggedright\arraybackslash}p{0.36\textwidth}
                  >{\raggedright\arraybackslash}p{0.16\textwidth}@{}}
\toprule
 & Object & Standing hypotheses & Conclusion & Kind \\
\midrule
\textbf{A} (\S\ref{sec:branch}--\S\ref{sec:rigidity}) &
the lift set $\mathcal L(\tau)$, and the fibre over $(\kappa,|\tau|)$ &
(S1), (S2) only &
exactly $2^{c(\tau)}$ smooth signed lifts, reducing to $\{\tau,-\tau\}$ \textbf{iff} $c(\tau)\le1$; hence $|\mathrm{Fib}_{SE}|\le2^{c(\tau)}$, and $|\mathrm{Fib}_{E}|=1$ \textbf{whenever} $c(\tau)\le1$ &
qualitative; exact at the level of lifts, one-directional at the level of curves \\
\addlinespace
\textbf{B} (\S\ref{sec:flexibility}) &
one constructed datum, for any $m$ and any $K_1,\dots,K_m$ &
(H) for the knot-type count only &
$|\mathrm{Fib}_{SE}|=2^m$ \textbf{exactly} and $|\mathrm{Fib}_{E}|=2^{m-1}$, realising all $\#_iK_i^{\varepsilon_i}$; under (H), $|\mathcal F|=2^m$ in $2^{m-1}$ mirror classes &
construction; shows the bound of A is sharp \\
\addlinespace
\textbf{C} (\S\ref{sec:genericity}) &
the space $\mathcal E^r$ of parametrised $C^r$ embeddings, $r\ge4$ &
positive curvature; simple torsion zeros &
$\mathcal G^r$ is open and dense, hence residual, and each member is determined up to $E(3)$ among $C^4$ competitors &
genericity, in the parametrised space \\
\addlinespace
\textbf{D} (\S\ref{sec:dichotomy}) &
pairs of curves in $\mathcal G^r$, and in $\mathcal K^5_{\delta,M}(L_0)$ &
$r=5$; $\|\alpha\|_{C^5}\le M$, $\kappa_\alpha\ge M^{-1}$, $\Delta\ge\delta$ &
no global modulus on $\mathcal G^r$; $D\le d_{\mathrm{lab}}\le C_{\mathrm{geo}}\,\rho(\Theta)$ on every stratum and on a relative $C^5$ neighbourhood; $C_*\to\infty$ as $\delta\downarrow0$, and every uniformly bounded near-collision must satisfy $\Delta\to0$ &
global instability, conditional stability \\
\addlinespace
\textbf{E} (\S\ref{sec:engine}) &
two functions $f,g\in C^2(\mathbb R/L_0\mathbb Z)$ &
$\|f\|_{C^2}+\|g\|_{C^2}\le B$; $f^2+(f')^2\ge\delta^2$, likewise for $g$ &
$\inf_\varepsilon\|f-\varepsilon g\|_{L^1}\le C_A\,\rho(\|f^2-g^2\|_{C^0})$, with the logarithmic order not removable &
one-dimensional engine; order sharp in one dimension \\
\bottomrule
\end{tabular}}
\caption{\emph{The five main statements, their objects and their hypotheses. The standing
conventions (S1)--(S4) are collected in \S\ref{sec:2.4}; (H) is the chirality hypothesis of
\S\ref{sec:5.3}. Only in D and E does a quantitative non-degeneracy enter.}}
\label{tab:1}
\end{table}

\subsection{Where the difficulty lies}
\label{sec:1.3}

Three points deserve to be identified before the proofs begin, since each marks a place where the
naive approach does not work.

\textbf{(i) The problem is one-dimensional, but the relevant one-dimensional object is invisible in the
values of $\tau$.} If $\widetilde\gamma$ shares the unsigned datum of $\gamma$, then
$\widetilde\tau$ is a \emph{smooth} function with $|\widetilde\tau|=|\tau|$, and by the fundamental
theorem the curve is then determined by that function up to $SE(3)$. Everything therefore reduces
to listing the smooth functions with prescribed absolute value --- a question about one variable. The
subtlety is which feature of $\tau$ controls the answer. It is tempting to expect that many zeros
mean much freedom. That is false in both directions: a torsion with infinitely many simple zeros
accumulating at a point is rigid, whereas a torsion vanishing on two disjoint arcs already admits
four lifts. The sign may be flipped across a zero only if the flip is invisible to every
derivative, so the controlling object is the \emph{infinite-order} zero set $Z_\infty(\tau)$ --- a set that
cannot be read off from finitely many jets --- and only through the number of components of its
complement. This is the classification underlying Theorem~\ref{thm:A}; it restates, on the circle and in
torsion notation, a theorem of Bony--Colombini--Pernazza, and \S\ref{sec:branch} is written to make the transcription
explicit (see \S\ref{sec:1.4} and the attribution note opening \S\ref{sec:branch}).

\textbf{(ii) Exactness in Theorem~\ref{thm:B} requires computing an infinite-order condition while preserving a
topological one, and the two are governed by opposite kinds of hypothesis.} A lower bound of $2^m$
is inexpensive: place $m$ tangles in disjoint balls joined by planar circular arcs (Figure~\ref{fig:2}),
reflect the tangles independently, and observe that reflection changes only the sign of the torsion on the
corresponding arcs. But such a reference curve carries no information about the \emph{size} of
$Z_\infty(\tau_0)$, and without that the count remains a lower bound and the fibre could in
principle be larger. Making the count exact means producing a curve whose infinite-order zero set is
computed exactly --- it must be precisely the planar part --- while every topological feature survives:
embeddedness, the private balls, the tangle types relative to their boundary spheres, and the exactly
circular planar collars on which the reflections are glued. Here the two requirements pull in
opposite directions: the topological features are preserved by \emph{open}, low-order ($C^0$--$C^2$)
conditions, whereas pinning $Z_\infty$ is a high-order, \emph{full-measure} condition, namely relative
$4$-jet transversality. They are compatible because the transversality holds for almost every
arbitrarily small parameter, so the low-order gates may be fixed first and the jet perturbation
chosen inside them afterwards. A second, sharper obstacle sits at the two endpoints of each knotted
core arc, where transversality on a compact core says nothing and the curve must merge into the
exactly planar collar; there we replace the perturbation argument by an exact identity for a normal
displacement of a circular arc, linear in the displacement, which forces non-vanishing on one-sided
neighbourhoods of both endpoints for \emph{every} non-zero amplitude. One-sided planarity then forces
infinite-order flatness at the endpoints themselves. Together these give $Z_\infty(\tau_0)$ exactly,
hence $c(\tau_0)=m$; the details are \S\ref{sec:5.2} and Appendices~\ref{app:A}--\ref{app:B}.

\textbf{(iii) On the quantitative side the two obstacles are sign alignment and the exponential.} Fix two
curves with nearby data. Recovering the torsion from $\tau^2$ costs one sign per branch, and the
branches of $f$ and of $g$ do not sit at the same places: their zeros are merely \emph{near} one another.
The first half of the proof of Theorem~\ref{thm:E} therefore matches the zeros of $f$ and $g$ bijectively and
cyclically, at scale $d\sim\sqrt\eta$, and then shows that the parity of sign changes forces the
branch signs to agree --- a quantitative counterpart of the criterion $c\le1$ of Theorem~\ref{thm:A}, with
``simple zero'' playing the role that ``connected complement'' plays there. The residual error splits
into a misalignment window of width $d$, on which width $\times$ height returns to the linear order
$d^2\sim\eta$, and an outer region on which $|f-g|\le\eta/(\delta t)$. The single integral
$\int_\sigma^r\eta/(\delta t)\,dt$ is the only source of a logarithm in the whole paper, and
Theorem~\ref{thm:7.2} shows it is saturated rather than wasteful.

The second obstacle is the passage from this $L^1$ bound on coefficients to a $C^0$ bound on curves.
The standard route is Grönwall, which would attach a factor $e^{cL_0}$ with $c$ controlled by
$\|\Omega\|$ and hence by $M$ --- fatal on strata with large $M$, and enough to obscure the entire
$\delta$-dependence that Theorem~\ref{thm:D}(4) is about. It is avoidable: the Frenet coefficient matrix is
antisymmetric, so the fundamental solution takes values in $SO(3)$, and the Duhamel kernel
$F_2(t)\Delta\Omega(t)F_1(t)^{\mathsf T}F_1(s)$ is therefore a Frobenius isometry applied to
$\Delta\Omega(t)$. The frame error is bounded by the plain $L^1$ norm of the coefficient difference,
with constant $1$ and no exponential (Lemma~\ref{lem:8.16}) --- and the $L^1$ norm of the coefficient difference
is exactly what Theorem~\ref{thm:E} delivers. This is why the geometric modulus inherits the one-dimensional
modulus without loss.

Finally, the two halves of Theorem~\ref{thm:D} are compatible for a structural reason. A continuous injective map restricted to a compact set has a uniformly continuous inverse;
any counterexample to a uniform modulus must therefore exploit non-compactness, and must escape
along some degeneration. Part 1 escapes along curves whose torsion vanishes on planar arcs, so
$\Delta\to0$ there by inspection. Parts 2 and 4 show that, \emph{within the uniformly controlled
class carrying both the $C^5$ bound and the curvature lower bound}, no escape can avoid degenerating
in $\Delta$, and that this degeneration is genuinely used: the optimal constant on
$\Delta\ge\delta$ is finite for each $\delta>0$ and unbounded as $\delta\downarrow0$.

\subsection{Relation to earlier work}
\label{sec:1.4}

\begin{quote}
\textbf{Input.} The classification of \emph{all} smooth signed lifts of a smooth non-negative datum
in one variable is due to Bony--Colombini--Pernazza \cite[\S1]{ref20}; part 1 of Theorem~\ref{thm:A}
is that result transcribed to $S^1$ and to torsion notation, and the count $2^{c(\tau)}$ is its
immediate cardinality consequence. The only other external input is the constant-torsion
$h$-principle of Ghomi--Raffaelli \cite[Cor.~1.2]{ref3}, used once, in Corollary~\ref{cor:4.6}.

\textbf{What is added here.} (a) The geometric consequences of that classification: the fibre bounds
and rigidity theorems of \S\ref{sec:rigidity}, in the three counts of Definition~\ref{def:2.4}.
(b) The multi-factor local-mirroring construction of \S\ref{sec:flexibility} and the relative
$4$-jet transversality that makes its fibre \emph{exact} rather than merely large
(Theorems~\ref{thm:5.8}, \ref{thm:5.16}). (c) The genericity of rigidity in the parametrised space
and its reparametrisation quotient (\S\ref{sec:genericity}). (d) The quantitative half: the
one-dimensional inverse estimate for the signed square root and its sharpness in order
(\S\ref{sec:engine}), and the conditional geometric stability, the divergence of the uniform
constants and the forced torsion degeneracy of \S\ref{sec:dichotomy}.
\end{quote}

\textbf{Square roots and lifts.} The classification used in \S\ref{sec:branch} belongs to the theory of \emph{admissible
square roots}: continuous $g$ with $g^2=F$ for a given non-negative $F$, sign changes being
permitted. The classical concerns there are existence and optimal regularity, from Glaeser \cite{ref6} to
Bony--Broglia--Colombini--Pernazza \cite{ref22} and Bony--Colombini--Pernazza \cite{ref21}. The classification of \emph{all}
smooth roots, which is what the present question needs, is due to Bony--Colombini--Pernazza \cite[\S1]{ref20};
\S\ref{sec:branch} is a self-contained restatement of it on $S^1$ in signed-torsion notation, with full proofs
included both for self-containedness and because the finite-regularity variants isolated in
Remark~\ref{rem:3.4} are used later in Corollary~\ref{cor:6.10}, where the $C^\infty$ classification is unavailable.
The attribution note opening \S\ref{sec:branch} records clause by clause which part of
\cite[\S1]{ref20} supplies which statement. What we add on
this side is the translation of the sign-level classification into rigidity and fibre statements
about space curves (\S\ref{sec:rigidity}).

\textbf{Lifting curves of polynomials.} A closely related line studies the lifting of curves of
polynomials over their invariants \cite{ref7}, \cite{ref16}, \cite{ref17}, \cite{ref18}, \cite{ref19}, \cite{ref9}, surveyed in \cite{ref8}; its main concern
is again existence and regularity. It does contain a rigidity statement: Alekseevsky--Kriegl--Losik--Michor
\cite{ref7} prove that under a \emph{normal non-flatness} condition any two smooth systems of roots differ by a
constant permutation, a conclusion restated by Losik--Rainer \cite{ref19} and by Rainer \cite{ref9}. In the present
setting that condition is equivalent to $Z_\infty(\tau)=\varnothing$ (Proposition~\ref{prop:9.1}), hence
strictly stronger than $c(\tau)\le1$. Section~\ref{sec:relation} makes that comparison precise and locates Theorem~\ref{thm:E}
against the recent quantitative work of Parusiński--Rainer \cite{ref27}, \cite{ref28} on the continuity of root maps.

\textbf{Geometric inputs and neighbours.} Corollary~\ref{cor:4.6} uses, as its only external input, the
constant-torsion $h$-principle of Ghomi--Raffaelli \cite[Cor.~1.2]{ref3}, which supplies in every knot class a
smooth embedded representative with positive curvature and globally constant torsion. Ghomi \cite[Thm.~1.1]{ref2}
plays the analogous role for the curvature: it makes an embedded closed curve's curvature
identically $c$, for any $c\ge\max\kappa_0$, by a $C^1$-small isotopy preserving unit speed, domain
and length. That already shows, in passing, that $\kappa$ alone determines nothing. Given finitely
many prescribed knot types, first rescale unit-speed representatives to one common length $L$, then
fix $c$ above all the rescaled curvatures, then apply the theorem to each; the results share
$\kappa\equiv c$ on $\mathbb R/L\mathbb Z$ while differing in knot type. Grinevich--Schmidt \cite{ref13}
characterise which periodic pairs $(\kappa,\tau)$ integrate to a closed curve; their condition can
rule out individual sign choices, and Remark~\ref{rem:5.19} records that for the construction of \S\ref{sec:flexibility} it rules
out none. Three further
works touch the unsigned datum from other directions and are complementary to the present results:
Bray--Jauregui \cite{ref10} on the geometric cost of prescribing a global sign for $\tau$, Mucci--Saracco \cite{ref11}
on the naturality of $|\tau|$ at low regularity, and Honma--Saeki \cite{ref12} on the total absolute torsion
as an ingredient of an integral invariant; see \S\ref{sec:9.4}. Koch--Rüland--Salo \cite{ref14} survey instability
mechanisms in inverse problems, and we cite them only for the general pattern ``globally unstable,
conditionally stable under a priori bounds''; the mechanism in \S\ref{sec:dichotomy} is different in kind, being the
collapse of branches of the unsigned quotient rather than smoothing-induced compression.

\subsection{Conventions and organisation}
\label{sec:1.5}

All constants below are explicit and auditable, and are sufficient bounds rather than estimates of
the true order; each is recorded where it is proved, as an inequality on a stated range rather than as
an unquantified $O(\cdot)$ symbol (Lemma~\ref{lem:8.14}, Theorem~\ref{thm:7.1}, \eqref{eq:8.1},
Appendix~\ref{sec:C.3}). Sharpness of a modulus of continuity is asserted only for the function order
$\eta\log\frac1\eta$ of Theorem~\ref{thm:7.2}, and only at the one-dimensional level. The standing hypotheses --- positive
curvature, closedness, embeddedness and the common arclength label --- are collected in \S\ref{sec:2.4}, and each
statement carries its own regularity requirement explicitly.

Section~\ref{sec:datum} fixes the curve class and the information operator and establishes the two symmetries that
act on the Frenet data, reducing the whole problem to a one-dimensional one. Section~\ref{sec:branch} solves that
one-dimensional problem: it classifies the smooth signed lifts of an unsigned torsion and introduces
the branch invariant $c(\tau)$. Section~\ref{sec:rigidity} transports the classification back to curves and proves the
rigidity half of Theorem~\ref{thm:A} together with the fibre bound. Section~\ref{sec:flexibility} constructs the flexible examples
and computes their fibres exactly, proving Theorem~\ref{thm:B}. Section~\ref{sec:genericity} proves that rigidity is generic
(Theorem~\ref{thm:C}). Section~\ref{sec:engine} proves the one-dimensional stability theorem and its sharpness (Theorem~\ref{thm:E}).
Section~\ref{sec:dichotomy} proves Theorem~\ref{thm:D}: the absence of a global modulus, uniform
and local conditional stability on the non-degenerate strata, the divergence of the optimal constant
as $\delta\downarrow0$, and the forced torsion degeneracy of any uniformly bounded near-collision
sequence. Section~\ref{sec:relation} compares the results with the lifting and root-regularity
literature, and \S\ref{sec:questions} lists the questions the paper leaves open. Four appendices contain the
bookkeeping of the construction of \S\ref{sec:flexibility} (Appendix~\ref{app:A}), the proof of the
relative $4$-jet transversality proposition (Appendix~\ref{app:B}), the explicit constant computations
(Appendix~\ref{app:C}) and a notation index (Appendix~\ref{app:D}); the paper closes with the
statements and declarations required by the journal.

%% file: sections/02-unsigned-datum.tex
\section{The unsigned datum and the symmetries acting on it}
\label{sec:datum}

The purpose of this section is to reduce the question of \S\ref{sec:1.1} to a problem about one real
function of one variable. Three facts accomplish that reduction: the torsion is the reflection-odd
part of the Frenet datum (Lemma~\ref{lem:2.7}), parameter reversal contributes no further sign
(Lemma~\ref{lem:2.8}), and the \emph{signed} datum determines the curve up to $SE(3)$ under minimal
regularity (Lemma~\ref{lem:2.10}). Together they say that a competitor sharing the unsigned datum is
completely described by which smooth function with absolute value $|\tau|$ its torsion happens to be.
Section~\ref{sec:branch} then lists those functions.

\subsection{The curve class and the information operator}
\label{sec:2.1}

\begin{definition}[curve class]
\label{def:2.1}
$\gamma:S^1=\mathbb R/L\mathbb Z\to\mathbb R^3$ is $C^\infty$, regular ($\gamma'\neq0$), injective
(an embedding), of unit speed (arclength parameter), of total length $L$.
Sections~\ref{sec:genericity}--\ref{sec:dichotomy} introduce classes of finite regularity.
\end{definition}

\begin{definition}[knot type, mirror]
\label{def:2.2}
A knot type is an ambient isotopy class: there are homeomorphisms $h_t:\mathbb R^3\to\mathbb R^3$
with $h_0=\mathrm{id}$ and $h_1(\gamma_1(S^1))=\gamma_2(S^1)$. Since $h_0=\mathrm{id}$ and $h_t$ is
continuous, $h_1$ is orientation preserving, so ambient isotopy does \textbf{not} include mirroring.
Knot types are \textbf{unoriented} by default. We write $SE(3)$ for the rigid motions with $\det=+1$
--- \emph{orientation-preserving Euclidean congruences} --- and $E(3)=O(3)\ltimes\mathbb R^3$ for the
full group of Euclidean congruences, which includes reflections; $\overline K=\mathrm{mirror}(K)$.
For a sign $\varepsilon\in\{\pm1\}$ we abbreviate
\[
K^{+1}:=K,\qquad K^{-1}:=\overline K .
\]
The superscript is the sign $\varepsilon$: throughout this paper $K^{-1}$ is the mirror image of $K$,
not the reverse knot, which does not occur here since knot types are unoriented. The connected sum $\#$ is well defined, commutative and
associative on unoriented knot types, and $\mathrm{mirror}(K\#J)=\overline K\#\overline J$.
\end{definition}

\begin{definition}[Frenet data; the information operator]
\label{def:2.3}
Assume $\kappa=|\gamma''|>0$ everywhere. Then $T=\gamma'$, $N=T'/\kappa$ and $B=T\times N$ are
pointwise well defined and $C^\infty$, and
\[
T'=\kappa N,\qquad N'=-\kappa T+\tau B,\qquad B'=-\tau N,
\]
with $\tau:=-\langle B',N\rangle\in C^\infty(S^1)$. The \textbf{sign convention is fixed by
$B'=-\tau N$}; we follow do Carmo \cite[\S1-5]{ref1} for the standard account and differ from it only
in fixing that convention explicitly, since the entire paper turns on it. The \textbf{unsigned Frenet
datum} is the image of the information operator
\[
F_{(\kappa,|\tau|)}(\gamma)=\bigl(\kappa,|\tau|\bigr),
\]
read pointwise in the arclength label.
\end{definition}

\begin{definition}[the fibre, and the three counts attached to it]
\label{def:2.4}
Let $d=(\kappa_0,|\tau_0|)$ be an unsigned datum on $\mathbb R/L\mathbb Z$. The \textbf{fibre over $d$}
is the set of curves
\[
\mathrm{Fib}(d):=\bigl\{\gamma\ \text{of the class of Definition~\ref{def:2.1} on }
\mathbb R/L\mathbb Z:\ (\kappa_\gamma,|\tau_\gamma|)=d\ \text{pointwise}\bigr\},
\]
the equality being read \textbf{in the common arclength label}.
Three quotients of it are counted in this paper, and they must be kept apart:
\[
\mathrm{Fib}_{SE}(d):=\mathrm{Fib}(d)/SE(3),\qquad
\mathrm{Fib}_{E}(d):=\mathrm{Fib}(d)/E(3),
\]
\[
\mathcal F(d):=\bigl\{[\operatorname{Im}\gamma]:\gamma\in\mathrm{Fib}(d)\bigr\},
\]
the last being the set of \textbf{knot types} realised in the fibre; we also write
$\mathcal F(\kappa_0,|\tau_0|)$ for it. Since a knot type is an $SE(3)$-invariant and each
$E(3)$-orbit is a union of at most two $SE(3)$-orbits,
\begin{equation}\label{eq:2.1}\tag{2.1}
\bigl|\mathcal F(d)\bigr|\ \le\ \bigl|\mathrm{Fib}_{SE}(d)\bigr|,\qquad
\bigl|\mathrm{Fib}_{E}(d)\bigr|\ \le\ \bigl|\mathrm{Fib}_{SE}(d)\bigr|\ \le\
2\,\bigl|\mathrm{Fib}_{E}(d)\bigr| .
\end{equation}
None of the three inequalities can be reversed in general: a reflection acts trivially on $|\tau|$ but
may or may not change the knot type, and it may or may not move the curve within its $SE(3)$-orbit.
\end{definition}

\begin{remark}[why a compressed datum]
\label{rem:2.5}
If the operator $F$ retained the full embedding or its image, then ``$F$ determines the knot type''
would be true by definition and would not be a result. The substantive question always concerns
\textbf{compressed or derived} geometric information, of which $F_{(\kappa,|\tau|)}$ is the smallest
natural example that still contains the whole classical Frenet content except for one $\mathbb Z_2$.
\end{remark}

\begin{definition}[order of a zero]
\label{def:2.6}
For $f\in C^\infty(S^1)$ and $p\in S^1$ put
$\ord_p(f):=\min\{k\ge0:f^{(k)}(p)\neq0\}\in\mathbb N\cup\{\infty\}$. Write $Z(f)=f^{-1}(0)$,
$Z_{\mathrm{fin}}(f)=\{p\in Z:\ord_pf<\infty\}$ and
$Z_\infty(f)=\{p:\ord_pf=\infty\}=\bigcap_{j\ge0}\{f^{(j)}=0\}$. Both $Z$ and $Z_\infty$ are closed,
and $Z=Z_{\mathrm{fin}}\sqcup Z_\infty$.
\end{definition}

\subsection{Reflection is the only source of the sign}
\label{sec:2.2}

\begin{lemma}[reflection flips the sign of the torsion]
\label{lem:2.7}
Let $R\in O(3)$ with $\det R=-1$ and $\widetilde\gamma=R\gamma$. Then pointwise
$\widetilde\kappa=\kappa$, $\widetilde\tau=-\tau$ and $|\widetilde\tau|=|\tau|$.
\end{lemma}

\begin{proof}
$R$ is a constant linear isometry, so $\widetilde T=RT$ and $\widetilde T'=RT'=\kappa RN$, whence
$\widetilde\kappa=\kappa$ and $\widetilde N=RN$. From $(Ra)\times(Rb)=\det(R)\,R(a\times b)$ we get
$\widetilde B=\det(R)R(T\times N)=-RB$, so $\widetilde B'=-RB'=-R(-\tau N)=\tau\widetilde N$;
comparison with $\widetilde B'=-\widetilde\tau\widetilde N$ gives $\widetilde\tau=-\tau$.
\end{proof}

\begin{lemma}[parameter reversal does not flip the sign]
\label{lem:2.8}
Put $\beta(s):=\gamma(L-s)$. Then $\kappa_\beta(s)=\kappa(L-s)$ and $\tau_\beta(s)=\tau(L-s)$.
\end{lemma}

\begin{proof}[Proof 1]
$\beta'(s)=-\gamma'(L-s)$, so $T_\beta(s)=-T(L-s)$ and $T_\beta'(s)=T'(L-s)=\kappa(L-s)N(L-s)$; hence
$\kappa_\beta(s)=\kappa(L-s)$ and $N_\beta(s)=N(L-s)$, so $B_\beta=T_\beta\times N_\beta=-B(L-s)$ and
$B_\beta'(s)=B'(L-s)=-\tau(L-s)N_\beta(s)$, which compared with $B_\beta'=-\tau_\beta N_\beta$ gives
$\tau_\beta(s)=\tau(L-s)$.
\end{proof}

\begin{proof}[Proof 2 (parametrisation-free formulas)]
For a general regular curve $\kappa=|\gamma'\times\gamma''|/|\gamma'|^3$ and
$\tau=\langle\gamma'\times\gamma'',\gamma'''\rangle/|\gamma'\times\gamma''|^2$. Under reversal
$\beta'=-\gamma'$, $\beta''=\gamma''$ and $\beta'''=-\gamma'''$, hence
$\beta'\times\beta''=-(\gamma'\times\gamma'')$, and both numerator and denominator are unchanged.
\end{proof}

The second proof is recorded because the parametrisation-free formulas are used again in
\S\ref{sec:flexibility} and \S\ref{sec:genericity}, where curves are not assumed to be unit speed.

\begin{corollary}[covariance of the Frenet datum]
\label{cor:2.9}
Inside the signed datum $(\kappa,\tau)$ the torsion is the \textbf{reflection-odd} component
(Lemma~\ref{lem:2.7}), while reversal of the parameter acts only by the pullback $s\mapsto L-s$ and
produces \textbf{no extra sign} (Lemma~\ref{lem:2.8}). \textbf{At the level of Euclidean covariance},
passing to $|\tau|$ removes the reflection-odd sign of $\tau$ and nothing else. \hfill$\square$
\end{corollary}

The qualification is essential and is not a formality. What the corollary describes is the action of
$E(3)$ on the datum of \emph{one} curve; it says nothing about how many smooth functions have a given
modulus. Recovering a curve from $|\tau|$ requires choosing a smooth signed lift, and
\S\ref{sec:branch} shows that such a lift may carry an \textbf{independent} sign on each component of
$S^1\setminus Z_\infty(\tau)$. So a single global $\mathbb Z_2$ is what the group action discards;
the lifting problem may nonetheless offer $2^{c(\tau)}$ candidates, and both statements are in force
simultaneously. Nor should the corollary be read as saying that $\operatorname{sign}\tau$ alone
determines chirality: reversal also applies the pullback, so the torsion is not pointwise invariant
in a fixed label.

\subsection{The signed datum determines the curve}
\label{sec:2.3}

\begin{lemma}[Frenet uniqueness for equal signed data]
\label{lem:2.10}
Let $\gamma,\widetilde\gamma:\mathbb R/L\mathbb Z\to\mathbb R^3$ be closed unit-speed curves
\textbf{of class at least $C^3$} with $\kappa>0$ and $\widetilde\kappa>0$ everywhere, sharing
\textbf{one and the same arclength label}, and satisfying $\widetilde\kappa=\kappa$ and
$\widetilde\tau=\tau$ (\textbf{signed}) pointwise. Then $\widetilde\gamma=Q\gamma+a$ for some
$Q\in SO(3)$ and $a\in\mathbb R^3$; in particular the two curves are ambient isotopic.
\end{lemma}

\begin{proof}
First, the regularity suffices. From $\gamma\in C^3$ we get $T=\gamma'\in C^2$ and
$T'=\gamma''\in C^1$; where $\kappa>0$ one has $\kappa=|\gamma''|\in C^1$, $N=\gamma''/\kappa\in C^1$
and $B=T\times N\in C^1$, so the Frenet frame is $C^1$, $B'$ is continuous, and
$\tau=-\langle B',N\rangle$ is continuous. Continuity of the coefficients is all that the following
linear ODE argument needs.

Both frames $F,\widetilde F$ are therefore everywhere defined, of class $C^1$, take values in
$SO(3)$, and satisfy the \textbf{same} linear system
\[
F'=F\,\Omega(s),\qquad
\Omega=\begin{pmatrix}0&-\kappa&0\\ \kappa&0&-\tau\\ 0&\tau&0\end{pmatrix},
\]
whose coefficients depend only on the common pair $(\kappa,\tau)$; a linear ODE with merely
continuous coefficients has a unique solution. Put $Q:=\widetilde F(0)F(0)^{-1}\in SO(3)$. Then $QF$
and $\widetilde F$ solve the same equation with the same initial value, so $\widetilde F=QF$; in
particular $\widetilde T=QT$, and integrating $\widetilde\gamma'=\widetilde T=QT=(Q\gamma)'$ gives
$\widetilde\gamma=Q\gamma+a$. No closing argument is required: both curves are \emph{given} as closed
curves, $\Omega$ is $L$-periodic, and uniqueness on $[0,L]$ extends by periodicity. Finally $SE(3)$
is path connected to the identity, which produces an ambient isotopy.
\end{proof}

\begin{remark}
\label{rem:2.11}
Lemma~\ref{lem:2.10} assumes nothing whatever about orders of zeros or about $Z_\infty$, so it
applies verbatim in the degenerate situations of \S\ref{sec:rigidity} --- in particular when
$\tau\equiv0$, where every argument based on isolated zeros breaks down. It is the only place where
the classical fundamental theorem is used, and it is used in this minimal-regularity form because the
competitors in Corollary~\ref{cor:6.10} are merely $C^4$.
\end{remark}

Combining, we obtain the reduction announced at the head of the section. If $\widetilde\gamma$ shares
the unsigned datum of $\gamma$, then $\widetilde\tau$ is a smooth function with
$|\widetilde\tau|=|\tau|$; by Lemma~\ref{lem:2.10} that function determines $\widetilde\gamma$ up to
$SE(3)$, and by Lemma~\ref{lem:2.7} the two admissible global signs correspond to the two
determinants in $O(3)$. Everything therefore depends on how many smooth functions have absolute
value $|\tau|$.

\subsection{Standing assumptions}
\label{sec:2.4}

Unless a statement says otherwise, the following are in force throughout.

\begin{quote}
\textbf{(S1)} Curves are closed, embedded, of unit speed, and have \textbf{positive curvature
everywhere}. Points of zero curvature lie outside the scope: there the Frenet frame and $\tau$ are
undefined. \textbf{(S2)} Two curves are compared only \textbf{through one and the same arclength
label}, so that their data may be subtracted pointwise. No infimum over labels or reparametrisations
is taken. \textbf{(S3)} Regularity is $C^\infty$ in \S\S\ref{sec:branch}--\ref{sec:flexibility},
$C^r$ with $r\ge4$ in \S\ref{sec:genericity} and \S\ref{sec:8.1}, and $r=5$ from \S\ref{sec:8.3}
onwards, where $\tau\in C^{r-3}=C^2$ is what the computation requires. \textbf{(S4)} Knot types are
unoriented, and ambient isotopy does not include mirroring (Definition~\ref{def:2.2}).
\end{quote}

%% file: sections/03-branch-invariant.tex
\section{The branch invariant: smooth signed lifts of the unsigned torsion}
\label{sec:branch}

By the reduction just made, the object to be classified is
\[
\mathcal L(f):=\{g\in C^\infty(S^1):\ |g(s)|=|f(s)|\ \text{ for all }s\},
\]
the set of smooth signed lifts of a given unsigned datum, with $f=\tau$ in the application. In the
literature a continuous $g$ with $g^2=F$, sign changes being allowed, is called an \textbf{admissible
square root} of the non-negative function $F$ (the terminology of \cite{ref21}), so $\mathcal L(f)$ is the
set of smooth admissible square roots of $F=f^2$. The classical questions in that direction concern
\textbf{existence and optimal regularity}: \cite{ref22} shows that in one variable one cannot in general require
better than $g\in C^1$, while $g\in C^2$ --- and no better --- is available when $F$ vanishes at all its
local minima, and \cite{ref21} gives the corresponding optimal statements from $C^6$ on. The question needed
here is of a different kind: given that $F$ is already $C^\infty$ and that one smooth root is known,
\textbf{list all} smooth roots.

\begin{quote}
\textbf{Attribution.} That question is answered by Bony--Colombini--Pernazza \cite[\S1]{ref20}, and Theorem~\ref{thm:3.7},
Corollary~\ref{cor:3.8} and Theorem~\ref{thm:3.9} below are a restatement of their classification on the circle $S^1$
and in signed-torsion notation. \footnote{Clause by clause: the independent choice of signs on the components of the complement of
the flat set, and the resulting bijection, are stated in the unnumbered paragraph of
\cite[\S1]{ref20} that follows Def.~1.1 and precedes Lem.~1.2, simultaneously for
$m=1,2,\dots,\infty$; for $m=\infty$ the convention of \cite[(1.1)]{ref20} makes the closed set $G$
there exactly the infinite-order flat set, i.e.\ the $Z_\infty$ used here. The regularity statement
for finite $m$, together with the gluing across a possibly non-discrete $G$, is
\cite[Prop.~1.4]{ref20}; since that proposition does not admit $m=\infty$, the upgrade to $C^\infty$
must be taken from \cite[Cor.~1.5]{ref20} and the sentence following it.} Two transcriptions are
performed below: from an interval to $S^1$, and from the coefficient datum $F=f^2$ to the torsion
$\tau$. We nevertheless reproduce complete proofs, both to keep the paper self-contained and because
the finite-regularity variants isolated in Remark~\ref{rem:3.4}(iv) --- which the $C^\infty$ statement
does not contain --- are exactly what Corollary~\ref{cor:6.10} needs, where competitors are merely
$C^4$ and $Z_\infty$ is not even defined.
\end{quote}

\subsection{The structure of the zero set}
\label{sec:3.1}

The first three lemmas isolate the only mechanism by which a sign can, or cannot, be flipped.

\begin{lemma}[the ratio is locally constant]
\label{lem:3.1}
On the open set $S^1\setminus Z(f)$ the ratio
$\varepsilon:=g/f$ takes values in $\{\pm1\}$ and is locally constant.
\end{lemma}

\begin{proof}
Where $f\neq0$ the function $\varepsilon$ is continuous with $|\varepsilon|\equiv1$, and a
continuous map into a discrete set is locally constant.
\end{proof}

\begin{lemma}[finite-order zeros are isolated]
\label{lem:3.2}
If $\ord_pf=k<\infty$ then
$f(t)=(t-p)^kh(t)$ with $h\in C^\infty$ and $h(p)=f^{(k)}(p)/k!\neq0$; hence $p$ is isolated in
$Z(f)$.
\end{lemma}

\begin{proof}
Taylor with integral remainder gives
$f(t)=\frac{(t-p)^k}{(k-1)!}\int_0^1(1-u)^{k-1}f^{(k)}(p+u(t-p))\,du=:(t-p)^kh(t)$; the integrand is
smooth, so $h\in C^\infty$, and $h\neq0$ near $p$ by continuity.
\end{proof}

\begin{lemma}[a finite-order zero forbids a sign flip]
\label{lem:3.3}
Let $I\ni p$ be an open interval,
$f\in C^\infty(I)$ with $\ord_pf=k<\infty$, and let $g\in C^{k}(I)$ satisfy $|g|=|f|$
on $I$. Then there are a neighbourhood $J$ of $p$ and $\eta\in\{\pm1\}$ with $g\equiv\eta f$ on $J$.
\end{lemma}

\begin{proof}
By Lemma~\ref{lem:3.2} choose $J=(p-\delta,p+\delta)$ with $f\neq0$ on $J\setminus\{p\}$. By
Lemma~\ref{lem:3.1} the ratio $\varepsilon=g/f$ equals a constant $\varepsilon_-$ on $J^-$ and a constant
$\varepsilon_+$ on $J^+$. Suppose $\varepsilon_-\neq\varepsilon_+$; replacing $g$ by
$\varepsilon_-g$ we may assume $g=f$ on $J^-$ and $g=-f$ on $J^+$. Since $g\in C^k$, the derivative
$g^{(k)}$ is continuous at $p$, so
$g^{(k)}(p)=\lim_{t\uparrow p}f^{(k)}(t)=f^{(k)}(p)$ and
$g^{(k)}(p)=\lim_{t\downarrow p}\bigl(-f^{(k)}(t)\bigr)=-f^{(k)}(p)$; subtracting gives
$2f^{(k)}(p)=0$, contradicting $\ord_pf=k$. Hence
$\varepsilon_-=\varepsilon_+=:\eta$, and at $p$ both sides vanish.
\end{proof}

\begin{remark}[four observations on Lemma~\ref{lem:3.3}]
\label{rem:3.4}
(i) The contradiction occurs \textbf{exactly at order $k$}, the jump being $2f^{(k)}(p)$; derivatives of
order $<k$ have two-sided limits $0$ and produce no contradiction. This is why an infinite-order
zero is the only place where a flip can hide.
(ii) Only $g\in C^k$ is used, and $g=\operatorname{sign}(t)\,t^k=t^{k-1}|t|$ lies in
$C^{k-1}\setminus C^{k}$, so the hypothesis cannot be weakened.
(iii) The \textbf{parity of $k$ never enters the proof}: $k=1$ (with $f=t$, $g=|t|$) and $k=2$ (with
$f=t^2$, $g=t|t|$) behave identically.
(iv) \textbf{Finite-regularity versions.} Lemma~\ref{lem:3.1} needs only continuity of $f$ and $g$; Lemma~\ref{lem:3.2}, for
$f\in C^k$ with $\ord_pf=k$, gives $f=(t-p)^kh$ with $h$ merely \textbf{continuous} and
$h(p)\neq0$, which still yields isolation; and Lemma~\ref{lem:3.3} needs only $f,g\in C^k$. Hence for $k=1$
all three lemmas apply to functions that are \textbf{merely $C^1$}. This is precisely the situation of
Corollary~\ref{cor:6.10}, where the $C^\infty$ classification below is unavailable.
\end{remark}

\begin{lemma}[a gluing fact]
\label{lem:3.5}
Let $u:J\to\mathbb R$ be continuous, $C^1$ on $J\setminus\{p\}$, with
$\lim_{t\to p}u'(t)=c$. Then $u'(p)=c$ and $u\in C^1(J)$.
\end{lemma}

\begin{proof}
By the mean value theorem $\frac{u(t)-u(p)}{t-p}=u'(\xi_t)$ with $\xi_t\to p$, so the
difference quotient tends to $c$; and $u'$ is continuous at $p$.
\end{proof}

\begin{proposition}[structure of $Z$]
\label{prop:3.6}
(i) Every point of $Z_{\mathrm{fin}}$ is isolated in $Z$;
(ii) every accumulation point of $Z$ lies in $Z_\infty$; (iii) if $Z_\infty=\varnothing$ then $Z$ is
\textbf{finite}; (iv) $S^1\setminus Z_\infty$ is open, its connected components are open arcs (the single
component being $S^1$ itself when $Z_\infty=\varnothing$), and their number is
$c\in\{0,1,2,\dots\}\cup\{\aleph_0\}$; moreover $c=0\iff Z_\infty=S^1\iff f\equiv0$.
\end{proposition}

\begin{proof}
(i) is Lemma~\ref{lem:3.2}. (ii) An accumulation point $q$ lies in the closed set $Z$; if
$\ord_qf<\infty$ then $q$ is isolated by (i), a contradiction. (iii) Every point of
$Z$ is isolated in $Z$, so for each $p\in Z$ pick an open $U_p$ with $U_p\cap Z=\{p\}$; $Z$ is closed
in the compact $S^1$, hence compact, and a finite subcover shows $Z$ is finite. (iv) Components of
an open set are open; a connected open subset of $S^1$ is an open arc or all of $S^1$; an open set
is a countable disjoint union of open arcs. Finally $Z_\infty=S^1$ means that all derivatives vanish
identically, i.e. $f\equiv0$.
\end{proof}

\subsection{Classification, count, and the rigidity criterion}
\label{sec:3.2}

We can now list $\mathcal L(f)$. The statement is that the only freedom is a choice of sign on each
component of $S^1\setminus Z_\infty(f)$, and that the choices are \textbf{independent}.

\begin{theorem}[{classification; restatement of \cite[\S1]{ref20} on $S^1$}]
\label{thm:3.7}
Put
$\mathcal E(f):=\{\varepsilon:S^1\setminus Z_\infty(f)\to\{\pm1\}\ \text{locally constant}\}$. Then
\[
\Phi:\mathcal E(f)\to\mathcal L(f),\qquad
(\Phi\varepsilon)(s)=\begin{cases}\varepsilon(s)f(s),&s\notin Z_\infty,\\ 0,&s\in Z_\infty,\end{cases}
\]
is a \textbf{bijection}.
\end{theorem}

\begin{proof}
\textbf{(i) $\Phi$ is well defined.} Put $g=\Phi\varepsilon$. We claim that for every $m\ge0$:
$g\in C^m$, $g^{(m)}\equiv0$ on $Z_\infty$, and $|g^{(m)}|=|f^{(m)}|$ on $S^1\setminus Z_\infty$.
For $m=0$: on $S^1\setminus Z_\infty$ the function $\varepsilon$ is locally constant, so $g$ is
locally $\pm f$; at $p\in Z_\infty$ one has $|g(t)|=|f(t)|\to0=g(p)$. For the induction step, on
$S^1\setminus Z_\infty$ the derivative $g^{(m+1)}$ exists with $|g^{(m+1)}|=|f^{(m+1)}|$, while at
$p\in Z_\infty$ and $t\neq p$
\[
\Bigl|\frac{g^{(m)}(t)-g^{(m)}(p)}{t-p}\Bigr|=\frac{|g^{(m)}(t)|}{|t-p|}
=\begin{cases}0,&t\in Z_\infty,\\[2pt] \dfrac{|f^{(m)}(t)-f^{(m)}(p)|}{|t-p|},&t\notin Z_\infty,\end{cases}
\]
using $g^{(m)}(p)=0=f^{(m)}(p)$ and the inductive hypothesis; the second expression tends to
$|f^{(m+1)}(p)|=0$. Hence $g^{(m+1)}(p)=0$, and continuity follows in the same way. Note that this
argument nowhere assumes that $Z_\infty$ is discrete: it may contain intervals, or be a Cantor set.

\textbf{(ii) $\Phi$ is injective.} If $\Phi\varepsilon_1=\Phi\varepsilon_2$ then
$\varepsilon_1=\varepsilon_2$ on $S^1\setminus Z$, where $f\neq0$; for $p\in Z_{\mathrm{fin}}$ a
punctured neighbourhood lies in $S^1\setminus Z$ by Lemma~\ref{lem:3.2}, and the $\varepsilon_i$ are locally
constant, so they agree at $p$ as well.

\textbf{(iii) $\Phi$ is surjective.} Let $g\in\mathcal L(f)$. For $s\notin Z$ set
$\varepsilon(s)=g(s)/f(s)$. For $p\in Z_{\mathrm{fin}}$, Lemma~\ref{lem:3.3} with
$k=\ord_pf<\infty$ and $g\in C^\infty\subset C^k$ gives a neighbourhood $J$ and a sign
$\eta$ with $g\equiv\eta f$ there, and we set $\varepsilon(p)=\eta$, consistently with the punctured
neighbourhood. Then $\varepsilon$ is locally constant and $g=\varepsilon f$ on
$S^1\setminus Z_\infty$, while on $Z_\infty$ both sides vanish.
\end{proof}

\begin{corollary}[the branch count]
\label{cor:3.8}
$|\mathcal L(f)|=2^{c}$, where $c=c(f)$ is the number of
connected components of $S^1\setminus Z_\infty(f)$, with the conventions $2^0=1$ and
$2^{\aleph_0}=|\mathcal L(f)|$ when $c=\aleph_0$. \hfill$\square$
\end{corollary}

\begin{theorem}[necessary and sufficient rigidity criterion]
\label{thm:3.9}
$\mathcal L(f)\subseteq\{f,-f\}$
\textbf{if and only if} $c(f)\le1$. \hfill$\square$
\end{theorem}

The invariant $c(f)$ is thus the exact measure of the freedom, and the two extremes $c\le1$ and
$c=m$ are the subjects of \S\ref{sec:rigidity} and \S\ref{sec:flexibility} respectively.

\subsection{The criterion is exactly $c\le1$, and not a statement about $Z$}
\label{sec:3.3}

\begin{figure}[tbp]
\centering
\includegraphics{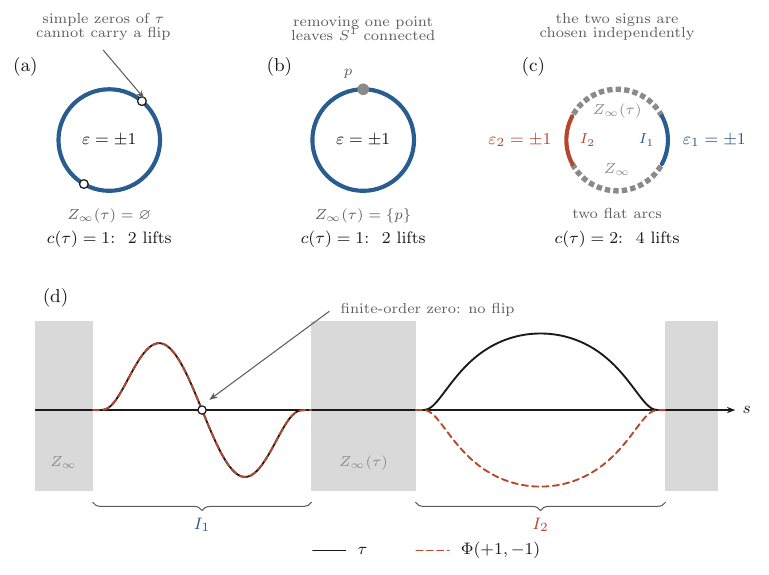}
\caption{\emph{Smooth sign choices across the infinite-order zero set.} Panels (a)--(c): the branch
invariant $c(\tau)$ counts the connected components of $S^1\setminus Z_\infty(\tau)$, and only those
components carry an independent sign. A finite-order zero (a) and a single infinitely flat point (b)
leave the circle connected, so only the global sign survives; two flat arcs (c) disconnect it, and
the signs on the two components $I_1,I_2$ may then be chosen independently. Panel (d): a torsion with
$c(\tau)=2$, drawn as the solid curve, together with the lift $\Phi(\varepsilon)$ for
$\varepsilon=(+1,-1)$, drawn dashed. The two lifts agree on $I_1$ --- including at the simple zero
inside it, where Lemma~\ref{lem:3.3} forbids a flip --- and differ on $I_2$. The leftmost and rightmost shaded
pieces of (d) are identified periodically, so the flat set drawn there has two components, not
three.}
\label{fig:1}
\end{figure}

No simpler invariant would do. The following three examples, drawn in
Figure~\ref{fig:1}, show that finiteness of the zero set is neither necessary nor sufficient for rigidity, and that a single flat
point is harmless while two flat arcs are not.

\begin{example}
\label{ex:3.10}
\leavevmode
\begin{enumerate}
\item \textbf{Two zero intervals: flexibility.} On $S^1=\mathbb R/2\pi\mathbb Z$ take the disjoint open arcs
   $J_1=(\tfrac\pi4,\tfrac{3\pi}4)$ and $J_2=(\tfrac{5\pi}4,\tfrac{7\pi}4)$. Let $x_j$ be the affine
   coordinate mapping $J_j$ onto $(-1,1)$ and put $\varphi_j=\exp\bigl(-1/(1-x_j^2)\bigr)$ for
   $|x_j|<1$ and $\varphi_j=0$ otherwise. Then $\varphi_j\in C^\infty(S^1)$, its positivity set is
   exactly $J_j$, its support is $\overline{J_j}$, and all derivatives vanish on $\partial J_j$. For
   $f=\varphi_1+\varphi_2$ one gets $Z_\infty(f)=S^1\setminus(J_1\cup J_2)$, hence $c=2$ and
   $|\mathcal L(f)|=4$.
\item \textbf{Infinitely many simple zeros accumulating: still rigid.} Put
   $u(\theta):=\sin^2(\theta/2)=\tfrac{1-\cos\theta}2\in[0,1]$, which is $2\pi$-periodic, and
   $f(\theta):=e^{-1/u^2}\sin(1/u)$ for $\theta\neq0$, $f(0):=0$. Then $f\in C^\infty(S^1)$. The
   zeros with $\theta\neq0$ are given by $u=1/(n\pi)$, $n\ge1$, each $n$ contributing two points of
   $(0,2\pi)$, so there are infinitely many, accumulating at $\theta=0$ from both sides; and they
   are \textbf{all of order one}, since there $\sin\frac1u=0$ and $\cos\frac1u=\pm1$, so
   $f'=\mp e^{-1/u^2}u'/u^2$, while $u'=\tfrac12\sin\theta$ vanishes only at $\theta=0,\pi$ and at
   $\theta=\pi$ one has $u=1$, which is not a zero. Hence $\theta=0$ is the unique infinite-order
   zero, $Z_\infty=\{0\}$, $c=1$ and $|\mathcal L(f)|=2$: \textbf{rigid}.
   Thus finiteness of $Z$ is not necessary for rigidity --- only the topology of $Z_\infty$ matters ---
   and the hypothesis of Proposition~\ref{prop:3.6}(iii) cannot be dropped.
   (The choice $u=\sin^2(\theta/2)$ is forced: for $\sin(\theta/2)$ one has $u(\theta+2\pi)=-u(\theta)$,
   and since $e^{-1/u^2}$ is even and $\sin(1/u)$ odd in $u$, $f$ would be antiperiodic and not a
   function on $S^1$.)
\item \textbf{One isolated flat zero is not enough.} If $Z_\infty$ is a single point, then
   $S^1\setminus Z_\infty$ is an open arc, hence connected, so $c=1$ and rigidity holds. To create a
   new lift one must \textbf{disconnect} $S^1\setminus Z_\infty$; on the circle this requires at least
   two flat points, and \S\ref{sec:flexibility} will supply them as the endpoints of planar arcs.
\end{enumerate}
\end{example}

%% file: sections/04-rigidity.tex
\section{Rigidity: from the branch count to the curve}
\label{sec:rigidity}

Section~\ref{sec:branch} answered the one-dimensional question. This section transports the answer back to curves.
Every argument below follows the same pattern: a competitor's torsion lies in $\mathcal L(\tau)$, Theorem~\ref{thm:3.7} lists that set, and Lemma~\ref{lem:2.10} converts each element of the
list into a curve determined up to $SE(3)$. What varies from statement to statement is only how the
list is shown to be short.

\subsection{Finite-order zeros force rigidity}
\label{sec:4.1}

\begin{theorem}[rigidity whenever the branch invariant is at most one]
\label{thm:4.1}
Let
$\gamma,\widetilde\gamma$ be curves as in Definition~\ref{def:2.1} sharing \textbf{one and the same arclength
label}, with $\widetilde\kappa=\kappa$ and $|\widetilde\tau|=|\tau|$ pointwise. If $c(\tau)\le1$ ---
in particular whenever $Z_\infty(\tau)=\varnothing$, and hence whenever every zero of $\tau$ has
finite order --- then $\widetilde\tau\equiv\pm\tau$ and there are $Q\in O(3)$, $a\in\mathbb R^3$ with
$\widetilde\gamma=Q\gamma+a$; in the $+$ case $\det Q=+1$ and in the $-$ case $\det Q=-1$. If both
curves are embeddings, then $[\widetilde\gamma]\in\{[\gamma],\mathrm{mirror}[\gamma]\}$.
\end{theorem}

\begin{proof}
Since $\widetilde\gamma$ is $C^\infty$, the function $\widetilde\tau$ lies in
$\mathcal L(\tau)$, so Theorem~\ref{thm:3.9} gives $\widetilde\tau=\pm\tau$ as soon as $c(\tau)\le1$. For the
stated special case, $Z_\infty=\varnothing$ makes $S^1\setminus Z_\infty=S^1$ connected, hence
$c(\tau)=1$; and if every zero of $\tau$ has finite order then $Z_\infty=\varnothing$ by definition.
The $+$ case is exactly Lemma~\ref{lem:2.10}. In the $-$ case pick a reflection $R$;
by Lemma~\ref{lem:2.7} the curve $R\widetilde\gamma$ has curvature $\kappa$ and torsion
$-\widetilde\tau=\tau$, so Lemma~\ref{lem:2.10} applied to $\gamma$ and $R\widetilde\gamma$ gives
$R\widetilde\gamma=Q_0\gamma+a_0$ with $Q_0\in SO(3)$, whence
$\widetilde\gamma=R^{-1}Q_0\gamma+R^{-1}a_0$ with $\det(R^{-1}Q_0)=-1$.
\end{proof}

\begin{remark}[where the hypothesis sits, and how much regularity the competitor needs]
\label{rem:4.2}
The
hypothesis of Theorem~\ref{thm:4.1} is imposed \textbf{only on the reference torsion $\tau$}; the competitor is
constrained only through $|\widetilde\tau|=|\tau|$. This asymmetry is what allows the regularity of
the competitor to be lowered. If the orders of the zeros of $\tau$ are uniformly bounded by $k$,
then Lemma~\ref{lem:3.3} requires of the competing lift only $\widetilde\tau\in C^{k}$, and a sufficient
condition at the level of curves is
\[
\widetilde\gamma\in C^{k+3},\qquad \widetilde\kappa>0,
\]
because $\widetilde\gamma\in C^{k+3}\Rightarrow\widetilde\gamma'''\in C^k\Rightarrow\widetilde\tau\in C^k$,
the denominator $|\widetilde\gamma'\times\widetilde\gamma''|^2$ being continuous and non-vanishing.
The exponent is optimal in the following sense: the classical torsion involves third derivatives, so
$\widetilde\tau\in C^k$ cannot be inferred from $\widetilde\gamma\in C^{k}$, nor from
$\widetilde\gamma\in C^{k+1}$. For $k=1$ --- the generic case of \S\ref{sec:genericity} --- the sufficient condition reads
$\widetilde\gamma\in C^4$, which is exactly the hypothesis of Corollary~\ref{cor:6.10}.
\end{remark}

\subsection{The planar branch and the real-analytic class}
\label{sec:4.2}

Theorem~\ref{thm:4.1} covers every reference curve with $c(\tau)\le1$, including those whose torsion has
infinite-order zeros. Two such situations still deserve a separate statement: the planar branch,
where strictly more is true than Theorem~\ref{thm:4.1} gives, and the real-analytic class, where the
hypothesis $c(\tau)\le1$ can be verified a priori rather than assumed.

\begin{proposition}[the planar branch $\tau\equiv0$]
\label{prop:4.3}
If $\tau\equiv0$ then $\mathcal L(\tau)=\{0\}$,
so every competitor sharing the unsigned datum has $\widetilde\tau\equiv0$ and, by Lemma~\ref{lem:2.10},
$\widetilde\gamma=Q\gamma+a$ with $Q\in SO(3)$. Moreover $\gamma$ is planar and
$[\gamma]=\mathrm{mirror}[\gamma]$, so the fibre contains exactly one knot type.
\end{proposition}

\begin{proof}
From $|\widetilde\tau|=0$ we get $\widetilde\tau\equiv0$ directly; \textbf{no} discussion of
orders of zeros is involved, and one cannot invoke isolation of zeros of analytic functions --- this
is precisely the case where that fails. Lemma~\ref{lem:2.10} applies because its hypotheses concern only the
pointwise equality of the signed data and involve $Z_\infty$ nowhere. For the last statement,
$\tau\equiv0$ and $\kappa>0$ give $B'=0$, so $B\equiv B_0$ and
$\langle\gamma(s)-\gamma(0),B_0\rangle'=\langle T,B_0\rangle=0$; hence $\gamma$ lies in the plane
$\Pi_0$ through $\gamma(0)$ with normal $B_0$, and the reflection in $\Pi_0$ fixes $\gamma$
pointwise.
\end{proof}

Note that here $Z_\infty=S^1$, so a hypothesis phrased as $Z_\infty=\varnothing$ would exclude this
case although rigidity holds. This is exactly why Theorem~\ref{thm:4.1} is stated in terms of the branch
invariant: $c(\tau)=0$ here, and the sharp criterion supplied by Theorem~\ref{thm:3.9} is $c\le1$, not
$Z_\infty=\varnothing$. What Proposition~\ref{prop:4.3} adds beyond Theorem~\ref{thm:4.1} is the conclusion $\det Q=+1$,
the planarity, and the resulting collapse of the knot-type fibre to a single element.

\begin{corollary}[the real-analytic class]
\label{cor:4.4}
Let $\gamma$ be real-analytic, unit speed, with
$\kappa>0$ everywhere. Then every $C^\infty$ competitor $\widetilde\gamma$ sharing the label and the
unsigned datum satisfies $\widetilde\gamma=Q\gamma+a$ with $Q\in O(3)$, $a\in\mathbb R^3$.
\end{corollary}

\begin{proof}
$T=\gamma'$ is analytic; $\kappa=\sqrt{\langle T',T'\rangle}>0$ is analytic, since
$\sqrt\cdot$ is analytic on the positive axis; hence $N=T'/\kappa$, $B=T\times N$ and
$\tau=-\langle B',N\rangle$ are analytic. Two branches. \textbf{(a) $\tau\equiv0$:} Proposition~\ref{prop:4.3},
which needs no isolation of zeros. \textbf{(b) $\tau\not\equiv0$:} the set
$W=\{q:\ord_q\tau=\infty\}$ is closed, and also open, since $q\in W$ makes the Taylor
series of $\tau$ at $q$ vanish and hence, by analyticity, $\tau\equiv0$ near $q$. As $S^1$ is
connected, $W=\varnothing$ or $W=S^1$, the latter being excluded; so $Z_\infty(\tau)=\varnothing$
and Theorem~\ref{thm:4.1} applies.
\end{proof}

Corollary~\ref{cor:4.4} is a statement \textbf{inside} the real-analytic class. It does not by itself imply that
rigidity is generic, because analyticity is not a generic property in the $C^\infty$ topology; the
genuine genericity statement is Theorem~\ref{thm:6.8}, proved by an entirely different argument.

\subsection{The bound on $|\mathrm{Fib}_{SE}|$, and a rigid representative in every knot class}
\label{sec:4.3}

\begin{corollary}[fibre upper bound, in the three counts of Definition~\ref{def:2.4}]
\label{cor:4.5}
Let $d=(\kappa,|\tau|)$ be the datum of a reference curve $\gamma$, read in the common arclength
label. Then
\[
\bigl|\mathcal F(d)\bigr|\ \le\ \bigl|\mathrm{Fib}_{SE}(d)\bigr|\ \le\
|\mathcal L(\tau)|=2^{\,c(\tau)},\qquad
\bigl|\mathrm{Fib}_{E}(d)\bigr|\ \le\ \bigl|\mathrm{Fib}_{SE}(d)\bigr| .
\]
If moreover $c(\tau)\le1$ then $\bigl|\mathrm{Fib}_{E}(d)\bigr|=1$: the whole fibre lies in one
$E(3)$-orbit, and $\bigl|\mathrm{Fib}_{SE}(d)\bigr|\le2$, the two $SE(3)$-classes --- when there are
two --- being exchanged by a reflection. They coincide, so that
$\bigl|\mathrm{Fib}_{SE}(d)\bigr|=1$, exactly when $\gamma$ is congruent to its own mirror image by
an element of $SE(3)$; this happens in particular when $\tau\equiv0$
(Proposition~\ref{prop:4.3}).
\end{corollary}

\begin{proof}
Any curve realising the datum has torsion in $\mathcal L(\tau)$ by Theorem~\ref{thm:3.7}; given the
torsion, Lemma~\ref{lem:2.10} determines the curve up to $SE(3)$; and the knot type is an
$SE(3)$-invariant, which gives the two displayed chains together with \eqref{eq:2.1}. If $c(\tau)\le1$
then Theorem~\ref{thm:4.1} places every competitor in $\{Q\gamma+a:Q\in O(3),a\in\mathbb R^3\}$, a
single $E(3)$-orbit, which is the union of the $SE(3)$-orbit of $\gamma$ and that of $R\gamma$ for one
fixed reflection $R$. Conversely every such $Q\gamma+a$ does lie in $\mathrm{Fib}(d)$: it is again a
unit-speed closed embedding of the class of Definition~\ref{def:2.1} in the same arclength label,
with the same curvature and, by Lemma~\ref{lem:2.7}, the same $|\tau|$. In particular
$R\gamma\in\mathrm{Fib}(d)$, so the two $SE(3)$-classes coincide precisely when $R\gamma$ is
$SE(3)$-congruent to $\gamma$, which is the stated criterion.
\end{proof}

The bound is informative exactly when $c(\tau)<\aleph_0$; it holds for every reference curve, which is
what no construction can supply. Section~\ref{sec:flexibility} produces a datum for which both
inequalities are equalities.

\begin{corollary}[every knot type has a rigid representative]
\label{cor:4.6}
For every knot type $K$ there is a
representative $\gamma_K$ of the class of Definition~\ref{def:2.1} --- a $C^\infty$ unit-speed closed embedding
with $\kappa>0$ everywhere and $[\gamma_K]=K$ --- which is determined up to $E(3)$ by its unsigned
datum; consequently its knot type is determined up to mirroring by that datum.
\end{corollary}

\begin{proof}
If $K$ is trivial, take a circle in a plane and apply Proposition~\ref{prop:4.3}. Let $K$ be
non-trivial.

\textbf{Step 1 (input).} By the constant-torsion $h$-principle \cite[Cor.~1.2]{ref3} the isotopy class of every
knot contains a $C^\infty$ embedded closed curve $\gamma^\ast$ with $\kappa>0$ everywhere and
torsion equal to a \textbf{global constant} $c$; that the constant may be taken with $c>0$ follows from
Prop.~3.1 and the common-constant assembly in \S3 of \cite{ref3}. If only $c\neq0$ is needed, it also follows
directly: $c=0$ together with $\kappa>0$ gives $B'\equiv0$, so $\gamma^\ast$ is planar, and a planar
simple closed curve is isotopic to a circle by Jordan--Schoenflies, making $K$ trivial.

\textbf{Step 2 (arclength label).} Reparametrise $\gamma^\ast$ by arclength once. The image is unchanged,
hence so is the knot type; $\kappa>0$ and the values of the torsion along the curve are unchanged.
Call the result $\gamma_K$, of length $L$.

\textbf{Step 3 (rigidity).} Since $\tau_{\gamma_K}\equiv c\neq0$ we have $Z(\tau_{\gamma_K})=\varnothing$,
in particular $Z_\infty=\varnothing$. Theorem~\ref{thm:4.1} gives, for every competitor $\widetilde\gamma$ of
the class of Definition~\ref{def:2.1} sharing the arclength label and the pointwise datum,
$\widetilde\gamma=Q\gamma_K+a$ with $Q\in O(3)$, so $[\widetilde\gamma]\in\{K,\overline K\}$.
\end{proof}

Only Step 1 is imported; the remaining two steps are carried out here. The corollary is the reason
why the flexibility of \S\ref{sec:flexibility} must be a statement about \emph{representatives} and not about knot types: every
knot type also has a rigid representative, so flexibility is a property of the curve, not of its
class. Proposition~\ref{prop:5.20} makes this dependence explicit.

%% file: sections/05-flexibility.tex
\section{Flexibility: a datum whose fibre has exactly $2^m$ $SE(3)$-classes}
\label{sec:flexibility}

We now build the extreme case of Theorem~\ref{thm:A}. The goal is a single unsigned datum whose fibre is as
large as Corollary~\ref{cor:4.5} permits, with $c(\tau_0)=m$ prescribed and with prescribed knot types in every
branch; see Figure~\ref{fig:2} for the shape of the construction. The specification the reference
curve must meet is:

\begin{quote}
\textbf{(R1)} it is a $C^\infty$ unit-speed closed embedding with $\kappa>0$ everywhere;
\textbf{(R2)} it carries $m$ tangles realising $K_1,\dots,K_m$ inside pairwise disjoint balls that are
symmetric under one fixed reflection $R$, joined by arcs lying \textbf{exactly} in the mirror plane of
$R$, so that reflecting any subfamily of tangles produces a smooth curve with the same unsigned
datum;
\textbf{(R3)} its torsion satisfies $Z_\infty(\tau_0)=S^1\setminus\coprod_i\Gamma_i$ \textbf{exactly}, where
$\Gamma_i$ is the knotted core arc of the $i$-th ball, so that $c(\tau_0)=m$.
\end{quote}

Requirement (R2) is what produces $2^m$ curves; requirement (R3) is what makes $2^m$ the whole fibre
rather than a lower bound, through Corollary~\ref{cor:4.5}. The two pull against each other, as explained in
\S\ref{sec:1.3}(ii): (R1) and (R2) are preserved by open low-order conditions, whereas (R3) is a statement about
infinite order. The construction below meets both by fixing the low-order gates first and choosing
the high-order perturbations inside them --- legitimate because the transversality involved holds for
almost every arbitrarily small parameter.

Throughout \S\ref{sec:flexibility} let $\Pi=\{z=0\}$ and $R(x,y,z)=(x,y,-z)$. A \textbf{planar circular-arc collar} is an arc
inside $\Pi$ of constant curvature $\kappa_c>0$ and $\tau\equiv0$, defined on an open arclength
interval; $R$ is the identity on it. We write $\kappa_c$ always for the constant curvature \textbf{value}
of the base circle and of the collars, and $\kappa_0(\cdot),\tau_0(\cdot)$ always for the curvature
and torsion \textbf{functions} of the reference curve; the two are kept distinct throughout.

\begin{figure}[tbp]
\centering
\includegraphics{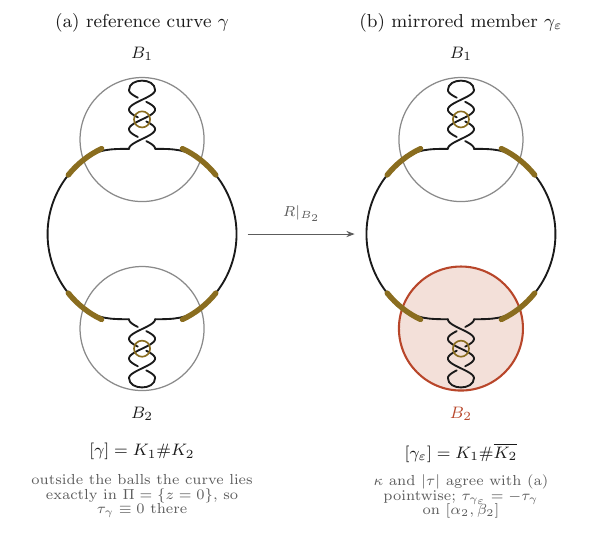}
\caption{\emph{Independent local mirroring inside private symmetric balls}, drawn looking down the
$z$-axis onto $\Pi=\{z=0\}$, in which view the reference curve is an honest knot diagram. The case
$m=2$ is shown; it is the one used in Example~\ref{ex:5.14} and throughout \S\ref{sec:dichotomy}, and the construction is
identical for any $m$. (a) Outside the balls the curve lies exactly in $\Pi$, so
$\tau_\gamma\equiv0$ there; the thick arcs straddling each $\partial B_i$ are the exact
$\Pi$-circular collars. (b) Applying $R|_{B_2}\colon(x,y,z)\mapsto(x,y,-z)$ to the shaded ball
exchanges every one of its crossings --- the gold marker is not part of the curve, and picks out the
middle crossing of each tangle; compare that crossing in the two panels, and note that
the marked crossing of $B_1$ is unchanged. This leaves $\kappa$ and $|\tau|$ unchanged pointwise and
replaces the factor $K_2$ by its mirror. The signed torsion satisfies
$\tau_{\gamma_\varepsilon}=-\tau_\gamma$ on $[\alpha_2,\beta_2]$ and $=\tau_\gamma$ off it; since
$\tau_\gamma\equiv0$ on the collars, the non-trivial sign change occurs only on the core arc
$\Gamma_2$ (Remark~\ref{rem:5.17}).}
\label{fig:2}
\end{figure}

\subsection{The topological seed and the low-order gates}
\label{sec:5.1}

\begin{lemma}[relative local knot insertion: the topological seed]
\label{lem:5.1}
Let $m\ge1$ and let $K_1,\dots,K_m$ be arbitrary knot types. Then there exist a base circle
$C_0\subset\Pi$ of curvature $\kappa_c>0$, pairwise disjoint closed round balls $B_1,\dots,B_m$
centred at points of $C_0$ --- hence $R(B_i)=B_i$ --- concentric balls
$B_i'\Subset\operatorname{int}B_i$, and a $C^\infty$ embedding $\gamma_0:S^1\to\mathbb R^3$, such that:
\begin{enumerate}
\item[\textbf{(T1)}] $\gamma_0$ is defined on the same circle as the arclength parametrisation $c$
   of $C_0$, and there are pairwise disjoint open parameter arcs $J_1,\dots,J_m$ with
   $\gamma_0(\overline{J_i})\subset\operatorname{int}B_i'$ and $\gamma_0=c$ \textbf{as maps} on the complement of
   $\bigcup_iJ_i$; in particular $\operatorname{Im}\gamma_0$ coincides with $C_0$ outside
   $\bigcup_i\operatorname{int}B_i'$, and $\gamma_0$ is an \emph{exact} unit-speed circular arc of
   $\Pi$ of curvature $\kappa_c$ on each component of
   $\gamma_0^{-1}\bigl(B_i\setminus\operatorname{int}B_i'\bigr)$;
\item[\textbf{(T2)}] $\gamma_0$ meets each $\partial B_i$ transversally in exactly two points, and
   $\gamma_0^{-1}(B_i)$ is a single closed parameter arc; the two intersection points therefore carry
   open two-sided $\Pi$-circular \textbf{collars} on which $\gamma_0$ agrees with $C_0$;
\item[\textbf{(T3)}] $(B_i,\ \operatorname{Im}\gamma_0\cap B_i)$ is a $1$-string tangle realising
   $K_i$;
\item[\textbf{(T4)}] $[\gamma_0]=\#_{i=1}^mK_i$, each $\partial B_i$ exhibiting the corresponding
   factorisation --- a decomposing sphere in the strict sense whenever $K_i$ is non-trivial;
\item[\textbf{(T5)}] for any $S\subseteq\{1,\dots,m\}$, replacing $\gamma_0$ by $R\circ\gamma_0$ on
   $\gamma_0^{-1}(B_i)$ for $i\in S$, and leaving it unchanged elsewhere, again yields a $C^\infty$
   embedding, whose knot type is $\#_i K_i^{\varepsilon_i}$ with $\varepsilon_i=-1$ exactly for
   $i\in S$.
\end{enumerate}
\end{lemma}

The curve is built so that it \emph{is} the base circle outside the small inner balls, rather than
merely isotopic to it there; that is what makes the collars exactly circular. The proof is in
Appendix~\ref{sec:A.1}.

\begin{convention}[ball arcs, collar width, core open arcs]
\label{conv:5.2}
Since $\gamma_0\cap B_i$ is a
$1$-string tangle, $\gamma_0^{-1}(B_i)$ is a \textbf{single closed arc} $[\alpha_i,\beta_i]$ and
$\gamma_0^{-1}(\partial B_i)=\{\alpha_i,\beta_i\}$. By (T1)--(T2) of Lemma~\ref{lem:5.1} one may fix
a common $\rho>0$, small enough that $\alpha_i+2\rho<\beta_i-2\rho$, such that $\gamma_0$ equals
\textbf{exactly} a $\Pi$-circular arc of curvature $\kappa_c$ on
\[
(\alpha_i-2\rho,\ \alpha_i+2\rho)\quad\text{and}\quad(\beta_i-2\rho,\ \beta_i+2\rho),
\]
while the \textbf{knotted part} --- by which we always mean the parameter arc $J_i$ of
Lemma~\ref{lem:5.1}(T1), on which $\gamma_0$ may leave $\Pi$ --- is contained in
$[\alpha_i+2\rho,\ \beta_i-2\rho]$. Such a $\rho$ exists because $\overline{J_i}$ is a compact subset
of the open arc $\gamma_0^{-1}(\operatorname{int}B_i')\subset(\alpha_i,\beta_i)$. Define the
\textbf{core open arc} of the $i$-th ball,
\[
\Gamma_i:=(c_i,\ d_i),\qquad c_i:=\alpha_i+\rho,\quad d_i:=\beta_i-\rho,\qquad
\Gamma:=\coprod_{i=1}^m\Gamma_i .
\]
Then: \textbf{(P1)} $c_i,d_i$ lie in the \textbf{interior} of the open collars, so $\gamma_0$ is still an exact
$\Pi$-circular arc on $[c_i,c_i+\rho)$ and on $(d_i-\rho,d_i]$; \textbf{(P2)}
$\overline{\Gamma_i}=[c_i,d_i]\subset(\alpha_i,\beta_i)$, so $\gamma_0(\overline{\Gamma_i})$ is a
\textbf{compact} subset of $\operatorname{int}B_i$; \textbf{(P3)} the knotted part is contained in $\Gamma$;
\textbf{(P4)} $\overline{\Gamma_1},\dots,\overline{\Gamma_m}$ are pairwise disjoint and $\Gamma\neq S^1$;
\textbf{(P5)} $\gamma_0(S^1\setminus\Gamma)\subset\Pi$.
\end{convention}

Figure~\ref{fig:3} shows the resulting anatomy. Property (P1) is the one that will be used at the delicate
point of \S\ref{sec:5.2}: it guarantees that the curve is \emph{exactly} a circular arc on a one-sided neighbourhood
of each core endpoint, which is what makes the exact identity of Lemma~\ref{lem:5.5} available there.

\begin{figure}[tbp]
\centering
\includegraphics{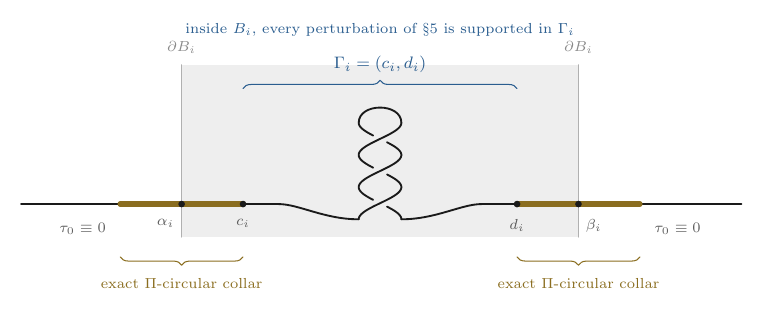}
\caption{\emph{The anatomy of one private ball, with the parameter circle unrolled.} The curve meets
$\partial B_i$ transversally at $\alpha_i$ and $\beta_i$, and is an exact $\Pi$-circular arc on the
collars straddling those two points. The core arc $\Gamma_i=(c_i,d_i)$ sits strictly inside, and
carries the support of every perturbation of \S\ref{sec:flexibility}. Outside $\coprod_i\Gamma_i$ the torsion vanishes
identically, and inside each $\Gamma_i$ its zeros are finite in number and all simple; this is what
makes $Z_\infty(\tau_0)$ exactly the complement of $\coprod_i\Gamma_i$, and hence $c(\tau_0)=m$
(Theorem~\ref{thm:5.8}).}
\label{fig:3}
\end{figure}

\begin{lemma}[relative removal of zero curvature]
\label{lem:5.3}
Let $\gamma_0$ be as in Lemma~\ref{lem:5.1} and let $A$ be
the closed union of the $2\rho$-collars of Convention~\ref{conv:5.2}, that is
$A:=\coprod_i\bigl([\alpha_i-2\rho,\alpha_i+2\rho]\cup[\beta_i-2\rho,\beta_i+2\rho]\bigr)$. Put $Z:=(j^2\gamma_0)^{-1}(\Sigma_2)$, where
$\Sigma_2=\{(x,v,a):v\neq0,\ v\times a=0\}$ is a smooth codimension-$2$ submanifold of the regular
$2$-jet space. Then $Z$ is compact and $Z\subset\bigcup_i\operatorname{int}(B_i)\setminus A$, and for
every $\epsilon>0$ there is a $C^\infty$ curve $\widetilde\gamma$ with
$\|\widetilde\gamma-\gamma_0\|_{C^2}<\epsilon$, equal to $\gamma_0$ on a neighbourhood of $A$, and
with $\kappa>0$ everywhere. Moreover the support of the perturbation may be taken to satisfy
\begin{equation}\label{eq:5.1}\tag{5.1}
\overline U\ \subset\ \coprod_i\ (\alpha_i+2\rho,\ \beta_i-2\rho)\ \subset\ \Gamma.
\end{equation}
\end{lemma}

Condition \eqref{eq:5.1} says that \textbf{the whole perturbation lives strictly inside the core arcs}; in
particular it is disjoint from $[c_i,c_i+\rho)$ and $(d_i-\rho,d_i]$, so those two pieces remain
exact $\Pi$-circular arcs. Every later perturbation will respect the same condition, and this is
what keeps property (P1) available at the very end. See Appendix~\ref{sec:A.2} for the proof.

\begin{lemma}[one common arclength label]
\label{lem:5.4}
Reparametrise $\widetilde\gamma$ by arclength \textbf{once},
obtaining a unit-speed curve of length $L$ with parameter $s\in\mathbb R/L\mathbb Z$.
Reparametrisation moves no point of the image, so the collars remain $\Pi$-arcs and $\kappa>0$ is
unchanged. Every local reflection $R$ is a Euclidean isometry and \textbf{preserves speed pointwise}, so
all reflected members constructed below are unit speed, of the same length $L$, and carry the \textbf{same
arclength label} $s$; consequently their data $(\kappa,|\tau|)$ may legitimately be compared at the
same $s$. \hfill$\square$
\end{lemma}

We write $\gamma_\flat$ for the curve produced by Lemmas~\ref{lem:5.1}--\ref{lem:5.4}. It satisfies (R1) and (R2) but
carries no information at all about (R3): its torsion vanishes identically on the entire planar part,
which is a set with $m$ complementary components, but nothing yet excludes further flat points inside
the cores, which would raise $c$ and destroy the equality. Removing that possibility is the subject
of the next subsection.

\subsection{Pinning the infinite-order zero set}
\label{sec:5.2}

For a $C^4$ regular curve $\eta$ put $D_\eta:=\det(\eta',\eta'',\eta''')$ and
$D'_\eta:=\det(\eta',\eta'',\eta'''')$; by Lemma~\ref{lem:6.3} below, $\tau_\eta=D_\eta h_\eta$ with
$h_\eta>0$, so every statement about the zeros of $D$ transfers verbatim to $\tau$. Two facts proved
in \S\ref{sec:genericity} are used in this section and in Appendix~\ref{app:A}: that lemma, and the
invariance of everything relevant under arclength reparametrisation (Lemma~\ref{lem:6.9}). Both are
elementary computations with the parametrisation-free formulas of Lemma~\ref{lem:2.8}, and
\textbf{\S\ref{sec:genericity} uses nothing from \S\ref{sec:flexibility}}, so there is no
circularity; they are placed there because that is where they are used most.

We must arrange $D\neq0$, or at worst simple zeros, at every point of every core arc $\Gamma_i$.
Two regions require different tools. On a compact subcore, jet transversality applies and yields
simple zeros for almost every small parameter. Near the two endpoints $c_i,d_i$ it does not: those
points sit at the boundary of the core, transversality on a compact set says nothing there, and the
curve must in addition remain exactly circular immediately outside. At those endpoints we therefore
use an identity instead of a perturbation argument; Figure~\ref{fig:4} shows what it delivers.

\begin{lemma}[collar witness identity]
\label{lem:5.5}
Let $\eta$ be, on a parameter interval $I$, a \textbf{unit-speed}
circular arc inside $\Pi$ of constant curvature $\kappa_c$, oriented so that $B=T\times N=+e_3$ with
$e_3=(0,0,1)$. Let $\phi\in C^\infty(I,\mathbb R)$ and $\eta_\phi:=\eta+\phi\,e_3$. Then, \textbf{exactly},
on $I$,
\[
D_{\eta_\phi}\;=\;\kappa_c^{3}\,\phi'+\kappa_c\,\phi'''\;=\;\kappa_c\bigl(\phi'''+\kappa_c^{2}\phi'\bigr).
\]
\end{lemma}

\begin{proof}
The vectors $\eta',\eta'',\eta'''$ all lie in the direction plane
$\operatorname{span}(e_1,e_2)$. Expand $\det(\eta'+\phi'e_3,\ \eta''+\phi''e_3,\ \eta'''+\phi'''e_3)$
by multilinearity: the term without $e_3$ is $\det(\eta',\eta'',\eta''')=0$, three coplanar vectors;
terms with two or three factors $e_3$ have a repeated column and vanish; only the three terms with
exactly one $e_3$ survive. Using $\eta'=T$, $\eta''=\kappa_cN$, $\eta'''=-\kappa_c^2T$ (constant
curvature, $\tau\equiv0$) and $T\times N=e_3$, these are
$\phi'\det(e_3,\kappa_cN,-\kappa_c^2T)=\kappa_c^3\phi'$, $\phi''\det(T,e_3,-\kappa_c^2T)=0$ and
$\phi'''\det(T,\kappa_cN,e_3)=\kappa_c\phi'''$.
\end{proof}

The formula is an \textbf{exact identity}, not a linearisation, and is \textbf{strictly linear} in
$\phi$, so the amplitude of the displacement enters only as a non-zero scalar factor and may be taken
as small as the low-order gates require. Taking $\phi\equiv0$ also
records the baseline: an unperturbed collar has $D\equiv0$ and lies entirely inside $Z_\infty(\tau)$,
which is precisely why the flat set has the $m$ components we want.

\begin{lemma}[endpoint witness]
\label{lem:5.6}
For each $i$ set
\[
\theta_i(s):=\exp\Bigl(\frac{-1}{s-c_i}\Bigr)\exp\Bigl(\frac{-1}{d_i-s}\Bigr)\ \ (s\in\Gamma_i),
\qquad \theta_i(s):=0\ \ (s\notin\Gamma_i).
\]
Then $\theta_i\in C^\infty(S^1)$, its positivity set is exactly $\Gamma_i$, and all derivatives
vanish at $c_i$ and $d_i$. There is $\sigma\in(0,\rho]$ such that for \textbf{every} $\lambda\neq0$ the
curve $\gamma_\flat+\lambda\theta_ie_3$ satisfies
\[
D\neq0\quad\text{on}\quad(c_i,\ c_i+\sigma)\ \cup\ (d_i-\sigma,\ d_i).
\]
\end{lemma}

\begin{proof}
Smoothness and flatness at the endpoints are standard bump facts. By Convention~\ref{conv:5.2}(P1) and
the support condition \eqref{eq:5.1}, $\gamma_\flat$ is on $[c_i,c_i+\rho)$ exactly a unit-speed
$\Pi$-circular arc of constant curvature $\kappa_c$, so Lemma~\ref{lem:5.5} applies and gives
$D=\lambda\kappa_c(\theta_i'''+\kappa_c^2\theta_i')$. Put $x:=s-c_i$ and $\Lambda:=d_i-c_i$; then
$\theta_i=e^{-1/x}g(x)$ with $g(x)=\exp(-1/(\Lambda-x))$ smooth and $g(0)>0$. Differentiating,
$\theta_i'=e^{-1/x}(x^{-2}g+g')$ and $\theta_i'''=e^{-1/x}\bigl(x^{-6}g+O(x^{-5})\bigr)$, whence
\[
x^{6}e^{1/x}\bigl(\theta_i'''+\kappa_c^{2}\theta_i'\bigr)\ \longrightarrow\ g(0)>0\qquad(x\to0^+).
\]
So there is $\sigma_1>0$ such that this quantity is positive on $(0,\sigma_1)$, i.e. $D$ has the sign
of $\lambda$ and is non-zero on $(c_i,c_i+\sigma_1)$. The right endpoint is symmetric: put
$x:=d_i-s$, each odd-order derivative acquires a factor $-1$, and $D$ has the sign of $-\lambda$.
Take $\sigma:=\min(\sigma_1,\sigma_2,\rho)$. Since the expression of Lemma~\ref{lem:5.5} is linear in $\phi$,
the amplitude $\lambda$ enters only as a non-zero factor, so \textbf{every} $\lambda\neq0$ works.
\end{proof}

The lemma asserts only the existence of $\sigma$, which depends on $\Lambda$ and $\kappa_c$; only
$\sigma\in(0,\rho]$ is used below. Fix now $\lambda\neq0$, small enough for the gates of Lemma~\ref{lem:A.2},
and put $\gamma_1:=\gamma_\flat+\lambda\sum_i\theta_ie_3$, which coincides pointwise with
$\gamma_\flat$ on $S^1\setminus\Gamma$.

\begin{figure}[tbp]
\centering
\includegraphics{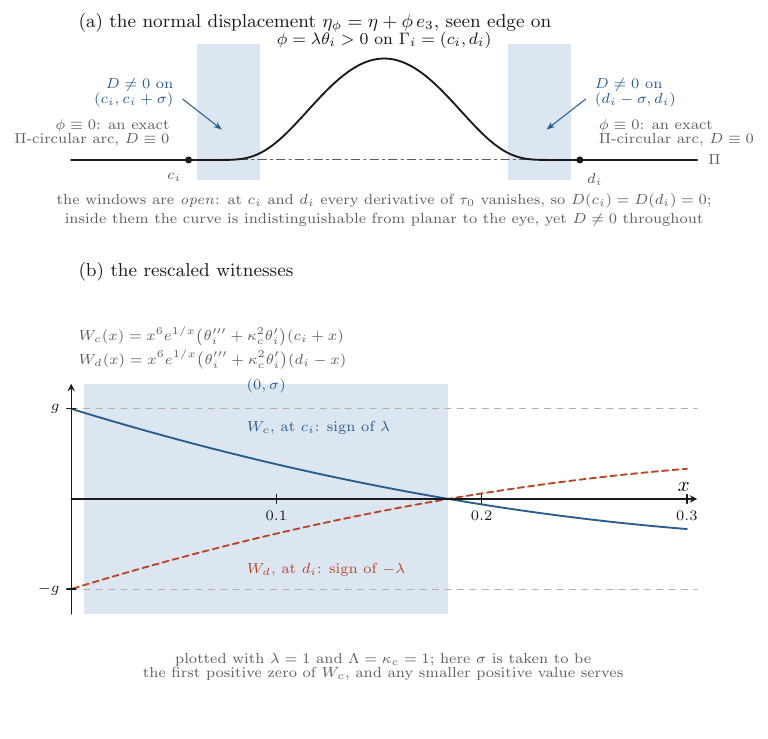}
\caption{\emph{Why the endpoints of a core arc remain infinitely flat while $D$ is already non-zero
inside them.} (a) The normal displacement $\eta_\phi=\eta+\phi\,e_3$ seen edge on, with the vertical
scale exaggerated: outside $\Gamma_i$ the curve is an exact $\Pi$-circular arc, and inside the two
shaded windows it is indistinguishable from planar to the eye. The windows are \textbf{open}: at $c_i$ and
$d_i$ every derivative of $\tau_0$ vanishes, so $D(c_i)=D(d_i)=0$ and both endpoints lie in
$Z_\infty(\tau_0)$. (b) The rescaled witnesses $W_c$ and $W_d$, each in its own local coordinate,
tend to $g(0)=e^{-1/\Lambda}>0$ and to $-g(0)$ respectively. By Lemma~\ref{lem:5.5}, $D\neq0$ on
$(c_i,c_i+\sigma)\cup(d_i-\sigma,d_i)$ for \textbf{every} $\lambda\neq0$. Both panels are exact for the explicit bump with
$\lambda=\Lambda=\kappa_c=1$; there $\sigma$ is taken to be the first positive zero of $W_c$, and
Lemma~\ref{lem:5.6} is content with any smaller positive value.}
\label{fig:4}
\end{figure}

\begin{proposition}[relative $4$-jet transversality on the core]
\label{prop:5.7}
In the open set
$U^{(4)}:=\{v\times a\neq0\}$ of $J^4(S^1,\mathbb R^3)$ put
\[
\Xi:=\{\det(v,a,j)=0\}\cap\{\det(v,a,q)=0\}\cap U^{(4)} .
\]
\textbf{(i)} $\Xi$ is a smooth submanifold of $U^{(4)}$ of \textbf{codimension $2$}.
\textbf{(ii)} Put $K:=\coprod_i[c_i+\tfrac\sigma2,\ d_i-\tfrac\sigma2]$, compact, and
$V:=\coprod_i(c_i+\tfrac\sigma4,\ d_i-\tfrac\sigma4)$, open with $\overline V\subset\Gamma$. Then
there is a finite family $\{\psi_k\}_{k=1}^N\subset C_c^\infty(V,\mathbb R^3)$ such that for \textbf{almost
every} sufficiently small $w\in\mathbb R^N$ the curve $\gamma_w:=\gamma_1+\sum_kw_k\psi_k$ satisfies
$j^4\gamma_w\pitchfork\Xi$ on $K$; equivalently $(D_{\gamma_w},D'_{\gamma_w})\neq(0,0)$ on $K$, that
is, all zeros of $D_{\gamma_w}$ on $K$ are \textbf{simple}.
\end{proposition}

Part (i) is a two-line computation: $D=\langle v\times a,j\rangle$ is linear in $j$ and independent
of $q$, while $D'=\langle v\times a,q\rangle$ is linear in $q$ and independent of $j$, so the
$2\times6$ Jacobian in the $(j,q)$ directions is block diagonal with rows $(v\times a)^{\mathsf T}$
in complementary summands of $\mathbb R^6$. The rows are independent iff both are non-zero, i.e. iff
$v\times a\neq0$; degeneracy of the gradients can occur \textbf{only at $\kappa=0$}, which is excluded
from $U^{(4)}$. This is the structural role of the hypothesis $\kappa>0$, and it recurs verbatim in
Lemma~\ref{lem:6.4}.

Part (ii) is the delicate one, not because of the transversality itself but because of the order in
which the objects must be produced: local candidate perturbations are built first, a finite subcover
extracted, the complete family and hence the dimension $N$ fixed, and only afterwards is a single
radius in the $w$-variable chosen, on the complete variable $(s,w)$. No radius is available before
$N$ exists. The full argument is Appendix~\ref{app:B}.

\begin{theorem}[the reference curve; $c(\tau_0)=m$ exactly]
\label{thm:5.8}
Insert the perturbations of Lemmas~\ref{lem:5.6}
and Proposition~\ref{prop:5.7} into the chain of \S\ref{sec:5.1} after Lemma~\ref{lem:5.4}, and reparametrise by arclength as in
Lemma~\ref{lem:A.2}(6). The resulting curve $\gamma$ satisfies (R1) and (R2) --- it is a $C^\infty$ unit-speed
closed embedding of length $L$ with $\kappa>0$ everywhere, $m$ symmetric private balls centred on
$\Pi$, open planar collars, a tangle realising $K_i$ inside $B_i$, and $[\gamma]=\#_iK_i$ --- and its
torsion satisfies (R3):
\[
Z_\infty(\tau_0)=S^1\setminus\coprod_{i=1}^m\Gamma_i,\qquad
S^1\setminus Z_\infty(\tau_0)=\coprod_{i=1}^m\Gamma_i,\qquad c(\tau_0)=m,
\]
where $\Gamma_i$ has been transported by the reparametrisation. Moreover $\tau_0$ has \textbf{finitely many
zeros in each $\Gamma_i$, all of them simple}.
\end{theorem}

\begin{proof}
Write $\gamma_2:=\gamma_w$ as in Proposition~\ref{prop:5.7} and let $\gamma$ be its arclength
reparametrisation. By Lemma~\ref{lem:A.2} the curve $\gamma$ has all the listed structural properties, and by
Lemma~\ref{lem:A.2}(6) every conclusion below about $\gamma_2$ transports to $\gamma$; so we argue with
$\gamma_2$ and its torsion.

\textbf{(a) All zeros in the core are simple, and finitely many.} On $(c_i,c_i+\frac\sigma4]$ we have
$\gamma_2=\gamma_1$, the perturbation being supported in $V$, so $D\neq0$ by Lemma~\ref{lem:5.6}. On the
compact set $[c_i+\frac\sigma4,\ c_i+\frac\sigma2]\subset(c_i,c_i+\sigma)$ the quantity
$|D_{\gamma_1}|$ has a positive lower bound, while $\eta\mapsto D_\eta$ is continuous from $C^3$ to
$C^0$ and $\|\sum w_k\psi_k\|_{C^3}$ may be taken arbitrarily small, so $D_{\gamma_2}\neq0$ there as
well; the right endpoint is symmetric. Hence $D_{\gamma_2}\neq0$ on
$(c_i,c_i+\frac\sigma2)\cup(d_i-\frac\sigma2,d_i)$. On $K_i=[c_i+\frac\sigma2,d_i-\frac\sigma2]$ all
zeros are simple by Proposition~\ref{prop:5.7}; simple zeros are isolated and the zero set is closed in the
compact $K_i$, hence finite. The three pieces cover exactly $\Gamma_i$. By Lemma~\ref{lem:6.3} the statements
about $D$ transfer verbatim to $\tau$.

\textbf{(b) $S^1\setminus\Gamma\subseteq Z_\infty(\tau_0)$.} By Convention~\ref{conv:5.2}(P5) and Lemma~\ref{lem:A.2}(5),
$\gamma_2(S^1\setminus\Gamma)\subset\Pi$. Let $p\in S^1\setminus\Gamma$. If $p$ is an \textbf{interior}
point of that closed set, then a whole open neighbourhood of $p$ has image in the plane $\Pi$, so
$\gamma_2',\gamma_2'',\gamma_2'''$ are coplanar there, $D\equiv0$ and $\tau\equiv0$ on that
neighbourhood; hence all derivatives vanish at $p$ and $p\in Z_\infty$. If $p\in\{c_i,d_i\}$, then
$\tau$ vanishes identically on a \textbf{one-sided} neighbourhood of $p$ and $\tau\in C^\infty$, so
$\tau^{(k)}(p)=\lim_{t\to p^\mp}\tau^{(k)}(t)=0$ for every $k$, i.e. $p\in Z_\infty$. No information
about the other side of $p$ is used: one-sided planarity alone forces infinite-order flatness.

\textbf{(c) $\Gamma\cap Z_\infty(\tau_0)=\varnothing$.} Let $s\in\Gamma_i$. If $\tau_0(s)\neq0$ then
$s\notin Z$. If $\tau_0(s)=0$ then by (a) the zero is simple, so
$\operatorname{ord}_s\tau_0=1<\infty$ and $s\notin Z_\infty$.

\textbf{(d) Counting the components.} By (b) and (c), $S^1\setminus Z_\infty(\tau_0)=\coprod_i\Gamma_i$.
The $\Gamma_i$ are pairwise disjoint open arcs with pairwise disjoint closures by Convention~\ref{conv:5.2}(P4),
hence pairwise distinct connected components; and $\coprod\Gamma_i\neq S^1$, so no further component
joins them. The number of components is therefore exactly $m$.
\end{proof}

Three boundary cases are handled by different clauses: ``no zero in the core'' by the first
alternative in (c), ``finitely many simple zeros'' by the second, and ``core endpoints lying inside the
open collar'' by the second case of (b), where one-sided planarity suffices.

\subsection{The mirrored family}
\label{sec:5.3}

With the reference curve in hand, the family is produced by reflecting tangles independently. Two
short lemmas make this legitimate.

\begin{lemma}[the reflected curve is $C^\infty$]
\label{lem:5.9}
Let the tangle piece of the reference $\gamma$
lying in a symmetric private ball $B$ (so $R(B)=B$) occupy the arclength interval $[a,b]$, and define
the reflected curve to be $R\circ\gamma$ on $[a,b]$ and unchanged elsewhere. That curve is
$C^\infty$.
\end{lemma}

\begin{proof}
The collar lies in $\Pi$ and $R|_\Pi=\mathrm{id}$, so on an open neighbourhood of $a$ and of
$b$ the reflected piece coincides \textbf{pointwise} with the original curve; both are the same $C^\infty$
function there. This pointwise coincidence --- rather than any matching of jets --- is the sole
justification of the gluing, and it is why the collars had to be \emph{exactly} circular arcs of $\Pi$
rather than merely close to such arcs.
\end{proof}

\begin{lemma}[properties of a reflected member]
\label{lem:5.10}
The reflected curve has $\kappa>0$ (Lemma~\ref{lem:2.7});
it is unit speed, of the same length $L$ and in the same arclength label (Lemma~\ref{lem:5.4}); it shares
\textbf{pointwise} the same $\kappa(s)$ and the same $|\tau(s)|$ with the reference, the signed torsion
changing sign on the reflected interval and being unchanged elsewhere; it is an embedding, since
$R(B)=B$ is a private ball, $R(\text{tangle})\subset B$, and the exterior is untouched; and its knot
type is obtained from that of $\gamma$ by replacing the factor inside $B$ with its mirror, the
sphere $\partial B$ --- decomposing whenever that factor is non-trivial --- being preserved. \hfill$\square$
\end{lemma}

\begin{theorem}[the multi-factor family]
\label{thm:5.11}
For every $m\ge1$ and arbitrary knot types
$K_1,\dots,K_m$ there is a family \textbf{indexed by $\{\pm1\}^m$} of closed embeddings
$\{\gamma_\varepsilon\}$ of common length $L$, unit speed, of class $C^\infty$, with $\kappa>0$
everywhere, labelled by one arclength parameter, such that
\begin{itemize}
\item \textbf{(i)} each $\gamma_\varepsilon$ satisfies pointwise
  $\kappa_{\gamma_\varepsilon}(s)=\kappa_0(s)$ and
  $|\tau_{\gamma_\varepsilon}(s)|=|\tau_0(s)|$, the reference datum, independent of $\varepsilon$;
\item \textbf{(ii)} $[\gamma_\varepsilon]=\#_{i=1}^mK_i^{\varepsilon_i}$.
\end{itemize}
\end{theorem}

\begin{proof}
Lemmas~\ref{lem:5.1}--\ref{lem:5.4}, together with Lemma~\ref{lem:5.6}, Proposition~\ref{prop:5.7} and Theorem~\ref{thm:5.8}, produce the reference member $\gamma$.
For each $\varepsilon$ put $S(\varepsilon)=\{i:\varepsilon_i=-1\}$ and reflect the tangles inside
$B_i$ for $i\in S(\varepsilon)$. The balls are pairwise disjoint and the collars lie in $\Pi$, which
$R$ fixes, so the reflections are independent and conflict-free; Lemmas~\ref{lem:5.9}--\ref{lem:5.10} applied ball by ball
give all stated properties.
\end{proof}

Injectivity of $\varepsilon\mapsto\gamma_\varepsilon$ is not part of this statement; it follows from
Theorem~\ref{thm:5.16}(E2), and injectivity of $\varepsilon\mapsto[\gamma_\varepsilon]$ from
Proposition~\ref{prop:5.13}(B1) under the chirality hypothesis below. The immediate consequence needs neither.

\begin{corollary}
\label{cor:5.12}
The unsigned Frenet datum $(\kappa,|\tau|)$ does not determine the knot type.
\hfill$\square$
\end{corollary}

\begin{assumptionH}
Each $K_i$ is prime and chiral ($K_i\neq\overline{K_i}$), and the $2m$ knots
$\{K_i,\overline{K_i}\}_{i=1}^m$ are pairwise distinct.
\end{assumptionH}

\begin{proposition}[the family realises $2^m$ types under (H)]
\label{prop:5.13}
Assume (H). Then
\begin{itemize}
\item \textbf{(B1)} the family of Theorem~\ref{thm:5.11} realises exactly $2^m$ distinct knot types;
\item \textbf{(B2)} consequently $\bigl|\mathcal F(\kappa_0,|\tau_0|)\bigr|\ge2^m$;
\item \textbf{(B3)} these $2^m$ types split into exactly $2^{m-1}$ mirror orbits.
\end{itemize}
\end{proposition}

\begin{proof}
(B1) By uniqueness of Schubert prime decomposition \cite{ref4} the multiset of prime factors of
$\#_iK_i^{\varepsilon_i}$ is $\{K_i^{\varepsilon_i}\}$; under (H) the $2m$ types are pairwise
distinct, so each $\varepsilon_i$ can be read off slot by slot, the map
$\varepsilon\mapsto\{K_i^{\varepsilon_i}\}$ is injective, and $2^m$ distinct types result. (B2) All
these types lie in $\mathcal F$. (B3) $\mathrm{mirror}([\gamma_\varepsilon])=[\gamma_{-\varepsilon}]$
and $-\varepsilon\neq\varepsilon$ always, so each mirror orbit $\{\varepsilon,-\varepsilon\}$ has
exactly two elements, distinct by (H), giving $2^{m-1}$ orbits.
\end{proof}

\begin{example}[an exact ambiguous pair]
\label{ex:5.14}
Under (H) with $m=2$ take $\varepsilon=(+,+)$ and
$\varepsilon'=(+,-)$. Then $[\gamma_\varepsilon]=K_1\#K_2$ and
$[\gamma_{\varepsilon'}]=K_1\#\overline{K_2}$ are \textbf{distinct} by (B1) and \textbf{not mirror images of one
another}, since $\mathrm{mirror}[\gamma_\varepsilon]=[\gamma_{(-,-)}]=\overline{K_1}\#\overline{K_2}
\neq K_1\#\overline{K_2}$ because $K_1\neq\overline{K_1}$. Such a pair --- identical data, distinct
non-mirror knot types --- is called an \textbf{exact ambiguous pair}. It is the seed of every construction
in \S\ref{sec:dichotomy}.
\end{example}

\begin{remark}[degenerations]
\label{rem:5.15}
Amphichiral factors, mutually mirror factors, unknots, or tangles
symmetric with respect to $\Pi$ make the count of knot types collapse, while the construction itself
remains valid. For instance $m=2$ with $K_1=K_2=T$ the trefoil gives the granny $T\#T$, the square
$T\#\overline T$ and the mirror granny $\overline T\#\overline T$, i.e. three types rather than four.
The counts (E1)--(E3) of Theorem~\ref{thm:5.16} are unaffected by this, since they do not refer to knot types.
\end{remark}

\subsection{$|\mathrm{Fib}_{SE}(d_m)|=2^m$ exactly}
\label{sec:5.4}

Proposition~\ref{prop:5.13} supplies a lower bound; Corollary~\ref{cor:4.5} with $c(\tau_0)=m$ from Theorem~\ref{thm:5.8} supplies
the matching upper bound. The following theorem closes them and, more than that, identifies the
fibre completely: the $2^m$ constructed curves are not merely $2^m$ of its members, they are all of
them.

\begin{theorem}[exact fibre]
\label{thm:5.16}
Let $\gamma$ be the reference curve of Theorem~\ref{thm:5.8} with Frenet datum
$(\kappa_0,\tau_0)$, and let $\{\gamma_\varepsilon\}_{\varepsilon\in\{\pm1\}^m}$ be the family of
Theorem~\ref{thm:5.11}, and write $d_m:=(\kappa_0,|\tau_0|)$ for the common unsigned datum. Under the
common-arclength-label convention:
\begin{itemize}
\item \textbf{(E1)} $\bigl|\mathcal L(\tau_0)\bigr|=2^{c(\tau_0)}=2^m$;
\item \textbf{(E2)} $\tau_{\gamma_\varepsilon}=\Phi(\varepsilon)$, where $\mathcal E(\tau_0)$ is identified
  with $\{\pm1\}^m$ via $\varepsilon\mapsto(\varepsilon|_{\Gamma_i})_{i=1}^m$; hence
  $\varepsilon\mapsto\tau_{\gamma_\varepsilon}$ is a \textbf{bijection} $\{\pm1\}^m\to\mathcal L(\tau_0)$;
\item \textbf{(E3)} $\bigl|\mathrm{Fib}_{SE}(d_m)\bigr|=2^m$ and \textbf{exactly one
  member of the family lies in each class}. Equivalently:
  \textbf{every} lift in $\mathcal L(\tau_0)$ is realised by a closed embedding, uniquely up to
  $SE(3)$;
\item \textbf{(E3$'$)} $\bigl|\mathrm{Fib}_{E}(d_m)\bigr|=2^{m-1}$: the $2^m$ classes of (E3) are
  interchanged in pairs $\{\varepsilon,-\varepsilon\}$ by a reflection;
\item \textbf{(E4)} under (H), $\bigl|\mathcal F(d_m)\bigr|=2^m$ and
  $\mathcal F(d_m)=\bigl\{[\#_{i=1}^mK_i^{\varepsilon_i}]:\varepsilon\in\{\pm1\}^m\bigr\}$,
  forming exactly $2^{m-1}$ mirror orbits.
\end{itemize}
\end{theorem}

Statements (E1)--(E3) do not use (H); only (E4) does.

\begin{proof}
\textbf{(E1)} Corollary~\ref{cor:3.8} together with Theorem~\ref{thm:5.8}.

\textbf{(E2)} By Lemma~\ref{lem:5.10}, $\gamma_\varepsilon$ has the same pointwise $\kappa$ as $\gamma$, and
$\tau_{\gamma_\varepsilon}=-\tau_0$ on the reflected intervals
$\bigcup_{i\in S(\varepsilon)}[\alpha_i,\beta_i]$ and $=\tau_0$ elsewhere. By Convention~\ref{conv:5.2}(P2),
$\Gamma_i\subset(\alpha_i,\beta_i)$, so on $\Gamma_i$ we have
$\tau_{\gamma_\varepsilon}=\varepsilon_i\tau_0$; on $S^1\setminus\coprod\Gamma_i$ Theorem~\ref{thm:5.8} gives
$\tau_0=0$, so both sides vanish. Pointwise equality is exactly
$\tau_{\gamma_\varepsilon}=\Phi(\varepsilon)$, and bijectivity follows from Theorem~\ref{thm:3.7} and the
identification above.

\textbf{(E3)} Let $\widetilde\gamma$ be a curve of the class of Definition~\ref{def:2.1} sharing the arclength label
with $\gamma$ and satisfying $\widetilde\kappa=\kappa_0$, $|\widetilde\tau|=|\tau_0|$ pointwise. Then
$\widetilde\tau\in\mathcal L(\tau_0)$, and by Lemma~\ref{lem:2.10} the function $\widetilde\tau$ determines
$\widetilde\gamma$ up to $SE(3)$. Hence
\[
\Theta:\{SE(3)\text{-classes in the fibre}\}\longrightarrow\mathcal L(\tau_0),\qquad
[\widetilde\gamma]_{SE(3)}\longmapsto\widetilde\tau
\]
is well defined, since $SE(3)$ preserves $\tau$ pointwise, and injective. By (E2) the torsions of the
$\gamma_\varepsilon$ already exhaust all $2^m$ elements of $\mathcal L(\tau_0)$, so $\Theta$ is
surjective, hence bijective, and the number of classes is $2^m$. Distinct $\varepsilon$ give distinct
$\tau_{\gamma_\varepsilon}$, so the $\gamma_\varepsilon$ lie in pairwise distinct classes; $2^m$
members distributed among $2^m$ classes, pairwise distinct, means exactly one member per class.

\textbf{(E3$'$)} Let $R_0$ be any reflection. By Lemma~\ref{lem:2.7},
$\tau_{R_0\gamma_\varepsilon}=-\tau_{\gamma_\varepsilon}=-\Phi(\varepsilon)=\Phi(-\varepsilon)
=\tau_{\gamma_{-\varepsilon}}$, while the curvatures agree and $R_0$ preserves the arclength label;
so Lemma~\ref{lem:2.10} places $R_0\gamma_\varepsilon$ in the $SE(3)$-class of $\gamma_{-\varepsilon}$.
Since $-\varepsilon\neq\varepsilon$ always, the $2^m$ classes of (E3) are permuted by the reflection
in $2^{m-1}$ orbits of size two, and each $E(3)$-class is the union of one such pair.

\textbf{(E4)} Lower bound: Proposition~\ref{prop:5.13}(B1),(B2). Upper bound: the knot type is an $SE(3)$-invariant,
so $|\mathcal F|\le\#\{SE(3)\text{-classes}\}=2^m$. Equality of the two bounds forces the fibre to
contain no knot type outside the listed set, and Proposition~\ref{prop:5.13}(B3) counts the mirror orbits.
\end{proof}

\begin{remark}[{what the exact collars of Lemma~\ref{lem:5.1} bought}]
\label{rem:5.17}
In the proof of (E2) the reflected interval
$[\alpha_i,\beta_i]$ is strictly \textbf{larger} than the core $\Gamma_i$, so the two sets on which the
sign is flipped and on which $\tau_0$ is allowed to be non-zero do not coincide. The identity at the
level of $\tau$ survives only because their difference lies in the collar, where $\tau_0\equiv0$.
That in turn is exactly what (T1) of Lemma~\ref{lem:5.1} provides: $\gamma_0$ \emph{equals} the base
circle on $B_i\setminus\operatorname{int}B_i'$, so the collar is planar and circular on the nose. Had
the knotted arc merely been isotopic to a circular arc near $\partial B_i$, the collar would have
been only approximately planar and the sign flip would have polluted a region where
$\tau_0\neq0$.
\end{remark}

\begin{remark}[the counts do not imply one another]
\label{rem:5.18}
\leavevmode
\begin{enumerate}
\item ``Exactly one member per class'' in (E3) means that the \textbf{constructed family} provides exactly
   one representative of each $SE(3)$ class; an $SE(3)$ class is a whole orbit and of course contains
   infinitely many curves. Likewise $|\mathcal F(d_m)|=2^m$ in (E4) is an equality of \textbf{knot
   types}, not of curves, and the injective direction of $\Theta$ runs from $SE(3)$ classes to lifts,
   so nothing may be inferred backwards from equality of knot types.
\item All the counts equal $2^m$ or $2^{m-1}$ here, but for different reasons. $|\mathcal L(\tau_0)|$,
   $|\mathrm{Fib}_{SE}|$ and $|\mathrm{Fib}_{E}|$ are independent of (H); $|\mathcal F|$ is not.
   Without (H) one still has $|\mathrm{Fib}_{SE}(d_m)|=2^m$, while $|\mathcal F(d_m)|$ may be
   strictly smaller --- for instance if some $K_i$ is amphichiral, so that $K_i^{+1}=K_i^{-1}$
   (Remark~\ref{rem:5.15}).
\item Theorem~\ref{thm:5.16} is a statement about the construction of this section. Corollary~\ref{cor:4.5} bounds the fibre
   for \textbf{every} reference curve; computing $c(\tau)$ for a general $\tau$ is a different question,
   raised in \S\ref{sec:questions}.
\end{enumerate}
\end{remark}

\begin{remark}[the closing and embeddedness constraints discard no lift here]
\label{rem:5.19}
In general, given
periodic signed data $(\kappa,g)$, the curve obtained by integrating the Frenet system need not close
up, and a closed one need not be embedded. Grinevich--Schmidt \cite{ref13} give a necessary and sufficient
condition, in terms of periodic $(\kappa,\tau)$, for the reconstructed curve to close, and that
condition can rule out individual sign choices, pushing the bound of Corollary~\ref{cor:4.5} strictly below
$2^{c(\tau)}$ in some examples. Theorem~\ref{thm:5.16}(E3) shows that \textbf{no such reduction occurs for the
construction of this section}: every lift is realised by an explicit locally mirrored member, so the
second inequality of Corollary~\ref{cor:4.5} is an equality here.
\end{remark}

\begin{proposition}[one knot type, two representatives]
\label{prop:5.20}
Let $K=K_1\#K_2$ with $K_1,K_2$ prime and
chiral and $\{K_1,\overline{K_1}\},\{K_2,\overline{K_2}\}$ pairwise distinct, i.e. (H) with $m=2$.
Then $K$ has \textbf{both}
\begin{itemize}
\item \textbf{(i)} a rigid representative, by Corollary~\ref{cor:4.6}; and
\item \textbf{(ii)} a flexible representative, namely the member $\gamma_{(+,+)}$ of Theorem~\ref{thm:5.11}, whose knot
  type is exactly $K$, while $\gamma_{(+,-)}$ shares its pointwise datum and has knot type
  $K_1\#\overline{K_2}$, distinct from $K$ and not its mirror by Example~\ref{ex:5.14}. \hfill$\square$
\end{itemize}
\end{proposition}

So whether the unsigned datum determines the knot type is a property of the chosen representative and
not of the knot type --- at least for composite types of this form. The analogous question for prime
knot types is raised in \S\ref{sec:questions}.

%% file: sections/06-genericity.tex
\section{Genericity: simple torsion zeros, and rigidity as the rule}
\label{sec:genericity}

Sections~\ref{sec:rigidity} and \ref{sec:flexibility} exhibited both extremes of the branch dichotomy.
This section shows that the rigid extreme is the typical one: having only simple torsion zeros is an
open and dense condition on the space of parametrised embeddings. For a $C^r$ curve, simple torsion
zeros already imply the finite-regularity rigidity proved in Corollary~\ref{cor:6.10} below; when the
curve is in addition $C^\infty$, they imply $Z_\infty(\tau)=\varnothing$ and hence $c(\tau)=1$, so
that Theorem~\ref{thm:4.1} applies as well. The distinction matters, because $Z_\infty$ and the
branch invariant are defined only in the smooth category, whereas $\mathcal E^r$ is not.
The argument is $3$-jet transversality, with two points requiring care.
Density is obtained from the $C^\infty$ transversality theorem together with mollification, so that no
finite-regularity version is needed; and openness has to be proved separately rather than quoted,
because the relevant jet stratum is \textbf{not closed}.

\begin{definition}
\label{def:6.1}
Fix an integer $r\ge4$ and set
\[
\mathcal E^r:=\{\gamma\in C^r(S^1,\mathbb R^3):\ \gamma'\neq0,\ \gamma\ \text{injective},\
\gamma'\times\gamma''\neq0\},
\]
with $\|\gamma\|_{C^r}=\max_{0\le i\le r}\sup|\gamma^{(i)}|$. Since $S^1$ is compact,
$C^r(S^1,\mathbb R^3)$ is a Banach space, in particular a Baire space, and the strong and weak
Whitney topologies coincide. Curves in $\mathcal E^r$ are \textbf{in general not unit speed}. Put
\[
D_\gamma(t):=\det\bigl(\gamma'(t),\gamma''(t),\gamma'''(t)\bigr),\qquad
\mathcal G^r:=\{\gamma\in\mathcal E^r:\ \text{all zeros of }D_\gamma\text{ are simple}\}.
\]
\end{definition}

The threshold $r\ge4$ is imposed for three independent reasons, none of them the regularity needed in
the jet transversality theorem: $D_\gamma\in C^{r-3}$ must be at least $C^1$ for ``simple zero'' to
mean anything; the openness proof needs continuity of $\gamma\mapsto D_\gamma$ from $C^r$ to $C^1$;
and the competitor-side bookkeeping of Remark~\ref{rem:4.2} with $k=1$ needs $C^4$.

\begin{lemma}
\label{lem:6.2}
$\mathcal E^r$ is open in $C^r(S^1,\mathbb R^3)$.
\end{lemma}

\begin{proof}
$t\mapsto|\gamma'\times\gamma''|$ is continuous with a positive lower bound $c_0$ on the
compact $S^1$, and $\gamma\mapsto\gamma'\times\gamma''$ is continuous from $C^2$ to $C^0$, so small
$C^r$ perturbations keep the bound $c_0/2>0$; similarly for $\gamma'\neq0$; and embeddings form a
$C^1$-open set on a compact domain.
\end{proof}

\begin{lemma}[torsion and $D_\gamma$ differ by a positive factor]
\label{lem:6.3}
On $\mathcal E^r$ one has $\tau_\gamma=D_\gamma h_\gamma$ with
$h_\gamma=|\gamma'\times\gamma''|^{-2}\in C^{r-2}$ and $h_\gamma>0$. Hence
$\tau_\gamma^{-1}(0)=D_\gamma^{-1}(0)$, and at a zero $\tau_\gamma'=D_\gamma'h_\gamma$: \textbf{a zero
of $\tau_\gamma$ is simple if and only if the corresponding zero of $D_\gamma$ is simple}.
\hfill$\square$
\end{lemma}

This is the lemma that allows the whole section --- and \S\ref{sec:5.2} --- to work with the polynomial
quantity $D_\gamma$, which is a function of the jet, rather than with $\tau$, which is not.

\begin{lemma}[the $3$-jet zero set is a smooth hypersurface]
\label{lem:6.4}
In
$J^3(S^1,\mathbb R^3)\cong S^1\times\mathbb R^3_x\times\mathbb R^3_v\times\mathbb R^3_a\times\mathbb R^3_j$
put $U:=\{v\times a\neq0\}$, $\Delta(t,x,v,a,j):=\det(v,a,j)$ and $\Sigma:=\Delta^{-1}(0)\cap U$.
Then $\Sigma$ is a smooth codimension-$1$ submanifold of $U$.
\end{lemma}

\begin{proof}
$\Delta=\langle v\times a,\ j\rangle$ is linear in $j$, so $\nabla_j\Delta=v\times a\neq0$
on $U$ and $0$ is a regular value of $\Delta|_U$.
\end{proof}

Degeneracy of the gradient can occur only where $v\times a=0$, i.e. $\kappa=0$, and that locus is
excluded from $U$: this is the structural role of the standing hypothesis $\kappa>0$, exactly as in
Proposition~\ref{prop:5.7}(i). Note also that $\Sigma$ is \textbf{not closed} in $J^3$, since its
closure may meet $\{v\times a=0\}$. We therefore do not invoke the version of the transversality
theorem that concludes openness from closedness of the stratum; openness is proved separately in
Proposition~\ref{prop:6.6}.

\begin{lemma}[transversality $\iff$ simple zeros]
\label{lem:6.5}
For $\gamma\in\mathcal E^r$: $j^3\gamma\pitchfork\Sigma$ if and only if $0$ is a regular value of
$D_\gamma$.
\end{lemma}

\begin{proof}
For $\gamma\in\mathcal E^r$ we have $j^3\gamma(S^1)\subset U$ and
$\Delta\circ j^3\gamma=D_\gamma$. If $0$ is a regular value of $\Delta|_U$, then $g\pitchfork\Sigma$
at $t$ iff $dg_t(T_tS^1)+T\Sigma=TU$ iff $d\Delta\bigl(dg_t(T_tS^1)\bigr)\neq0$ iff
$(\Delta\circ g)'(t)\neq0$, using $T\Sigma=\ker d\Delta$.
\end{proof}

\begin{proposition}[openness, including the zero-free branch]
\label{prop:6.6}
Let $r\ge4$. Then $\mathcal G^r$ is open in $\mathcal E^r$.
\end{proposition}

\begin{proof}
Let $\gamma_0\in\mathcal G^r$, so $D_{\gamma_0}\in C^{r-3}\subseteq C^1$.

\textbf{Case A ($D_{\gamma_0}$ has no zero).} By compactness $m:=\min_{S^1}|D_{\gamma_0}|>0$; by
continuity of $\gamma\mapsto D_\gamma$ from $C^r$ to $C^0$, a small perturbation gives
$\|D_\gamma-D_{\gamma_0}\|_{C^0}<m/2$, hence $|D_\gamma|>m/2>0$: no zeros, and ``all zeros are simple''
holds vacuously. Case B needs at least one zero to anchor its arcs, so this branch is listed
separately.

\textbf{Case B (the number of zeros is $n\ge1$).} The zeros are simple, hence isolated, and the zero set
is closed in the compact $S^1$, so a finite subcover gives $Z(D_{\gamma_0})=\{t_1,\dots,t_n\}$.
Choose pairwise disjoint closed arcs $I_i\ni t_i$ and $c>0$ with $|D_{\gamma_0}'|\ge c$ on $I_i$; put
$K=S^1\setminus\bigcup_i\operatorname{int}I_i$, which is compact and non-empty when the $I_i$ are
small enough and carries no zero, so $\delta:=\min_K|D_{\gamma_0}|>0$. By continuity of
$\gamma\mapsto D_\gamma$ from $C^r$ to $C^1$, which is where $r\ge4$ enters, a small perturbation
gives $\|D_\gamma-D_{\gamma_0}\|_{C^1}<\min\{c/2,\delta/2\}$. Then $|D_\gamma|\ge\delta/2>0$ on $K$,
while $|D_\gamma'|\ge c/2>0$ on each $I_i$, so $D_\gamma$ is strictly monotone there and has at most
one zero, necessarily simple.
\end{proof}

\begin{proposition}[density, using only $C^\infty$ transversality and smoothing]
\label{prop:6.7}
Let $r\ge4$, $\gamma\in\mathcal E^r$ and $\varepsilon>0$. Then there is
$\widetilde\gamma\in C^\infty(S^1,\mathbb R^3)\cap\mathcal G^r$ with
$\|\widetilde\gamma-\gamma\|_{C^r}<\varepsilon$.
\end{proposition}

\begin{proof}
(a) $\mathcal E^r$ is open by Lemma~\ref{lem:6.2}, so shrink $\varepsilon$ until the $C^r$-ball
$B(\gamma,\varepsilon)\subseteq\mathcal E^r$. (b) $C^\infty$ is dense in $C^r$ by mollification on
the circle, so pick $\gamma_0\in C^\infty$ with $\|\gamma_0-\gamma\|_{C^r}<\varepsilon/2$. (c) Put
$W:=\{h\in C^\infty:\|h-\gamma_0\|_{C^r}<\varepsilon/2\}$; the $C^\infty$ topology is generated by
all the seminorms $\|\cdot\|_{C^m}$, so $W$ is a non-empty open set. (d) $S^1$ is compact, so
$C^\infty(S^1,\mathbb R^3)$ is a Fréchet, hence Baire, space; by Thom's jet transversality theorem in
its $C^\infty$ version \cite[Thm.~3.2.8]{ref5} the set $\{h\in C^\infty:j^3h\pitchfork\Sigma\}$ is residual,
hence dense, so it meets $W$. (e) The triangle inequality gives
$\|\widetilde\gamma-\gamma\|_{C^r}<\varepsilon$, the whole $\varepsilon$-ball lies in $\mathcal E^r$,
and Lemma~\ref{lem:6.5} places $\widetilde\gamma$ in $\mathcal G^r$.
\end{proof}

Only the $C^\infty$ version of the transversality theorem is used, and the conclusion obtained is
correspondingly that $\mathcal G^r$ is \textbf{open and dense} in $\mathcal E^r$. A by-product worth
noting for \S\ref{sec:dichotomy} is that the approximating curve may always be taken $C^\infty$.

\begin{theorem}[generic rigidity, parametrised version]
\label{thm:6.8}
For every integer $r\ge4$, the set $\mathcal G^r$ of curves all of whose torsion zeros are simple is
\textbf{open and dense} in $\mathcal E^r$. \hfill$\square$
\end{theorem}

\begin{lemma}[arclength reparametrisation preserves everything]
\label{lem:6.9}
Let $\gamma\in\mathcal G^r$, $s_\gamma(t)=\int_0^t|\gamma'|$, $L=L(\gamma)$ and
$\widehat\gamma=\gamma\circ s_\gamma^{-1}$. Then $\widehat\gamma\in C^r$ is unit speed, the image and
knot type are unchanged, $\widehat\kappa>0$, and $\widehat\tau(s)=\tau_\gamma(s_\gamma^{-1}(s))$ with
$\widehat\tau'(s)=\tau_\gamma'(t)/|\gamma'(t)|$ --- a \textbf{positive factor} --- so torsion zeros
correspond bijectively and \textbf{simplicity is preserved}. The simple zeros of
$\widehat\tau\in C^{r-3}$ are isolated and the zero set is closed in the compact $\mathbb R/L\mathbb Z$,
so their \textbf{number is finite}. \hfill$\square$
\end{lemma}

For a general $\gamma\in\mathcal G^r$ one has only $\widehat\tau\in C^{r-3}$, which for $r=4$ is
merely $C^1$; there ``zero of infinite order'' and the symbol $Z_\infty(\widehat\tau)$ are not defined,
and the lemma asserts only that all zeros are simple and finite in number. When $\gamma$ is in
addition $C^\infty$ one may write $Z_\infty(\widehat\tau)=\varnothing$ as well. The next corollary is
stated so as to need only the former.

\begin{corollary}[generic curves are determined up to $E(3)$]
\label{cor:6.10}
Let $r\ge4$, $\gamma\in\mathcal G^r$ and let $\widehat\gamma$ be its arclength reparametrisation, of
length $L$. Let $\widetilde\gamma:\mathbb R/L\mathbb Z\to\mathbb R^3$ be $C^4$, unit speed, with
$\widetilde\kappa>0$, sharing \textbf{the same arclength label} with $\widehat\gamma$ and satisfying
$\widetilde\kappa=\widehat\kappa$ and $|\widetilde\tau|=|\widehat\tau|$ pointwise. Then
$\widetilde\gamma=Q\widehat\gamma+a$ for some $Q\in O(3)$ and $a\in\mathbb R^3$.
\end{corollary}

\begin{proof}
We argue directly, without the $C^\infty$ classification of \S\ref{sec:3.2}, which is unavailable at
this regularity. By Lemma~\ref{lem:6.9} the zeros of $\widehat\tau\in C^{r-3}\subseteq C^1$ are all
simple and finite in number, say $Z=\{s_1,\dots,s_n\}$ with $n=0$ allowed; by Remark~\ref{rem:4.2},
$\widetilde\gamma\in C^4$ gives $\widetilde\tau\in C^1$.

\begin{enumerate}
\item On $(\mathbb R/L\mathbb Z)\setminus Z$ we have $\widehat\tau\neq0$, so
   $\varepsilon:=\widetilde\tau/\widehat\tau$ is continuous with values in $\{\pm1\}$, hence locally
   constant by Lemma~\ref{lem:3.1}, which needs only continuity.
\item Each $s_i$ is a simple zero and $\widehat\tau,\widetilde\tau\in C^1$, so the case $k=1$ of
   Lemma~\ref{lem:3.3} applies --- its proof uses only $C^k=C^1$, as recorded in
   Remark~\ref{rem:3.4}(iv) --- giving a neighbourhood $W_i$ of $s_i$ and $\eta_i\in\{\pm1\}$ with
   $\widetilde\tau\equiv\eta_i\widehat\tau$ on $W_i$. Hence $\varepsilon$ extends to a locally
   constant function on all of $\mathbb R/L\mathbb Z$.
\item $\mathbb R/L\mathbb Z$ is \textbf{connected}, so $\varepsilon$ is a constant $\pm1$, i.e.
   $\widetilde\tau\equiv\pm\widehat\tau$.
\item The $+$ case is Lemma~\ref{lem:2.10}; in the $-$ case apply Lemma~\ref{lem:2.7} to
   $R\widetilde\gamma$ and then Lemma~\ref{lem:2.10}, obtaining $\widetilde\gamma=Q\widehat\gamma+a$
   with $\det Q=-1$. Lemma~\ref{lem:2.10} requires both curves to be at least $C^3$: here
   $\widehat\gamma\in C^r$ with $r\ge4$ and $\widetilde\gamma\in C^4$, so both qualify, and accordingly
   $\widehat\kappa,\widetilde\kappa\in C^{r-2},C^2$ and $\widehat\tau,\widetilde\tau\in C^{r-3},C^1$
   are continuous.
\end{enumerate}
\end{proof}

\begin{corollary}[residuality, and genericity in the reparametrisation quotient]
\label{cor:6.11}
Let $r\ge4$.
\begin{enumerate}
\item $\mathcal G^r$ is residual in $\mathcal E^r$.
\item Let $q:\mathcal E^r\to\mathcal E^r/\mathrm{Diff}^+(S^1)$ be the quotient map for the action by
   orientation-preserving $C^r$ reparametrisation, the target carrying the quotient topology. Then
   $q(\mathcal G^r)$ is open and dense, hence residual, in $\mathcal E^r/\mathrm{Diff}^+(S^1)$.
\end{enumerate}
\end{corollary}

\begin{proof}
(1) An open dense subset of any topological space is residual, its complement being closed with
empty interior, hence nowhere dense; no Baire property is used.

(2) The property ``all zeros of $D_\gamma$ are simple'' is invariant under orientation-preserving
$C^r$ reparametrisation. Indeed, for $\varphi\in\mathrm{Diff}^+(S^1)$ the chain rule gives
$(\gamma\circ\varphi)'=\varphi'\,\gamma'\circ\varphi$,
$(\gamma\circ\varphi)''=(\varphi')^2\gamma''\circ\varphi+\varphi''\,\gamma'\circ\varphi$ and
$(\gamma\circ\varphi)'''=(\varphi')^3\gamma'''\circ\varphi+3\varphi'\varphi''\,\gamma''\circ\varphi
+\varphi'''\,\gamma'\circ\varphi$; every term of the determinant other than the leading one repeats a
column, so
\[
D_{\gamma\circ\varphi}=(\varphi')^{6}\,\bigl(D_\gamma\circ\varphi\bigr),\qquad \varphi'>0,
\]
and zeros correspond bijectively with their simplicity preserved. Hence $\mathcal G^r$ is saturated:
$q^{-1}\bigl(q(\mathcal G^r)\bigr)=\mathcal G^r$. Since $\mathcal G^r$ is open, the definition of the
quotient topology makes $q(\mathcal G^r)$ open. For density, let $U\neq\varnothing$ be open in the
quotient; then $q^{-1}(U)$ is a non-empty open subset of $\mathcal E^r$, so it meets the dense set
$\mathcal G^r$, and therefore $U\cap q(\mathcal G^r)\neq\varnothing$. Being open and dense,
$q(\mathcal G^r)$ is residual by (1).
\end{proof}

\begin{remark}[the exact scope of the genericity statement]
\label{rem:6.12}
Corollary~\ref{cor:6.11} is a statement about the \textbf{space of maps} $\mathcal E^r$ and its
reparametrisation quotient, and about nothing else. It does not assert genericity inside a fixed
isotopy class, nor in any space of images or of knots. Density is asserted relative to $\mathcal E^r$, which is the space in which the problem is posed,
since $\kappa>0$, regularity and embeddedness are standing hypotheses (S1); the position of
$\mathcal G^r$ inside the larger space $C^r(S^1,\mathbb R^3)$ is not considered.
\end{remark}

%% file: sections/07-engine.tex
\section{The one-dimensional engine: inverse stability of the signed square root}
\label{sec:engine}

Everything quantitative in this paper rests on one estimate about functions of one variable, which we
now prove and show to be of optimal order. The question it answers is the quantitative form of
\S\ref{sec:branch}: there, knowing $|f|$ exactly, we listed the smooth $g$ with $|g|=|f|$; here we
know $f^2$ only approximately, and ask how well $f$ is then determined, modulo the one global sign
that no unsigned datum can ever see.

Throughout this section and the next, all norms are \textbf{sums},
$\|u\|_{C^k}=\sum_{j\le k}\|u^{(j)}\|_{C^0}$, with $\|u\|_{L^1}=\int_{\mathbb R/L_0\mathbb Z}|u|$
unnormalised, and
\[
\log_+\tfrac1t:=\max\{0,\ln\tfrac1t\},\qquad \rho(0):=0,\qquad
\rho(t):=t\bigl(1+\log_+\tfrac1t\bigr).
\]
Then $\rho$ is continuous and strictly increasing on $[0,\infty)$ and satisfies $t\le\rho(t)$:
for $t\ge1$ one has $\rho(t)=t$; for $0<t<1$, $\rho(t)=t-t\ln t$ with $\rho'(t)=-\ln t>0$; and the
two branches match at $t=1$.

\begin{theorem}[uniform log-Lipschitz bound]
\label{thm:7.1}
Fix $B,\delta,L_0>0$ with $\delta\le B$, and let
$f,g\in C^2(\mathbb R/L_0\mathbb Z)$ satisfy
\[
\|f\|_{C^2}+\|g\|_{C^2}\le B,\qquad f^2+(f')^2\ge\delta^2,\qquad g^2+(g')^2\ge\delta^2 .
\]
Put $\eta:=\|f^2-g^2\|_{C^0}$. Then there are a constant $C_A$ depending only on $(B,\delta,L_0)$ and
a \textbf{global} sign $\varepsilon\in\{\pm1\}$, independent of $s$, with
\[
\|f-\varepsilon g\|_{L^1}\ \le\ C_A(B,\delta,L_0)\,\rho(\eta);
\]
consequently, $\{\pm1\}$ having two elements,
\[
\inf_{\varepsilon'=\pm1}\|f-\varepsilon'g\|_{L^1}
=\min_{\varepsilon'=\pm1}\|f-\varepsilon'g\|_{L^1}\ \le\ C_A(B,\delta,L_0)\,\rho(\eta).
\]
One may take the explicit constant
\[
C_A=\frac{2BL_0}{\eta_0}+\frac{L_0}{2\delta}+\frac{2BL_0}{\delta^2}+\frac{L_0}{a_0}
+N_*\Bigl(\frac{32B}{\delta^2}+4Bc_1^2+\frac2\delta\bigl(1+\bigl|\ln(r\sqrt{2B\delta})\bigr|\bigr)\Bigr),
\]
with $r=\frac{\delta}{4B}$, $a_0=\frac{\delta^2}{16B}$, $N_*=\frac{2BL_0}{\delta}$,
$c_1=\max\{\frac4\delta,\frac1{\sqrt{2B\delta}}\}$ and $\eta_0=\frac{a_0^2}2=\frac{\delta^4}{512B^2}$.
In the range in which it is used below --- namely
\[
B\ge1,\qquad 0<\delta\le\min\{1,B\}
\]
--- this constant obeys the \textbf{non-asymptotic} bound
\begin{equation}\label{eq:7.3}\tag{7.3}
C_A(B,\delta,L_0)\ \le\ 1300\,B^3L_0\,\delta^{-4},
\end{equation}
verified term by term in Appendix~\ref{sec:C.3}. Equivalently, in asymptotic notation with the
quantifiers displayed, for each fixed $B\ge1$ and $L_0>0$ one has
$C_A(B,\delta,L_0)=O_{B,L_0}\bigl(\delta^{-4}\bigr)$ as $\delta\downarrow0$. No claim is made that
the exponent $4$ is optimal.
\end{theorem}

\begin{proof}
Let $u$ stand for $f$ or $g$ and let $Z(u)$ be its zero set. We may assume $\delta\le B$,
since replacing $\delta$ by $\min(\delta,B)$ only weakens the hypothesis; note that the
non-degeneracy condition already forces $\delta\le\sqrt2\,B$. We use the convention
$0\cdot\log_+\frac10:=0$, so $\rho(0)=0$.

\textbf{(P1) Structure of the zero set.}
(a) If $u(z)=0$, substituting into $u^2+(u')^2\ge\delta^2$ gives $|u'(z)|\ge\delta$: the
non-degeneracy hypothesis is exactly a quantitative simplicity of zeros.
(b) For $|s-z|\le2r$ we get $|u'(s)|\ge|u'(z)|-\|u''\|_{C^0}\cdot2r\ge\delta-2Br=\frac\delta2$; in
particular $u'$ does not change sign on $[z-2r,z+2r]$.
(c) By (b) and the mean value theorem, $|u(s)|\ge\frac\delta2|s-z|$ for $|s-z|\le2r$. Hence two
distinct zeros are more than $2r$ apart, so that
\[
N(u):=\#Z(u)\ \le\ \frac{L_0}{2r}\ =\ \frac{2BL_0}{\delta}\ =\ N_* .
\]

(c$'$) \textbf{The circle is long compared with $r$, so the arcs used below are ordinary intervals.}
Suppose $Z(u)\ne\varnothing$. By (a) each zero is a sign change, and a continuous function on the
circle has an even number of sign changes, so $N(u)$ is \textbf{even}, hence $N(u)\ge2$. Listing $Z(u)$
cyclically as $z_1<\dots<z_{N}$ with $z_{N+1}:=z_1+L_0$, the $N\ge2$ adjacent gaps sum to exactly
$L_0$ and each exceeds $2r$ by (c), so
\[
L_0=\sum_{j=1}^{N}(z_{j+1}-z_j)\ >\ N\cdot2r\ \ge\ 4r .
\]
Consequently, for each $z\in Z(u)$ the periodic lift centred at $z$ realises
$[z-2r,z+2r]\subset(z-\frac{L_0}2,z+\frac{L_0}2)$ \textbf{injectively} as an ordinary closed interval of
$\mathbb R$, with no wrap-around; monotonicity arguments, the mean value theorem and integrations
performed on $[z\pm r]$ and $[z\pm2r]$ may therefore be handled as on a real interval. In (b) and
(c) themselves, $|s-z|$ always denotes distance on the circle and the estimates are made along the
shorter arc, so those two steps do not use the present conclusion. Every later step that treats such
an arc as an ordinary interval --- namely (d), (P3), (P5) and (P6) --- is invoked only in the branch of
(P2) in which $Z(f)$ and $Z(g)$ are both non-empty, so (c$'$) applies there to both functions.

(d) If $\operatorname{dist}(s_0,Z(u))>r$ then $|u(s_0)|\ge a_0$. Indeed, suppose $|u(s_0)|<a_0$.
Since $\delta\le B$ we have $a_0=\frac{\delta^2}{16B}\le\frac{\delta}{16}<\frac\delta2$, so
non-degeneracy gives $|u'(s_0)|\ge\sqrt{\delta^2-\frac{\delta^2}{4}}=\frac{\sqrt3}{2}\delta>\frac\delta2$;
hence on $J:=[s_0-r,s_0+r]$ one has $|u'|\ge\frac\delta2-Br=\frac\delta4$ without change of sign, so
$u$ is strictly monotone on $J$. Moving in the direction in which $|u|$ decreases, $u$ reaches $0$
after a distance at most $\frac{|u(s_0)|}{\delta/4}<\frac{4a_0}{\delta}=r$, and that zero lies in
$J$, contradicting $\operatorname{dist}(s_0,Z(u))>r$.

\textbf{(P0) The degenerate case $\eta=0$.} Here $f^2\equiv g^2$, so $|f|\equiv|g|$ and $Z(f)=Z(g)=:Z$;
moreover $f\not\equiv0$, else $f^2+(f')^2\equiv0<\delta^2$. If $Z=\varnothing$ then $g/f$ is
continuous on the connected $S^1$ with values in $\{\pm1\}$, hence constant. If $Z\ne\varnothing$
then by (P1)(a),(c) the set $Z$ consists of finitely many simple zeros, and on each connected
component of $S^1\setminus Z$ the quotient $g/f$ equals some $\varepsilon_j\in\{\pm1\}$. Let $z\in Z$
be the common endpoint of two adjacent components $A^-$ and $A^+$ with $\varepsilon_-\ne\varepsilon_+$.
Since $f,g\in C^1$ and $f(z)=g(z)=0$, computing the derivative from both sides gives
\[
g'(z)=\lim_{s\uparrow z}\frac{g(s)}{s-z}=\varepsilon_-f'(z),\qquad
g'(z)=\lim_{s\downarrow z}\frac{g(s)}{s-z}=-\varepsilon_-f'(z),
\]
so $f'(z)=0$, contradicting (P1)(a). Hence adjacent components carry the same sign, and going once
around the circle all of them agree. Thus $f\equiv\varepsilon g$ and both sides of the asserted
inequality vanish. From now on $\eta>0$.

\textbf{(P2) The three trivial branches.}
\begin{enumerate}
\item \textbf{$\eta\ge\eta_0$}: $\inf_{\varepsilon}\|f-\varepsilon g\|_{L^1}\le\|f\|_{L^1}+\|g\|_{L^1}\le2BL_0
   \le\frac{2BL_0}{\eta_0}\,\eta$.
\item \textbf{$Z(f)=Z(g)=\varnothing$}: at a minimum point of $|u|$ on the compact circle one has $u'=0$,
   since $u$ does not change sign and $|u|$ is differentiable there, whence $\min|u|\ge\delta$ by
   non-degeneracy and neither $f$ nor $g$ changes sign. Taking
   $\varepsilon:=\operatorname{sign}(f)\operatorname{sign}(g)$ gives $|f+\varepsilon g|=|f|+|g|\ge2\delta$
   and pointwise $|f-\varepsilon g|=\frac{|f^2-g^2|}{|f+\varepsilon g|}\le\frac{\eta}{2\delta}$;
   integrating gives $\le\frac{L_0}{2\delta}\eta$.
\item \textbf{Exactly one of them has a zero}, say $f(z)=0$ and $Z(g)=\varnothing$: then
   $\eta\ge|f(z)^2-g(z)^2|=g(z)^2\ge\delta^2$, which reduces to branch 1 and gives
   $\le\frac{2BL_0}{\delta^2}\eta$.
\end{enumerate}

All three are already \textbf{linear} estimates: no logarithm arises in any of them. From now on
$0<\eta<\eta_0$ and both $Z(f)$ and $Z(g)$ are non-empty; put
\[
d:=\frac{4\sqrt\eta}{\delta}\ <\ r\qquad\text{since }\sqrt\eta<\sqrt{\eta_0}<a_0=\tfrac{\delta r}{4}.
\]
The number $d$ is the scale at which the two zero sets can fail to coincide, and it is the source of
everything that follows.

\textbf{(P3) The zeros are matched bijectively and cyclically.} Let $z\in Z(f)$. From $f(z)=0$ we get
$g(z)^2=|g(z)^2-f(z)^2|\le\eta$, i.e. $|g(z)|\le\sqrt\eta<a_0$. Running the \textbf{proof} of (P1)(d)
rather than its statement: $|g(z)|<a_0<\frac\delta2$ gives $|g'(z)|>\frac\delta2$, so $g$ is strictly
monotone on $[z-r,z+r]$ with $|g'|\ge\frac\delta4$, and there is $w\in Z(g)$ with
\[
|w-z|\ \le\ \frac{|g(z)|}{\delta/4}\ \le\ \frac{4\sqrt\eta}{\delta}\ =\ d .
\]
Symmetrically every $w\in Z(g)$ lies within $d$ of some point of $Z(f)$. \emph{Injectivity}: if
$z\ne z'\in Z(f)$ were matched to the same $w$ then $|z-z'|\le2d<2r$, contradicting the separation in
(P1)(c); the same applies in the other direction, and the two maps are mutually inverse, since
$|w-z|\le d$ and $w$ matched back to $z''$ force $|z''-z|\le2d<2r$ and hence $z''=z$. \emph{Order
preservation}: both maps move points by at most $d$ while the gaps within each set exceed $2r>2d$, so
they cannot cross. \emph{Evenness}: each of $N(f)$ and $N(g)$ is even by (P1)(c$'$), and the bijection
just established makes them equal; write $N(f)=N(g)=:2k$.

\textbf{(P4) The global sign is unique.} List $Z(f)$ cyclically as $z_1<\dots<z_{2k}$ with matched points
$w_j=w(z_j)$, and set
\[
A_j:=(z_j+d,\ z_{j+1}-d),
\]
non-empty since $z_{j+1}-z_j>2r>2d$. Now $f$ has no zero in $(z_j,z_{j+1})$, and all zeros of $g$ lie
in $\bigcup_i[z_i-d,z_i+d]$ by the reverse part of (P3), which is disjoint from $A_j$; hence neither
$f$ nor $g$ changes sign on $A_j$, and we may put $\varepsilon_j:=\operatorname{sign}(fg)|_{A_j}$.
Between $A_j$ and $A_{j+1}$ the function $f$ changes sign at the simple zero $z_{j+1}$ and $g$
changes sign at $w_{j+1}\in[z_{j+1}-d,z_{j+1}+d]$, and there is no other sign change in that stretch,
so $\varepsilon_{j+1}=(-1)(-1)\varepsilon_j=\varepsilon_j$. Going once around the circle,
$\varepsilon_1=\dots=\varepsilon_{2k}=:\varepsilon$.

This step is the quantitative counterpart of the criterion $c\le1$ of Theorem~\ref{thm:3.9}. There, a sign flip
was forbidden because it would have to be invisible to every derivative; here it is forbidden because
the two functions cross zero in matched pairs and the parity of the crossings is rigid. Replacing $g$
by $\varepsilon g$, which changes neither the hypotheses nor $\eta$, it remains to estimate
$\|f-g\|_{L^1}$.

\textbf{(P5) Away from the zeros.} Let $\Omega:=\{s:\operatorname{dist}(s,Z(f))>r\}$. By (P1)(d),
$|f|\ge a_0$; moreover $g^2\ge f^2-\eta\ge a_0^2-\eta\ge\frac{a_0^2}{2}$ since
$\eta<\eta_0=\frac{a_0^2}{2}$, so $|g|\ge\frac{a_0}{\sqrt2}$. As $r>d$ we have
$\Omega\subset\bigcup_jA_j$, hence $fg>0$ and $|f+g|=|f|+|g|\ge a_0$, so
\[
|f-g|=\frac{|f^2-g^2|}{|f+g|}\le\frac{\eta}{a_0}\ \Longrightarrow\
\int_\Omega|f-g|\le\frac{L_0}{a_0}\,\eta .
\]

\textbf{(P6) Near a single zero: the misalignment window and the split at $\sigma$.} Fix $z\in Z(f)$,
$w=w(z)$ and $I:=[z-r,z+r]$. For $s\in I$ we have $|s-w|\le r+d\le2r$, so (P1)(b) gives
$|f'|,|g'|\ge\frac\delta2$ on $I$, each without change of sign; since $fg>0$ on
$A_j\cap I\ne\varnothing$ and both functions cross their zeros monotonically, $f'$ and $g'$ have the
\textbf{same} sign on $I$, and we may assume both positive. Write $m:=\max(z,w)$ and $m_-:=\min(z,w)$, so
$m-m_-\le d$.

\begin{itemize}
\item \textbf{(i) The misalignment window $[m_-,m]$}, of length $\le d$. Here
  $|f(s)|=|f(s)-f(z)|\le B|s-z|\le Bd$ and likewise $|g(s)|\le Bd$, so
  \[
  \int_{m_-}^{m}|f-g|\ \le\ 2Bd\cdot d\ =\ 2Bd^2\ =\ \frac{32B}{\delta^2}\,\eta .
  \]
  This is the stretch where width $\times$ height returns to the linear order, since
  $d\sim\sqrt\eta$ and hence $d^2\sim\eta$: the misalignment of the zeros, though large compared with
  $\eta$, occupies a correspondingly short interval.
\item \textbf{(ii) The one-sided region $[m,z+r]$.} Put $t:=s-m\ge0$. By (P1)(c),
  $f(s)\ge\frac\delta2(s-z)\ge\frac\delta2t$ and $g(s)\ge\frac\delta2(s-w)\ge\frac\delta2t$, so
  $f+g\ge\delta t$ and
  \begin{equation}\label{eq:7.1}\tag{7.1}
  |f-g|=\frac{|f^2-g^2|}{f+g}\ \le\ \frac{\eta}{\delta t}.
  \end{equation}
  On the other hand one of the two functions vanishes at $m$ while the other has absolute value
  $\le Bd$, so $|f(m)-g(m)|\le Bd$; with $|(f-g)'|\le2B$ this gives
  \begin{equation}\label{eq:7.2}\tag{7.2}
  |f-g|\ \le\ Bd+2Bt.
  \end{equation}
  Let $\tau:=\sqrt{\eta/(2B\delta)}$ be the scale at which \eqref{eq:7.1} meets the linear part
  of \eqref{eq:7.2}, and put
  \[
  \sigma:=\max(d,\tau)\ \le\ c_1\sqrt\eta,\qquad c_1=\max\Bigl\{\frac4\delta,\frac1{\sqrt{2B\delta}}\Bigr\}.
  \]
  Using \eqref{eq:7.2} on $[0,\sigma]$ and \eqref{eq:7.1} on $[\sigma,r]$,
  \[
  \int_0^{r}|f-g|\,dt\ \le\ \bigl(Bd\sigma+B\sigma^2\bigr)+\frac{\eta}{\delta}\ln\frac{r}{\sigma}
  \ \le\ 2Bc_1^2\eta+\frac{\eta}{\delta}\Bigl(\ln\bigl(r\sqrt{2B\delta}\bigr)+\tfrac12\log_+\tfrac1\eta\Bigr),
  \]
  the last step using $d\le\sigma\le c_1\sqrt\eta$ and $\sigma\ge\tau=\sqrt{\eta/(2B\delta)}$.
\item \textbf{(iii) The other side $[z-r,m_-]$} is completely symmetric and gives the same bound.
\end{itemize}

Adding the three pieces,
\[
\int_I|f-g|\ \le\ \frac{32B}{\delta^2}\eta+4Bc_1^2\eta
+\frac{2\eta}{\delta}\Bigl(\bigl|\ln\bigl(r\sqrt{2B\delta}\bigr)\bigr|+\tfrac12\log_+\tfrac1\eta\Bigr).
\]

\textbf{(P7) Assembly.} The $r$-neighbourhoods of the points of $Z(f)$ are pairwise disjoint by (P1)(c),
their complement is contained in $\Omega$, and there are $N(f)\le N_*$ of them. Adding (P5) and
$N(f)$ copies of (P6), and taking the largest of the constants together with the three branches of
(P2), gives $\|f-\varepsilon g\|_{L^1}\le C_A\,\eta\bigl(1+\log_+\frac1\eta\bigr)=C_A\,\rho(\eta)$
with exactly the $C_A$ displayed above, $\varepsilon$ being the sign constructed in (P4).

Finally, $\{\pm1\}$ being finite,
$\inf_{\varepsilon'=\pm1}\|f-\varepsilon'g\|_{L^1}\le\|f-\varepsilon g\|_{L^1}\le C_A\,\rho(\eta)$,
which is the infimum form.
\end{proof}

The whole estimate is linear in $\eta$ except at one point. The only source of a logarithm in the
entire proof is the integral of \eqref{eq:7.1} over $t\in[\sigma,r]$ in (P6)(ii), namely
$\int_\sigma^r\frac{\eta}{\delta t}\,dt$; Figure~\ref{fig:5} places the two regions of (P6) side by
side. That estimate is \textbf{saturable}: as soon as $f-g$ follows the
profile $\eta/(2\delta t)$, the difference $f^2-g^2$ is flattened to the constant $\eta$ across the
whole range of scales, and the integral genuinely accumulates a logarithm. The $C^2$ regularity
available here does not remove it: a $C^2$ bound only lifts the inner cut-off from $\sqrt\eta$ to
$\eta^{1/3}$, since
$(\eta/t)''\sim\eta/t^3\le B$ requires $t\gtrsim\eta^{1/3}$, which lowers the logarithmic coefficient
from $\tfrac12$ to $\tfrac13$ without removing it. The next theorem confirms that this is a genuine
feature by exhibiting a family that realises the profile.

\begin{figure}[tbp]
\centering
\includegraphics{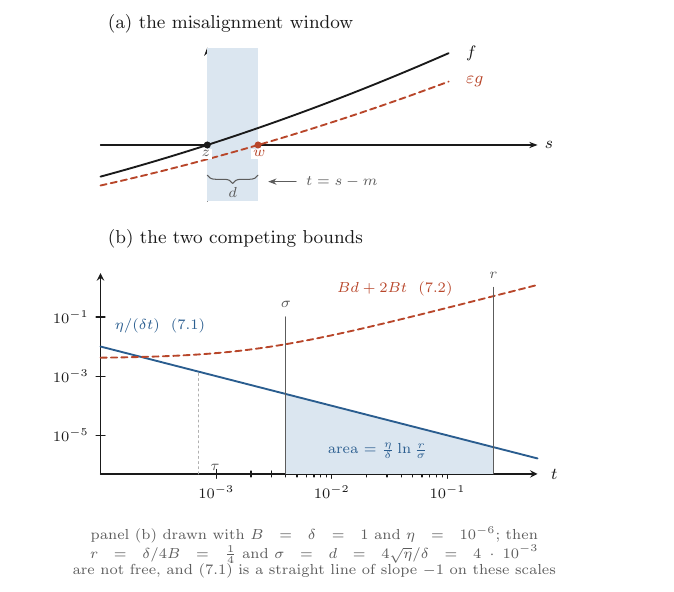}
\caption{\emph{Where the logarithm comes from.} Step (P6) of the proof of
Theorem~\ref{thm:7.1}, near one matched pair of zeros $z\in Z(f)$, $w=w(z)\in Z(g)$, with
$m_-=\min(z,w)$, $m=\max(z,w)$, $t=s-m$. (a) On the misalignment window $f$ and $\varepsilon g$
have opposite signs, but the window is only $d=|z-w|$ wide while $|f|,|\varepsilon g|\le Bd$ on it,
so width $\times$ height returns the linear order $2Bd^2=32B\eta/\delta^2$. (b) To the right of $m$
the two bounds \eqref{eq:7.1} and \eqref{eq:7.2} compete; the proof uses \eqref{eq:7.2} on
$[0,\sigma]$ and \eqref{eq:7.1} on $[\sigma,r]$, where the shaded area is
$\frac\eta\delta\ln\frac r\sigma=\frac\eta\delta(\tfrac12\log_+\tfrac1\eta+O(1))$ --- the only
logarithm in the paper. As $\eta\downarrow0$ the left edge $\sigma\asymp\sqrt\eta$ slides left and
the shaded region gains $\tfrac12\log\tfrac1\eta$ of width; that widening \emph{is} the logarithm.
Both curves are schematic. The dotted mark is $\tau=\sqrt{\eta/(2B\delta)}$; the split nevertheless
always sits at $\sigma=\max(d,\tau)=d$, since $\delta\le B$ forces $d/\tau\ge4\sqrt2$.}
\label{fig:5}
\end{figure}

\begin{theorem}[the logarithm is not removable]
\label{thm:7.2}
Let $L_0=2\pi$ and $\lambda=\frac14$. Let
$q(x)=10x^3-15x^4+6x^5$ and define the piecewise extension
\[
\psi(x)=\begin{cases}0,&x\le0,\\ q(x),&0<x<1,\\ 1,&x\ge1,\end{cases}\qquad
\chi(s)=\psi\bigl(2(1-|s|)\bigr),
\]
so that $\chi\in C^2$, $\chi=1$ for $|s|\le\frac12$ and $\chi=0$ for $|s|\ge1$. For
$0<\epsilon\le\epsilon_0:=10^{-3}$ put
\[
T=\epsilon^{1/3},\qquad h_\epsilon(s)=\frac{\lambda\epsilon\,\chi(s)}{\sqrt{s^2+T^2}},\qquad
f(s)=\sin s,\qquad g_\epsilon=f-h_\epsilon .
\]
Then $f,g_\epsilon\in C^2(\mathbb R/2\pi\mathbb Z)$ and, with $B,\delta,L_0$ \textbf{independent of
$\epsilon$},
\[
\|f\|_{C^2}+\|g_\epsilon\|_{C^2}\le7=:B,\qquad
f^2+(f')^2\equiv1,\qquad g_\epsilon^2+(g_\epsilon')^2\ge0.97>0.81=(0.9)^2=:\delta^2 .
\]
Moreover $1.87\lambda\epsilon\le\eta_\epsilon:=\|f^2-g_\epsilon^2\|_{C^0}\le2.1\lambda\epsilon\to0$
and
\[
\inf_{\varepsilon=\pm1}\|f-\varepsilon g_\epsilon\|_{L^1}=\|h_\epsilon\|_{L^1}
\ge\frac{2\lambda\epsilon}3\ln\frac1\epsilon,
\]
\[
\text{so}\qquad
\frac{\inf_{\varepsilon}\|f-\varepsilon g_\epsilon\|_{L^1}}{\eta_\epsilon}
\ \ge\ 0.3174\log_+\tfrac1{\eta_\epsilon}-0.2415\ \longrightarrow\ \infty .
\]
In particular no constant $C$ depending only on $(B,\delta,L_0)$ satisfies
$\inf_\varepsilon\|f-\varepsilon g\|_{L^1}\le C\eta$ on this family, so $\rho$ cannot be replaced by
the identity in Theorem~\ref{thm:7.1}.
\end{theorem}

\begin{proof}
The design of the family is the point: $h_\epsilon$ follows the saturating profile
$\eta/t$ down to an inner cut-off $T$, and $T=\epsilon^{1/3}$ is chosen as the largest cut-off
compatible with a \textbf{fixed} $C^2$ bound, since $\epsilon T^{-3}=1$. The verification that
$f,g_\epsilon$ meet the hypotheses of Theorem~\ref{thm:7.1} with the stated $\epsilon$-independent constants ---
the $C^2$ splice and periodic extension, the uniform $C^2$ bound $B=7$, the non-degeneracy bound
$0.972>\delta^2$, and the two-sided bounds for $\eta_\epsilon$ --- is elementary and is carried out in
Appendix~\ref{sec:C.2}. We give here the two steps that produce the logarithm.

\textbf{(e) The $L^1$ lower bound, and the other global sign.} Since $\chi\equiv1$ on $|s|\le\frac12$ and
$h_\epsilon\ge0$,
\[
\|h_\epsilon\|_{L^1}\ \ge\ 2\int_0^{1/2}\frac{\lambda\epsilon}{\sqrt{s^2+T^2}}\,ds
=2\lambda\epsilon\operatorname{arcsinh}\frac1{2T}\ \ge\ 2\lambda\epsilon\ln\frac1T
=\frac{2\lambda\epsilon}3\ln\frac1\epsilon,
\]
using $\operatorname{arcsinh}x=\ln(x+\sqrt{x^2+1})\ge\ln(2x)$. The \textbf{other} sign must be checked
separately, since the infimum in Theorem~\ref{thm:7.1} is over both:
\[
\|f+g_\epsilon\|_{L^1}=\|2f-h_\epsilon\|_{L^1}\ \ge\ 2\int_0^{2\pi}|\sin|-\|h_\epsilon\|_{L^1}
=8-\|h_\epsilon\|_{L^1}\ \ge\ 7,
\]
because $\operatorname{supp}h_\epsilon\subset[-1,1]$ gives
$\|h_\epsilon\|_{L^1}\le2\|h_\epsilon\|_{C^0}\le2\lambda\epsilon^{2/3}\le5\times10^{-3}$. The same
estimate shows $\|h_\epsilon\|_{L^1}\to0$, so for small $\epsilon$ the infimum is attained at
$\varepsilon=+1$ and equals $\|h_\epsilon\|_{L^1}$, the branch $\varepsilon=-1$ remaining at distance
$\ge7$.

\textbf{(f) Change of variables.} By the upper bound $\eta_\epsilon\le2.1\lambda\epsilon$ of Appendix~\ref{sec:C.2}
and by (e),
\[
\frac{\inf_{\varepsilon=\pm1}\|f-\varepsilon g_\epsilon\|_{L^1}}{\eta_\epsilon}
\ \ge\ \frac{(2\lambda\epsilon/3)\ln(1/\epsilon)}{2.1\,\lambda\epsilon}=\frac2{6.3}\ln\frac1\epsilon
\ \ge\ 0.3174\ln\frac1\epsilon .
\]
Next, the \textbf{lower} bound $\eta_\epsilon\ge1.87\lambda\epsilon=0.4675\,\epsilon$ gives
$\epsilon\le\eta_\epsilon/0.4675$, hence
$\ln\frac1\epsilon\ge\ln\frac{0.4675}{\eta_\epsilon}=\ln\frac1{\eta_\epsilon}-0.7605$; the lower
bound on $\eta_\epsilon$ is what this step needs. Combining, and
noting that $\eta_\epsilon\le1$ for small $\epsilon$ so that
$\ln\frac1{\eta_\epsilon}=\log_+\frac1{\eta_\epsilon}$, together with
$0.3174\times0.7605\le0.2415$,
\[
\frac{\inf_{\varepsilon=\pm1}\|f-\varepsilon g_\epsilon\|_{L^1}}{\eta_\epsilon}
\ \ge\ 0.3174\log_+\tfrac1{\eta_\epsilon}-0.2415\ \xrightarrow[\epsilon\to0]{}\ \infty .
\]
Were there a constant $C=C(B,\delta,L_0)$ with $\inf_\varepsilon\|f-\varepsilon g\|_{L^1}\le C\eta$,
the left-hand side would be bounded along this family.
\end{proof}

\begin{remark}[the exact form of the sharpness, and a Hölder corollary]
\label{rem:7.3}
Theorems~\ref{thm:7.1} and \ref{thm:7.2}
together determine the \textbf{order of the function}: $\eta\log\frac1\eta$ cannot be improved to
$O(\eta)$. The logarithmic \textbf{coefficients} of the two bounds are not matched --- assembling
(P6)--(P7) shows that the coefficient of $\eta\log_+\frac1\eta$ delivered by the proof above is
exactly $N_*/\delta=2BL_0\delta^{-2}$, while the family of Theorem~\ref{thm:7.2} gives $0.3174$ ---
and closing that gap is a separate
question (\S\ref{sec:questions}). Finally, since
$\sup_{0<t\le1}t^\vartheta\ln\frac1t=\frac1{e\vartheta}$, Theorem~\ref{thm:7.1} yields a Hölder form: for every
$\vartheta\in(0,1)$,
\[
\inf_\varepsilon\|f-\varepsilon g\|_{L^1}\le\bigl(1+\tfrac1{e\vartheta}\bigr)C_A\,\eta^{1-\vartheta},
\]
so \textbf{every} Hölder exponent below $1$ is admissible, and only the exponent $1$ is excluded by
Theorem~\ref{thm:7.2}.
\end{remark}

%% file: sections/08-dichotomy.tex
\section{Global instability and conditional stability on the non-degenerate strata}
\label{sec:dichotomy}

Corollary~\ref{cor:6.10} says that a generic curve is determined by its unsigned datum. It says nothing about
how the curve depends on the datum. On the whole of $\mathcal G^r$ there is no such dependence in the
sense of a modulus of continuity: \S\ref{sec:8.2} constructs two sequences
of curves with fixed, distinct, non-mirror knot types whose data converge while their orbits stay
apart. What makes this more than a negative statement is that the failure can be localised. The
single quantity
\[
\Delta(\alpha):=\inf_{s}\sqrt{\tau_\alpha(s)^2+\tau_\alpha'(s)^2}
\]
--- positive precisely when all torsion zeros are simple, and measuring how simple --- is a
\textbf{quantitative non-degeneracy parameter}, and once the competing curves are held under uniform
$C^5$ and positive-curvature bounds it is the parameter in which the failure is confined.
Sections~\ref{sec:8.3}--\ref{sec:8.6} prove that stability holds, with the modulus $\rho$ of \S\ref{sec:engine},
uniformly on every stratum $\Delta\ge\delta$ and on a neighbourhood of every fixed generic curve;
\S\ref{sec:8.7} proves that the optimal constant on those strata diverges as $\delta\downarrow0$, and
that \emph{every} uniformly bounded near-collision sequence must satisfy $\Delta\to0$
(Proposition~\ref{prop:8.23}). Figure~\ref{fig:6} summarises the resulting picture.

``Globally unstable, yet conditionally stable under a priori bounds'' is a familiar pattern in inverse
problems; Koch--Rüland--Salo \cite{ref14} give a systematic account of several instability mechanisms, all
based on compression properties of a forward operator --- singular value or entropy number bounds
obtained from strong, weak or microlocal smoothing. We cite \cite{ref14} for the pattern only. The mechanism
here is of a different kind: it is not smoothing-induced compression, but the collapse of the
branches of the unsigned quotient $\tau\mapsto\tau^2$ at the torsion-degenerate locus, which is
exactly what makes it localisable in $\Delta$.

\begin{figure}[tbp]
\centering
\includegraphics{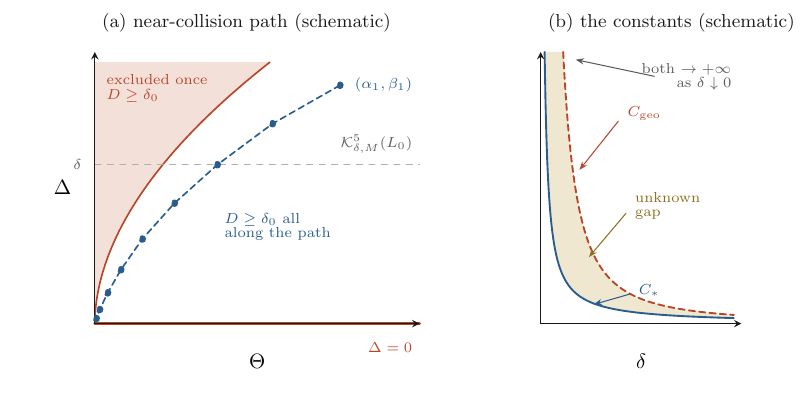}
\caption{\emph{Global instability and conditional stability across the torsion-degenerate locus.} Both panels are
drawn for one fixed exact ambiguous pair, so that the orbit distance stays at least $\delta_0$, with
$M\ge M_0(\gamma_+,\gamma_-,L_0)$ held fixed; $D$ and the constants refer to the reflection-inclusive
group $E(3)$, and $\Delta$ of a pair means $\min\{\Delta(\alpha),\Delta(\beta)\}$, as in
Proposition~\ref{prop:8.23}. (a) On the stratum $\Delta\ge\delta$, Theorem~\ref{thm:8.17} forces
$\rho(\Theta)\ge\delta_0/C_{\mathrm{geo}}(M,\delta,L_0)$ for any pair with $D\ge\delta_0$, which
excludes the shaded region; since $C_{\mathrm{geo}}$ blows up as $\delta\downarrow0$, that region
pinches to the origin, and such a sequence can reach $\Theta\to0$ only by descending to $\Delta=0$.
This is Proposition~\ref{prop:8.23}. (b) What is proved about the two constants: finiteness for every $\delta>0$,
with $C_*\le C_{\mathrm{geo}}$ bounded explicitly by \eqref{eq:8.1} (Theorem~\ref{thm:8.17}), and divergence as
$\delta\downarrow0$ (Theorem~\ref{thm:8.22}). \textbf{Both panels are schematic}: only finiteness for $\delta>0$, the monotonicity of $C_*$ in
$\delta$ and in $M$ (Definition~\ref{def:8.21}) and divergence as $\delta\downarrow0$ are asserted.}
\label{fig:6}
\end{figure}

\subsection{The normalised setting}
\label{sec:8.1}

\begin{definition}[normalisation, data, orbit distance]
\label{def:8.1}
Fix $L_0>0$. For $\gamma\in\mathcal E^r$
let $N_{L_0}(\gamma)$ be the composition of three steps: (1) arclength reparametrisation, with base
point the image of $\gamma(0)$; (2) the homothety $\lambda=L_0/L(\gamma)$; (3) the translation making
the centroid $0$. The result is a curve on $\mathbb R/L_0\mathbb Z$ of unit speed, length $L_0$ and
centroid $0$, and $N_{L_0}$ is idempotent exactly on curves that are unit speed, of length $L_0$ and
centroid $0$, whose parameter origin is the arclength origin and whose orientation agrees. Write
\[
\mathcal G^r_{L_0}:=\bigl\{\alpha\in\mathcal G^r:\ \text{unit speed, length }L_0,\
\text{centroid }0,
\]
\[
\text{parameter origin at the arclength origin, orientation fixed}\bigr\}
\]
for the normalised slice, and put
\[
\mathcal D(\gamma):=\bigl(\kappa_c,\ \tau_c^2\bigr)\in C^{r-3}\bigl(\mathbb R/L_0\mathbb Z,\mathbb R^2\bigr),
\qquad c=N_{L_0}(\gamma),
\]
\[
D(\alpha,\beta):=\inf_{g\in E(3)}d_H\bigl(g(\operatorname{Im}\alpha),\operatorname{Im}\beta\bigr),
\qquad \bar D(\alpha,\beta):=D\bigl(N_{L_0}(\alpha),N_{L_0}(\beta)\bigr),
\]
with $d_H$ the Hausdorff distance of non-empty compact sets. Since $D$ is defined on \textbf{images}, it
automatically quotients by $\mathrm{Diff}(S^1)$, by the base point and by translations; the infimum
runs over \textbf{all of $E(3)$, including the branch $\det=-1$}. For curves defined on one and the
same $\mathbb R/L_0\mathbb Z$ we also use the finer \textbf{labelled orbit distance}
\[
d_{\mathrm{lab}}(\alpha,\beta):=\inf_{g\in E(3)}\|\alpha-g\beta\|_{C^0(\mathbb R/L_0\mathbb Z)} ,
\]
which compares the two \emph{parametrisations} in the common label instead of the two images.
\end{definition}

\begin{lemma}[the two distances, and why both are recorded]
\label{lem:8.2}
On curves defined on the common circle $\mathbb R/L_0\mathbb Z$, $d_{\mathrm{lab}}$ is a
pseudo-metric, invariant under independent $E(3)$ actions, and
\[
D(\alpha,\beta)\ \le\ d_{\mathrm{lab}}(\alpha,\beta).
\]
\end{lemma}

\begin{proof}
Symmetry follows because $g$ is an isometry, so
$\|\alpha-g\beta\|_{C^0}=\|g^{-1}\alpha-\beta\|_{C^0}$ and $g\mapsto g^{-1}$ is a bijection of $E(3)$;
the triangle inequality follows by composing the two near-optimal isometries. Invariance under
independent actions, $d_{\mathrm{lab}}(g_1\alpha,g_2\beta)=d_{\mathrm{lab}}(\alpha,\beta)$, holds
because $\|g_1\alpha-gg_2\beta\|_{C^0}=\|\alpha-(g_1^{-1}gg_2)\beta\|_{C^0}$ and
$h=g_1^{-1}gg_2$ ranges over all of $E(3)$ as $g$ does. For the inequality, a pointwise bound
$\sup_s|\alpha(s)-g\beta(s)|\le c$ gives
$d_H\bigl(\operatorname{Im}\alpha,\operatorname{Im}(g\beta)\bigr)\le c$, because each point of either
image lies within $c$ of the point of the other carrying the same label; applying the isometry
$g^{-1}$ turns this into
$d_H\bigl(g^{-1}\operatorname{Im}\alpha,\operatorname{Im}\beta\bigr)\le c$ with $g^{-1}\in E(3)$, so
taking infima gives $D(\alpha,\beta)\le d_{\mathrm{lab}}(\alpha,\beta)$.
\end{proof}

The hierarchy this produces is the one the statements below follow. The \textbf{input} is a pair of
data read in the common arclength label; the \textbf{primary output} is a bound on
$d_{\mathrm{lab}}$, that is, on the parametrised curves modulo $E(3)$ but \emph{not} modulo
relabelling; and the \textbf{secondary output}, obtained from Lemma~\ref{lem:8.2}, is the bound on the
label-free Hausdorff orbit distance $D$. The negative results are stated for $D$ and therefore hold a
fortiori for $d_{\mathrm{lab}}$.

A chiral curve and its mirror image have \emph{identical} unsigned data, so restricting the infimum
to $SE(3)$ would leave a pair with zero data difference and positive distance, and
Theorem~\ref{thm:8.17} would fail. Lemma~\ref{lem:8.15} is where the reflection re-enters the proof.

\begin{remark}[scale, and the argument of the logarithm]
\label{rem:8.3}
Everything in \S\S\ref{sec:engine}--\ref{sec:dichotomy} is stated on the normalised slice, where the
length is the fixed number $L_0$ and the label is a fixed arclength label. All quantities occurring
are therefore real numbers computed in one fixed unit of length, and expressions such as
$\rho(\Theta)=\Theta(1+\log_+\frac1\Theta)$ are well defined. The dependence on that unit is explicit: writing
$\widehat\gamma(u):=\gamma(L_0u)/L_0$ for the length-one normalisation, $u\in\mathbb R/\mathbb Z$, one
has $\widehat\kappa=L_0\kappa$, $\widehat\tau=L_0\tau$ and $\widehat\tau\,'=L_0^2\tau'$ pointwise, so
the length-one datum $(L_0\kappa,L_0^2\tau^2)$ is dimensionless and distances are divided by $L_0$.
Both $\Theta$ and $\Delta$ are then comparable to their length-one counterparts within a factor
$\max\{L_0,L_0^2\}/\min\{L_0,L_0^2\}$, so every statement below may be read on the length-one slice at
the cost of an explicit $L_0$-dependent constant. Note that $\Delta$, which adds $\tau$ to $\tau'$, is
\textbf{not} scale invariant and is meaningful only relative to the fixed $L_0$; we keep $L_0$
explicit because the sharpness family of Theorem~\ref{thm:7.2} lives on a circle of circumference
$2\pi$. A reader who prefers a dimensionless formulation may therefore simply set $L_0=1$
throughout, read $\mathcal D=(\widehat\kappa,\widehat\tau^{\,2})$ and
$\Delta=\inf_u\sqrt{\widehat\tau^{\,2}+(\partial_u\widehat\tau)^2}$ on
$\mathbb R/\mathbb Z$, and recover the general case from the displayed scaling; the constants
\eqref{eq:8.1} and \eqref{eq:7.3} are stated with their $L_0$-dependence explicit precisely so that
this passage is mechanical.
\end{remark}

\begin{remark}[why the squared datum $\tau^2$ and not $|\tau|$]
\label{rem:8.4}
The two are equivalent \textbf{at the
level of values}, since $|\tau|=\sqrt{\tau^2}$, but not at the level of topologies. Let $s_0$ be a
simple zero of $\tau$, so $\tau(s_0)=0$ and $\tau'(s_0)=c\neq0$. For a general $C^r$ curve with
$r\ge4$ one has only $\tau\in C^{r-3}$, and for $r=4$ merely $\tau\in C^1$, so one can only write
$\tau(s)=c(s-s_0)+o(|s-s_0|)$ while $\tau^2\in C^{r-3}$. Then $|\tau|$ is \textbf{not differentiable} at
$s_0$: the difference quotient $|\tau(s)|/(s-s_0)$ has right limit $|c|$ and left limit $-|c|$, which
differ. Square roots do not preserve $C^k$ regularity of the same order at the boundary of the
non-negative cone, so a $C^{r-3}$ statement cannot be transported unconditionally from $\tau^2$ to
$|\tau|$. Sections~\ref{sec:dichotomy} and \ref{sec:relation} are therefore stated for $\tau^2$; Lemma~\ref{lem:8.11} supplies what is available
for $|\tau|$, namely the $C^0$ statement.
\end{remark}

\begin{lemma}[$D$ is invariant under independent $E(3)$ actions]
\label{lem:8.5}
For all $g_1,g_2\in E(3)$,
$D(g_1\alpha,g_2\beta)=D(\alpha,\beta)$.
\end{lemma}

\begin{proof}
Isometries preserve $d_H$, so
$d_H(g\,g_1\operatorname{Im}\alpha,g_2\operatorname{Im}\beta)
=d_H(g_2^{-1}g\,g_1\operatorname{Im}\alpha,\operatorname{Im}\beta)$; as $g$ ranges over $E(3)$ so
does $h=g_2^{-1}gg_1$, and taking infima gives the claim.
\end{proof}

\begin{lemma}[continuity of $N_{L_0}$ and $\mathcal D$]
\label{lem:8.6}
Let $r\ge4$. Then
$N_{L_0}:\mathcal E^r\to C^r(\mathbb R/L_0\mathbb Z,\mathbb R^3)$ and
$\mathcal D:\mathcal E^r\to C^{r-3}$ are continuous.
\end{lemma}

\begin{proof}
To avoid a domain varying with $\gamma$, work first on the \textbf{fixed} parameter circle
$\mathbb R/\mathbb Z$. (1) $\gamma\mapsto|\gamma'|$ is continuous from $C^r$ to $C^{r-1}$ and bounded
below on $\mathcal E^r$; put $L(\gamma)=\int|\gamma'|$, continuous, and
$\sigma_\gamma(t)=\frac1{L(\gamma)}\int_0^t|\gamma'|\in\mathbb R/\mathbb Z$, so that
$\sigma_\gamma'>0$, $\sigma_\gamma$ is a $C^r$ diffeomorphism of the fixed circle, and
$\gamma\mapsto\sigma_\gamma$ is continuous. (2) \textbf{Inversion is continuous} on
$\{\varphi\in C^r(\mathbb R/\mathbb Z):\varphi'>0\}$: from $(\varphi^{-1})'=1/(\varphi'\circ\varphi^{-1})$
one gets recursively that $(\varphi^{-1})^{(k)}$ is a polynomial in $\varphi',\dots,\varphi^{(k)}$
and $\varphi^{-1}$ divided by $(\varphi'\circ\varphi^{-1})^{2k-1}$, with denominator uniformly
bounded below on the compact circle. Only continuity is claimed here, not differentiability or norm
preservation. (3) Composition $(\gamma,\psi)\mapsto\gamma\circ\psi$ is continuous from
$C^r\times C^r$ to $C^r$. (4) With $\widetilde\gamma=\gamma\circ\sigma_\gamma^{-1}$ pull back:
$c(u)=\frac{L_0}{L(\gamma)}\widetilde\gamma(u/L_0)-\text{centroid}$, each step continuous.
(5) $\kappa_c,\tau_c$ are rational expressions in the first three derivatives of $c$ with
denominators bounded below, so $c\mapsto(\kappa_c,\tau_c^2)$ is continuous from $C^r$ to $C^{r-3}$.
\end{proof}

\subsection{No global modulus of continuity}
\label{sec:8.2}

The counterexample is built from an exact ambiguous pair (Example~\ref{ex:5.14}): two curves with \emph{identical}
data and distinct non-mirror knot types. Such a pair does not itself contradict anything --- its
members have $c(\tau)=2$ and lie outside $\mathcal G^r$ --- but perturbing each member into
$\mathcal G^r$ produces genuinely rigid curves whose data are merely close. The knot types are
preserved because they are locally constant, and the orbit separation is preserved because it is
positive at the pair, which is the content of the following lemma.

\begin{lemma}[positive orbit separation of an exact ambiguous pair]
\label{lem:8.7}
Let $\gamma_\pm$ be the exact
ambiguous pair of Example~\ref{ex:5.14}. Then $D(\gamma_+,\gamma_-)>0$.
\end{lemma}

\begin{proof}
Suppose $D=0$ and take $g_k=(Q_k,a_k)\in E(3)$ with
$d_H(g_k\operatorname{Im}\gamma_-,\operatorname{Im}\gamma_+)\to0$. If
$\operatorname{Im}\gamma_\pm\subset B(0,R)$ and $|a_k|>3R+1$ then $d_H>1$, so the $|a_k|$ are
bounded; $O(3)$ is compact, so a subsequence converges, $g_k\to g$ uniformly on compact sets, and
$d_H(g\operatorname{Im}\gamma_-,\operatorname{Im}\gamma_+)=0$. Both sets being compact,
$g(\operatorname{Im}\gamma_-)=\operatorname{Im}\gamma_+$. If $\det g=+1$ then $g\in SE(3)$ is path
connected to the identity, giving an ambient isotopy and $[\gamma_+]=[\gamma_-]$, contradicting
Example~\ref{ex:5.14}; if $\det g=-1$ then $[\gamma_+]=\mathrm{mirror}[\gamma_-]$, again a contradiction.
\end{proof}

The infimum runs over $E(3)$ \textbf{including reflections}, so the mirror orbit is already accounted for.
This is precisely why the ``non-mirror'' clause of Assumption (H) is indispensable here: with
distinctness alone the lemma would fail.

\begin{theorem}[near-collision]
\label{thm:8.8}
Let $r\ge4$. Under (H) with $m=2$ take the exact ambiguous pair
$\gamma_\pm$, of common length $L_0$, in a common arclength label, with
$\mathcal D(\gamma_+)=\mathcal D(\gamma_-)$ exactly, $[\gamma_+]=K_1\#K_2$ and
$[\gamma_-]=K_1\#\overline{K_2}$. Put $\bar\gamma_\pm:=N_{L_0}(\gamma_\pm)$ and
$3\delta_0:=D(\gamma_+,\gamma_-)>0$. Then there are $\alpha_n,\beta_n\in\mathcal G^r\cap C^\infty$,
defined on the common domain $\mathbb R/L_0\mathbb Z$, unit speed, of length $L_0$ and centroid $0$,
such that for all large $n$:
\begin{enumerate}
\item $[\alpha_n]=K_1\#K_2$ and $[\beta_n]=K_1\#\overline{K_2}$: \textbf{fixed, distinct and non-mirror};
\item all torsion zeros are simple and finite in number, so each curve is determined up to $E(3)$ by its
   own unsigned datum (Corollary~\ref{cor:6.10}); being moreover $C^\infty$, they satisfy
   $Z_\infty=\varnothing$;
\item $\bigl\|(\kappa_{\alpha_n},\tau_{\alpha_n}^2)-(\kappa_{\beta_n},\tau_{\beta_n}^2)\bigr\|_{C^{r-3}}\to0$;
\item $D(\alpha_n,\beta_n)\ge\delta_0>0$.
\end{enumerate}
\end{theorem}

\begin{proof}
Take $\varepsilon_n\downarrow0$.

\textbf{(1) Generic approximation.} By Proposition~\ref{prop:6.7} there is $a_n\in\mathcal G^r\cap C^\infty$ with
$\|a_n-\gamma_+\|_{C^r}<\varepsilon_n$, and similarly $b_n$ for $\gamma_-$. Density is used once for each branch separately.

\textbf{(2) Knot types are preserved.} $\mathcal E^r$ is open, and for small $\varepsilon_n$ the
straight-line homotopy $(1-t)\gamma_++ta_n$ consists of embeddings, embeddings being $C^1$-open on a
compact domain, so the Isotopy Extension Theorem gives $[a_n]=[\gamma_+]$; similarly
$[b_n]=[\gamma_-]$.

\textbf{(3) Equalising the length.} Put $\alpha_n:=N_{L_0}(a_n)$ and $\beta_n:=N_{L_0}(b_n)$. Of the three
steps of $N_{L_0}$, the first preserves $\mathcal G^r$ and the knot type by Lemma~\ref{lem:6.9}; the homothety
sends $\tau\mapsto\tau/\lambda$ and the parameter to $\lambda s$, both positive factors, so zeros and
their simplicity survive, and a homothety is homotopic to the identity so the knot type is unchanged;
the translation is an isometry. Hence $\alpha_n,\beta_n\in\mathcal G^r\cap C^\infty$ have length
exactly $L_0$ and live on \textbf{the same} $\mathbb R/L_0\mathbb Z$, which gives 1 and 2.

Independent generic perturbations change the length, so the two arclength domains differ and the data cannot be subtracted pointwise at all; rescaling restores the exact
common length $L_0$ at the cost of a homothety factor $\lambda_n$ with
$|\lambda_n-1|\le C\varepsilon_n$ for a constant $C$ depending only on $\gamma_\pm$, hence an error of
the same size in the data, which is harmless.

\textbf{(4) The data become close.} $N_{L_0}$ is idempotent on $\alpha_n,\beta_n$, so
$\mathcal D(\alpha_n)=(\kappa_{\alpha_n},\tau_{\alpha_n}^2)$. By Lemma~\ref{lem:8.6},
$\|\mathcal D(a_n)-\mathcal D(\gamma_+)\|_{C^{r-3}}\to0$ and
$\|\mathcal D(b_n)-\mathcal D(\gamma_-)\|_{C^{r-3}}\to0$, while
$\mathcal D(\gamma_+)=\mathcal D(\gamma_-)$ \textbf{exactly}; the triangle inequality gives 3.

\textbf{(5) The separation persists.} By Lemma~\ref{lem:8.6},
$\|\alpha_n-\bar\gamma_+\|_{C^0}=\|N_{L_0}(a_n)-N_{L_0}(\gamma_+)\|_{C^0}\to0$, hence
$d_H(\operatorname{Im}\alpha_n,\operatorname{Im}\bar\gamma_+)\to0$, and similarly for $\beta_n$; and
$D(\bar\gamma_+,\bar\gamma_-)=D(\gamma_+,\gamma_-)=3\delta_0$ by Lemma~\ref{lem:8.5}. For any $g\in E(3)$,
\[
d_H(g\operatorname{Im}\alpha_n,\operatorname{Im}\beta_n)\ \ge\
d_H(g\operatorname{Im}\bar\gamma_+,\operatorname{Im}\bar\gamma_-)
-d_H(\operatorname{Im}\alpha_n,\operatorname{Im}\bar\gamma_+)
-d_H(\operatorname{Im}\bar\gamma_-,\operatorname{Im}\beta_n),
\]
and taking the infimum over $g$ yields $D(\alpha_n,\beta_n)\ge3\delta_0-o(1)\ge\delta_0$.
\end{proof}

The limits in step (5) must be $\bar\gamma_\pm$ and not $\gamma_\pm$: $N_{L_0}$ contains the centering
step, so $N_{L_0}(a_n)\to N_{L_0}(\gamma_+)=\bar\gamma_+$, and using $\gamma_\pm$ would ignore a fixed
translation that does not shrink with $n$. Lemma~\ref{lem:8.5} is what guarantees that this normalisation does
not change the orbit distance.

\begin{corollary}[no global modulus]
\label{cor:8.9}
Let $r\ge4$. There is \textbf{no} function
$\omega:[0,\infty)\to[0,\infty]$ with $\lim_{x\to0^+}\omega(x)=0$ such that
\[
\bar D(\alpha,\beta)\ \le\ \omega\Bigl(\bigl\|\mathcal D(\alpha)-\mathcal D(\beta)\bigr\|_{C^{r-3}(\mathbb R/L_0\mathbb Z)}\Bigr)
\qquad\text{for all }\alpha,\beta\in\mathcal G^r .
\]
\end{corollary}

\begin{proof}
The curves of Theorem~\ref{thm:8.8} satisfy the idempotence conditions of $N_{L_0}$, so
$\bar D(\alpha_n,\beta_n)=D(\alpha_n,\beta_n)\ge\delta_0$ while the right-hand side tends to
$\omega(0^+)=0$.
\end{proof}

Both sides must be normalised: if the data side used
$\mathcal D$ while the separation side used the unnormalised $D$, then $\beta=2\alpha$ would give
$\mathcal D(\alpha)=\mathcal D(\beta)$ and $D(\alpha,\beta)>0$, making the statement true for the
trivial reason of a change of scale. Equivalently one may restrict all quantifiers to the normalised
slice $\mathcal G^r_{L_0}$, where $\mathcal D$ and $\bar D$ reduce to $(\kappa,\tau^2)$ and $D$.

\begin{corollary}
\label{cor:8.10}
For every $\varepsilon>0$ there are $\alpha,\beta\in\mathcal G^r_{L_0}$ with
$\|(\kappa_\alpha,\tau_\alpha^2)-(\kappa_\beta,\tau_\beta^2)\|_{C^{r-3}}<\varepsilon$ while
$[\alpha]\neq[\beta]$ and $[\alpha]\neq\mathrm{mirror}[\beta]$. \hfill$\square$
\end{corollary}

\begin{lemma}[the $|\tau|$ version, in $C^0$]
\label{lem:8.11}
For all real $a,b$ one has
$\bigl||a|-|b|\bigr|\le\sqrt{|a^2-b^2|}$, since
$|a^2-b^2|=\bigl||a|-|b|\bigr|\cdot(|a|+|b|)\ge\bigl||a|-|b|\bigr|^2$. Consequently the sequences of
Theorem~\ref{thm:8.8} also satisfy
$\bigl\||\tau_{\alpha_n}|-|\tau_{\beta_n}|\bigr\|_{C^0}
\le\bigl\|\tau_{\alpha_n}^2-\tau_{\beta_n}^2\bigr\|_{C^0}^{1/2}\to0$. \hfill$\square$
\end{lemma}

By Remark~\ref{rem:8.4} this is available at the $C^0$ level and there is no $C^{r-3}$ analogue for $|\tau|$.

\begin{remark}[the counterexample must escape, and where]
\label{rem:8.12}
Provided the normalised quotient space
and the data space are metric, a continuous injective map restricted to a compact set has a uniformly
continuous inverse. Hence \textbf{any} counterexample to a uniform modulus necessarily exploits
non-compactness, and Theorem~\ref{thm:8.8} cannot be strengthened to compact subsets --- the escape is required
by the conclusion. The torsion of
$\bar\gamma_\pm$ vanishes identically on the planar collars, so $Z_\infty\neq\varnothing$ and
$\bar\gamma_\pm\notin\mathcal G^r$: the sequences run out along the \textbf{torsion-degenerate locus}
$\{\Delta=0\}$. Two consequences follow. First, quantitative stability on a uniformly non-degenerate
subset is in no conflict with Corollary~\ref{cor:8.9}, and \S\S\ref{sec:8.3}--\ref{sec:8.6} turn that
observation into theorems. Second, one may ask whether a near-collision can avoid that locus;
Proposition~\ref{prop:8.23} shows that, under uniform $C^5$ and curvature bounds, it cannot.
Note that $\Delta(\alpha)=0$ says only that some zero of the torsion fails to be simple: that zero may
be double, of higher finite order, or of infinite order, and only in the last case is the curve
locally planar. For the sequences constructed here the degeneration does occur along genuinely planar
arcs.
\end{remark}

\subsection{The non-degenerate strata}
\label{sec:8.3}

From here on $r=5$. The reason is visible in Lemma~\ref{lem:8.14}: the torsion is a ratio whose second
derivative involves $\alpha^{(5)}$ and no more, so $C^5$ is exactly the regularity at which
$\tau\in C^{r-3}=C^2$, which is what Theorem~\ref{thm:7.1} requires of its arguments.

\begin{definition}[non-degeneracy and the normalised strata]
\label{def:8.13}
For $\alpha\in\mathcal G^5_{L_0}$ put
\[
\Delta(\alpha):=\inf_{s}\sqrt{\tau_\alpha(s)^2+\tau_\alpha'(s)^2},
\]
\[
\mathcal K^5_{\delta,M}(L_0):=\bigl\{\alpha\in\mathcal G^5_{L_0}:\ \|\alpha\|_{C^5}\le M,\
\kappa_\alpha\ge M^{-1},\ \Delta(\alpha)\ge\delta\bigr\}.
\]
\end{definition}

The condition $\Delta(\alpha)\ge\delta$ is the pointwise inequality
$\tau_\alpha^2+(\tau_\alpha')^2\ge\delta^2$. Unit speed gives $\|\alpha'\|_{C^0}=1$, hence
$\|\alpha\|_{C^5}\ge1$, so the stratum is empty for $M<1$; and Fenchel's theorem gives
$\int_0^{L_0}\kappa\ge2\pi$, whence also $M\ge2\pi/L_0$.

The condition is a \textbf{quantitative} non-degeneracy: it is weaker than ``the torsion never vanishes'', since a sign change is permitted as
long as $|\tau'|\ge\delta$ there, and stronger in demanding a uniform bound. Frenet non-degeneracy
has an independent qualitative life in the literature: Shapiro \cite{ref15} studies a question of Agrachev on
how many times a plane circle must be traversed before arbitrarily small perturbations make its
Frenet frame nowhere degenerate --- in $\mathbb R^3$ exactly the non-vanishing of the torsion --- and
obtains turn counts depending on the topology used. No result of \cite{ref15} is used here; the reference
records that the same notion appears elsewhere as a \textbf{qualitative accessibility} question, whereas
what \S\ref{sec:dichotomy} needs is a \textbf{quantitative uniform lower bound}.

\begin{lemma}[explicit bound on $\|\tau\|_{C^2}$]
\label{lem:8.14}
Let $\alpha\in\mathcal K^5_{\delta,M}(L_0)$.
Then $\tau_\alpha\in C^2$, $M^{-1}\le\kappa_\alpha\le M$, and
\[
\|\tau_\alpha\|_{C^2}\ \le\ 4M^4+10M^8+8M^{12}\ \le\ 22M^{12}=:B(M).
\]
\end{lemma}

The computation, which also shows that no derivative beyond $\alpha^{(5)}$ occurs, is in
Appendix~\ref{sec:C.1}.

\subsection{Two structural lemmas: the sign is realised by a reflection, and Duhamel propagates
without exponential loss}
\label{sec:8.4}

The passage from Theorem~\ref{thm:7.1} to geometry needs exactly two things: a way to convert the \emph{abstract}
global sign $\varepsilon$ supplied by Theorem~\ref{thm:7.1} into an actual isometry of $\mathbb R^3$, and a way
to convert an $L^1$ bound on the Frenet coefficients into a $C^0$ bound on the curves. The first is
Lemma~\ref{lem:8.15}, and it is where reflections enter; the second is Lemma~\ref{lem:8.16}, and it is where the
orthogonality of the Frenet frame pays for itself.

\begin{lemma}[reflection flips the torsion, stays in the stratum, and aligns frame and position]
\label{lem:8.15}
Let $\beta\in\mathcal K^5_{\delta,M}(L_0)$, $Q\in O(3)$ with $\det Q=-1$, and
$\widetilde\beta=Q\beta$. Then pointwise
\[
T_{\widetilde\beta}=QT_\beta,\quad\kappa_{\widetilde\beta}=\kappa_\beta,\quad
N_{\widetilde\beta}=QN_\beta,\quad B_{\widetilde\beta}=-QB_\beta,\quad
\tau_{\widetilde\beta}=-\tau_\beta,
\]
the frame matrix $F=(T|N|B)$ satisfies $F_{\widetilde\beta}=QF_\beta J\in SO(3)$ with
$J=\operatorname{diag}(1,1,-1)$, and $\widetilde\beta\in\mathcal K^5_{\delta,M}(L_0)$ with
$\mathcal D(\widetilde\beta)=\mathcal D(\beta)$. Moreover, for any
$\alpha\in\mathcal K^5_{\delta,M}(L_0)$ and $\varepsilon\in\{\pm1\}$, taking $Q_\varepsilon=I$ or
$\operatorname{diag}(-1,-1,-1)$, the latter of determinant $-1$,
$Q_0:=F_\alpha(0)F_{Q_\varepsilon\beta}(0)^{\mathsf T}\in SO(3)$ and a translation, one obtains
$\widehat\beta=g_0\beta$ with $g_0\in E(3)$ whose linear part has determinant $\varepsilon$ and
\[
\widehat\beta(0)=\alpha(0),\qquad F_{\widehat\beta}(0)=F_\alpha(0),\qquad
\kappa_{\widehat\beta}=\kappa_\beta,\qquad \tau_{\widehat\beta}=\varepsilon\tau_\beta .
\]
\end{lemma}

\begin{proof}
The first three identities follow from $\widetilde\beta^{(j)}=Q\beta^{(j)}$ and the fact
that $Q$ is an isometry; $(Qa)\times(Qb)=\det(Q)\,Q(a\times b)$ gives $B_{\widetilde\beta}=-QB_\beta$;
and $B_{\widetilde\beta}'=-QB_\beta'=\tau_\beta QN_\beta=\tau_\beta N_{\widetilde\beta}$, compared
with $B'=-\tau N$, gives $\tau_{\widetilde\beta}=-\tau_\beta$ --- this is Lemma~\ref{lem:2.7}. Also
$\det F_{\widetilde\beta}=(-1)(1)(-1)=+1$. The stratum is preserved: a constant isometry preserves
the norms of all derivatives, the centroid, unit speed, length, parameter origin and orientation;
$\kappa$ is unchanged; $\tau_{\widetilde\beta}=-\tau_\beta$ gives
$\tau_{\widetilde\beta}'=-\tau_\beta'$, so $\tau^2+(\tau')^2$ and hence $\Delta$ are unchanged;
$D_{\widetilde\beta}=-D_\beta$, so the zero set and the simplicity of its points are unchanged; and
$\tau^2,\kappa$ unchanged gives $\mathcal D$ unchanged. For the alignment, $F_\alpha(0)$ and
$F_{Q_\varepsilon\beta}(0)$ lie in $SO(3)$, so $Q_0\in SO(3)$, and neither the action of $SO(3)$ nor
a translation changes $\kappa$ or $\tau$.
\end{proof}

Only a reflection can perform this: the factor $\det Q$ is the sole source of the sign flip. This is
Lemma~\ref{lem:2.7}, used constructively.

\begin{lemma}[Duhamel estimate for the Frenet system; no exponential factor]
\label{lem:8.16}
Let
$F_1,F_2:[0,L_0]\to SO(3)$ be $C^1$ with $F_i'=F_i\Omega_i$, for $\Omega$ as in Lemma~\ref{lem:2.10}, and
$F_1(0)=F_2(0)$. Write $E=F_1-F_2$ and $\Delta\Omega=\Omega_1-\Omega_2$. Then
\[
E(s)=\int_0^sF_2(t)\,\Delta\Omega(t)\,F_1(t)^{\mathsf T}F_1(s)\,dt,
\]
\[
\max_{s}\|E(s)\|_F\ \le\ \int_0^{L_0}\|\Delta\Omega\|_F\,dt\ \le\
\sqrt2\bigl(\|\Delta\kappa\|_{L^1}+\|\Delta\tau\|_{L^1}\bigr).
\]
\end{lemma}

\begin{proof}
We have $E(0)=0$ and $E'=E\Omega_1+F_2\Delta\Omega$. The homogeneous fundamental solution is
$F_1$, so $Z(s):=\int_0^sF_2\Delta\Omega F_1^{-1}$ makes $ZF_1$ satisfy the same equation with the
same zero initial value; the coefficient $\Omega_1$ is continuous, since $\kappa\in C^3$ and
$\tau\in C^2$, so uniqueness gives $E=ZF_1$, which is the stated formula because $F_1^{-1}=F_1^{\mathsf T}$.
The Frobenius norm is invariant under multiplication by orthogonal matrices \textbf{on either side}, and
the integrand is $F_2(t)\,\Delta\Omega(t)\,\bigl(F_1(t)^{\mathsf T}F_1(s)\bigr)$ with both outer
factors in $SO(3)$; hence its Frobenius norm is exactly
$\|\Delta\Omega(t)\|_F=\sqrt{2(\Delta\kappa)^2+2(\Delta\tau)^2}\le\sqrt2(|\Delta\kappa|+|\Delta\tau|)$.
\end{proof}

Because $\Omega$ is antisymmetric the Duhamel kernel is an isometry, so no Grönwall factor
$e^{cL_0}$ appears (\S\ref{sec:1.3}(iii)), and the right-hand side involves only the \textbf{$L^1$}
norm of the coefficient difference --- exactly the norm that Theorem~\ref{thm:7.1} delivers.

\subsection{Uniform conditional stability on the strata}
\label{sec:8.5}

\begin{theorem}[uniform log-Lipschitz stability on $\mathcal K^5_{\delta,M}(L_0)$]
\label{thm:8.17}
Fix $L_0>0$,
$M\ge1$ and $\delta>0$ with $\mathcal K^5_{\delta,M}(L_0)\ne\varnothing$. Put
$\mathsf B:=2B(M)=44M^{12}$ and
\[
C_{\mathrm{geo}}(M,\delta,L_0):=\sqrt2\,L_0\Bigl(L_0+C_A(\mathsf B,\delta,L_0)\Bigr).
\]
Then for all $\alpha,\beta\in\mathcal K^5_{\delta,M}(L_0)$, the data being compared pointwise in the
\textbf{given common arclength label}, and with
$\Theta:=\bigl\|(\kappa_\alpha,\tau_\alpha^2)-(\kappa_\beta,\tau_\beta^2)\bigr\|_{C^0}$, there are
$\varepsilon\in\{\pm1\}$ and $g_0\in E(3)$, the linear part of $g_0$ having determinant $\varepsilon$,
such that $\widehat\beta:=g_0\beta$ satisfies
\[
\|\alpha-\widehat\beta\|_{C^0}\ \le\ C_{\mathrm{geo}}(M,\delta,L_0)\,\rho(\Theta).
\]
Consequently, in the notation of Definition~\ref{def:8.1} and Lemma~\ref{lem:8.2},
\[
D(\alpha,\beta)\ \le\ d_{\mathrm{lab}}(\alpha,\beta)\ \le\ C_{\mathrm{geo}}(M,\delta,L_0)\,\rho(\Theta).
\]
\end{theorem}

The estimate is therefore proved first for the \textbf{parametrised} curves in the common label,
modulo $E(3)$ only, and the label-free Hausdorff bound is a corollary of it. Explicitly, for
$M\ge1$ and $0<\delta\le1$ the constant obeys the non-asymptotic bound
\begin{equation}\label{eq:8.1}\tag{8.1}
C_{\mathrm{geo}}(M,\delta,L_0)\ \le\ 1.6\times10^{8}\;L_0^2\,M^{36}\,\delta^{-4},
\end{equation}
by \eqref{eq:7.3} applied with $\mathsf B=44M^{12}\ge1$; in asymptotic notation, for each fixed
$M\ge1$ and $L_0>0$, $C_{\mathrm{geo}}(M,\delta,L_0)=O_{M,L_0}(\delta^{-4})$ as $\delta\downarrow0$.

\begin{proof}
The proof is four moves: check the one-dimensional hypotheses, extract the sign, realise it
geometrically, and integrate twice.

\textbf{(1) The one-dimensional hypotheses hold.} By Lemma~\ref{lem:8.14}, $f:=\tau_\alpha$ and $g:=\tau_\beta$ lie
in $C^2$ and are defined on \textbf{one and the same} $\mathbb R/L_0\mathbb Z$ --- this is what the
normalisation built into the stratum provides --- with $\|f\|_{C^2}+\|g\|_{C^2}\le2B(M)=\mathsf B$,
pointwise $f^2+(f')^2\ge\delta^2$ and $g^2+(g')^2\ge\delta^2$, and
$\delta\le\Delta(\alpha)\le\|\tau_\alpha\|_{C^0}+\|\tau_\alpha'\|_{C^0}\le B(M)\le\mathsf B$.

\textbf{(2) The global sign.} Theorem~\ref{thm:7.1} supplies a global $\varepsilon\in\{\pm1\}$ with
$\|\tau_\alpha-\varepsilon\tau_\beta\|_{L^1}\le C_A(\mathsf B,\delta,L_0)\,\rho(\eta)$, where
$\eta=\|\tau_\alpha^2-\tau_\beta^2\|_{C^0}$.

\textbf{(3) Realising that sign by a reflection, and aligning.} Lemma~\ref{lem:8.15} gives $\widehat\beta=g_0\beta$
with $\kappa_{\widehat\beta}=\kappa_\beta$, $\tau_{\widehat\beta}=\varepsilon\tau_\beta$,
$\widehat\beta(0)=\alpha(0)$ and $F_{\widehat\beta}(0)=F_\alpha(0)$.

\textbf{(4) Duhamel, then integration of the position.} The two frames satisfy their Frenet systems with
the same initial value, so Lemma~\ref{lem:8.16} gives
\[
\max_s\|F_\alpha-F_{\widehat\beta}\|_F\le\sqrt2\bigl(\|\Delta\kappa\|_{L^1}+\|\Delta\tau\|_{L^1}\bigr),
\qquad \Delta\kappa=\kappa_\alpha-\kappa_\beta,\quad\Delta\tau=\tau_\alpha-\varepsilon\tau_\beta .
\]
The first column of $E$ is $T_\alpha-T_{\widehat\beta}$ and
$\alpha(s)-\widehat\beta(s)=\int_0^s(T_\alpha-T_{\widehat\beta})$, so
\[
\|\alpha-\widehat\beta\|_{C^0}\le L_0\max_s\|E\|_F\le\sqrt2L_0\bigl(\|\Delta\kappa\|_{L^1}+\|\Delta\tau\|_{L^1}\bigr).
\]
\textbf{(5) Both terms in terms of $\rho(\Theta)$.} We have
$\|\Delta\kappa\|_{L^1}\le L_0\|\kappa_\alpha-\kappa_\beta\|_{C^0}\le L_0\Theta\le L_0\rho(\Theta)$,
using $t\le\rho(t)$, while $\eta\le\Theta$ and the monotonicity of $\rho$ give
$\|\Delta\tau\|_{L^1}\le C_A\rho(\eta)\le C_A\rho(\Theta)$.

\textbf{(6) Passing to the two orbit distances.} Since $\widehat\beta=g_0\beta$ with $g_0\in E(3)$, the
bound of step (5) is a bound on one admissible competitor in the infimum defining
$d_{\mathrm{lab}}$, whence $d_{\mathrm{lab}}(\alpha,\beta)\le\|\alpha-\widehat\beta\|_{C^0}$;
Lemma~\ref{lem:8.2} then gives $D(\alpha,\beta)\le d_{\mathrm{lab}}(\alpha,\beta)$.
\end{proof}

The only $\delta$-dependence enters through $C_A$, that is, through the one-dimensional theorem: the
geometric constant is an explicit product of the one-dimensional constant with two factors of $L_0$
from the two integrations.

\begin{corollary}[$\mathcal D$ determines a curve of the stratum up to $E(3)$]
\label{cor:8.18}
If $\alpha,\beta\in\mathcal K^5_{\delta,M}(L_0)$ and $\mathcal D(\alpha)=\mathcal D(\beta)$, then
$\Theta=0$ and $\rho(0)=0$, so the curve $\widehat\beta=g_0\beta$ produced by
Theorem~\ref{thm:8.17} satisfies $\|\alpha-\widehat\beta\|_{C^0}=0$: that is, $\alpha=g_0\beta$
\textbf{as parametrised curves}, for an explicit $g_0\in E(3)$, and in particular
$d_{\mathrm{lab}}(\alpha,\beta)=D(\alpha,\beta)=0$. This is the limiting case $\Theta\to0$ of
Corollary~\ref{cor:6.10}, and shows that Theorem~\ref{thm:8.17} is a genuine quantitative refinement
of it rather than a separate statement. \hfill$\square$
\end{corollary}

\subsection{Local stability at a fixed generic curve}
\label{sec:8.6}

The stratum version is uniform but requires membership of a stratum. The following localisation
removes that requirement at a fixed curve: every curve with simple torsion zeros has a \textbf{relative}
$C^5$ neighbourhood inside the normalised slice $\mathcal G^5_{L_0}$ on which the same modulus holds,
because $\Delta$ is bounded below near it.

\begin{lemma}
\label{lem:8.19}
$\Delta(\alpha)>0$ for every $\alpha\in\mathcal G^5_{L_0}$.
\end{lemma}

\begin{proof}
The function $\phi:=\tau_\alpha^2+(\tau_\alpha')^2$ is continuous on the compact circle,
since $\tau_\alpha\in C^2$, and therefore attains its minimum. Were that minimum $0$, we would have
$\tau_\alpha=\tau_\alpha'=0$ at some point. But $\tau_\alpha=D/u$ with $u=\kappa^2>0$ and $u\in C^3$,
so $Z(\tau_\alpha)=Z(D)$, and at a zero $\tau_\alpha'=D'/u$; hence $\tau_\alpha'=0$ iff $D'=0$, and
the point would be a non-simple zero of $D_\alpha$, contradicting $\alpha\in\mathcal G^5$.
\end{proof}

\begin{theorem}[local log-Lipschitz stability]
\label{thm:8.20}
Fix $\alpha\in\mathcal G^5_{L_0}$ and put
$\kappa_0:=\min\kappa_\alpha>0$, $m:=\|\alpha\|_{C^5}+\frac{\kappa_0}2$ (note the first power of
$\kappa_0$ here, against the square in $\mu$), $\mu:=\frac{\kappa_0^2}4$,
$\Lambda:=\frac{28m^8}{\mu^4}=\frac{7168\,m^8}{\kappa_0^8}$, and
\[
\rho_0(\alpha):=\min\Bigl\{\frac{\kappa_0}2,\ \frac{\Delta(\alpha)}{2\Lambda}\Bigr\}>0,\qquad
\delta_\alpha:=\frac{\Delta(\alpha)}2,\qquad M_\alpha:=\max\Bigl\{m,\ \frac2{\kappa_0}\Bigr\}.
\]
Then $U_\alpha:=\{\beta\in\mathcal G^5_{L_0}:\|\beta-\alpha\|_{C^5}<\rho_0(\alpha)\}$, a relative ball
in the normalised slice containing $\alpha$, satisfies $U_\alpha
\subseteq\mathcal K^5_{\delta_\alpha,M_\alpha}(L_0)$, and hence, by Theorem~\ref{thm:8.17}, for all
$\beta,\gamma\in U_\alpha$
\[
D(\beta,\gamma)\ \le\ d_{\mathrm{lab}}(\beta,\gamma)\ \le\ C_{\mathrm{geo}}(M_\alpha,\delta_\alpha,L_0)\
\rho\bigl(\|\mathcal D(\beta)-\mathcal D(\gamma)\|_{C^0}\bigr).
\]
\end{theorem}

\begin{proof}
It suffices to show that $\Delta$ keeps a positive lower bound on the small ball. Let
$\epsilon:=\|\beta-\alpha\|_{C^5}\le\frac{\kappa_0}2$, the single bound from which the three
constraints are read off separately. First, $\|\beta''-\alpha''\|_{C^0}\le\epsilon$ and
$\kappa=|\gamma''|$ on unit-speed curves give $\kappa_\beta\ge\kappa_0-\epsilon\ge\frac{\kappa_0}2$,
whence $u_\beta=\kappa_\beta^2\ge\frac{\kappa_0^2}4=\mu$. Second, and independently,
$\|\beta\|_{C^5}\le\|\alpha\|_{C^5}+\epsilon\le\|\alpha\|_{C^5}+\frac{\kappa_0}2=m$; this is why $m$
carries the first power of $\kappa_0$ and not its square, and no smallness of $\kappa_0$ is needed
anywhere. Note also $m\ge\|\alpha'\|_{C^0}=1$. On
$\{u\ge\mu,\ \text{all derivatives}\le m\}$ the quantities $D,D',u,u'$ are multilinear, respectively
quadratic, in the derivatives, so
\[
|D_\beta-D_\alpha|,\ |D_\beta'-D_\alpha'|\le3m^2\epsilon,\qquad
|u_\beta-u_\alpha|\le2m\epsilon,\qquad|u_\beta'-u_\alpha'|\le4m\epsilon;
\]
substituting into $\tau=\frac Du$ and $\tau'=\frac{D'}u-\frac{Du'}{u^2}$ and enlarging term by term,
using $|D|,|D'|\le m^2$, $|u'|\le2m^2$, $|u_\beta^2-u_\alpha^2|\le2m^2|u_\beta-u_\alpha|$,
$\mu\le m^2$ and $m\ge1$, gives $\|\tau_\beta-\tau_\alpha\|_{C^1}\le\Lambda\epsilon$. Comparing
pointwise by $\bigl|\sqrt{a^2+b^2}-\sqrt{c^2+d^2}\bigr|\le|a-c|+|b-d|$ and taking the infimum,
$\Delta(\beta)\ge\Delta(\alpha)-\Lambda\epsilon\ge\frac{\Delta(\alpha)}2=\delta_\alpha$. The other
two constraints hold because $\|\beta\|_{C^5}\le m\le M_\alpha$ and
$\kappa_\beta\ge\frac{\kappa_0}2\ge M_\alpha^{-1}$; and $\alpha$ itself satisfies all three.
\end{proof}

Theorem~\ref{thm:8.20} is the precise sense in which the inverse map is continuous at every fixed generic
curve, with modulus $\rho$. In particular the instability of \S\ref{sec:8.2} is not the discontinuity of the
inverse at any single curve --- that would contradict this theorem --- but the absence of a modulus valid
simultaneously at all of them.

\subsection{Divergence of the uniform constants, and forced torsion degeneracy}
\label{sec:8.7}

The two preceding subsections and \S\ref{sec:8.2} are compatible for any number of trivial reasons. The
following theorem shows what actually reconciles them: the constants of
Theorem~\ref{thm:8.17} must blow up as $\delta\downarrow0$, and no uniformly bounded near-collision
sequence can keep $\Delta$ away from $0$.

\begin{definition}[the optimal constant]
\label{def:8.21}
\[
C_*(\delta,M,L_0):=\sup\Bigl\{\frac{d_{\mathrm{lab}}(\alpha,\beta)}{\rho\bigl(\|\mathcal D(\alpha)-\mathcal D(\beta)\|_{C^0}\bigr)}:
\ \alpha,\beta\in\mathcal K^5_{\delta,M}(L_0),\ \mathcal D(\alpha)\ne\mathcal D(\beta)\Bigr\}\in[0,+\infty],
\]
with the convention $\sup\varnothing:=0$. By Theorem~\ref{thm:8.17}, $C_*\le C_{\mathrm{geo}}<\infty$ for every
$\delta>0$, so $C_*$ is the \textbf{least admissible constant} for an inequality of the form of
Theorem~\ref{thm:8.17} on that stratum; the pairs excluded by $\mathcal D(\alpha)=\mathcal D(\beta)$
have $d_{\mathrm{lab}}=0$ by Corollary~\ref{cor:8.18}. Since $\mathcal K^5_{\delta,M}$ decreases in $\delta$ and increases in $M$, the
quantity $C_*$ is non-increasing in $\delta$ and non-decreasing in $M$; consequently
$\lim_{\delta\downarrow0}C_*(\delta,M,L_0)$ exists in $[0,+\infty]$ as a monotone limit.
\end{definition}

\begin{theorem}[divergence of the uniform stability constants as $\delta\downarrow0$]
\label{thm:8.22}
Assume (H) with $m=2$. \textbf{First fix} an exact ambiguous
pair $(\gamma_+,\gamma_-)$ satisfying (H) with $m=2$, jointly normalised to the common length $L_0$
and the common arclength label --- that is, the input of Theorem~\ref{thm:8.8}. Then there exists
\[
M_0=M_0(\gamma_+,\gamma_-,L_0)<\infty
\]
such that for \textbf{every fixed} $M\ge M_0$,
\[
\lim_{\delta\downarrow0}C_*(\delta,M,L_0)=+\infty .
\]
\end{theorem}

\begin{proof}
Take the sequences $\alpha_n,\beta_n\in\mathcal G^5\cap C^\infty$ produced by Theorem~\ref{thm:8.8}
with $r=5$; the construction of $N_{L_0}$ makes them unit speed of length $L_0$ with centroid $0$ and
parameter origin at the arclength origin, so $\alpha_n,\beta_n\in\mathcal G^5_{L_0}$, and
\[
\Theta_n:=\|\mathcal D(\alpha_n)-\mathcal D(\beta_n)\|_{C^0}\le\|\cdot\|_{C^2}\to0,\qquad
D(\alpha_n,\beta_n)\ge\delta_0>0\quad(3\delta_0=D(\gamma_+,\gamma_-)).
\]
\textbf{(i) A uniform $C^5$ bound.} By Lemma~\ref{lem:8.6} and $\|a_n-\gamma_+\|_{C^5}\to0$ we get
$\|\alpha_n-\bar\gamma_+\|_{C^5}\to0$, where $\bar\gamma_\pm=N_{L_0}(\gamma_\pm)$; since
$\gamma_\pm$ are already unit speed of length $L_0$ with arclength origin $\gamma_\pm(0)$, the map
$N_{L_0}$ acts on them merely as the translation to centroid $0$, so $\bar\gamma_\pm\in C^\infty$ and
$\|\bar\gamma_\pm\|_{C^5}<\infty$; hence $\sup_n\|\alpha_n\|_{C^5}<\infty$, and likewise for
$\beta_n$.

\textbf{(ii) A uniform lower bound on the curvature.} $\gamma_\pm$ have $\kappa>0$ everywhere, hence a
positive minimum on the compact circle, and translation does not change $\kappa$; by the $C^2$
convergence from (i), for all large $n$ we have $\kappa_{\alpha_n},\kappa_{\beta_n}\ge\kappa_\flat>0$
with $\kappa_\flat=\frac12\min\{\min\kappa_{\bar\gamma_+},\min\kappa_{\bar\gamma_-}\}$. Discard
finitely many $n$.

\textbf{(iii) $\Delta_n>0$ for every $n$.} By Lemma~\ref{lem:8.19},
$\Delta_n:=\min\{\Delta(\alpha_n),\Delta(\beta_n)\}>0$.

\textbf{(iv) $M_0$.} Put
\[
M_0(\gamma_+,\gamma_-,L_0):=\max\Bigl\{\sup_n\|\alpha_n\|_{C^5},\ \sup_n\|\beta_n\|_{C^5},\
\kappa_\flat^{-1},\ 1\Bigr\}<\infty .
\]
It does \textbf{not depend on $\delta$}, which is what the conclusion requires; it does depend on the pair
fixed in the first step, since the supremum in (i) is controlled by $\|\bar\gamma_\pm\|_{C^5}$ and
$\kappa_\flat$ by $\min\kappa_{\bar\gamma_\pm}$, and both vary with $K_1,K_2$ and with the particular
construction. Thus $\alpha_n,\beta_n\in\mathcal K^5_{\Delta_n,M_0}(L_0)$ for every retained $n$.

\textbf{(v) $\Theta_n>0$.} If some $\Theta_n$ vanished, Theorem~\ref{thm:8.17} would give
$D(\alpha_n,\beta_n)\le C_{\mathrm{geo}}\rho(0)=0$, contradicting $D\ge\delta_0>0$.

\textbf{(vi) The ratios diverge.} By Lemma~\ref{lem:8.2},
$d_{\mathrm{lab}}(\alpha_n,\beta_n)\ge D(\alpha_n,\beta_n)\ge\delta_0$, so
\[
R_n:=\frac{d_{\mathrm{lab}}(\alpha_n,\beta_n)}{\rho(\Theta_n)}\ \ge\ \frac{\delta_0}{\rho(\Theta_n)}
\ \longrightarrow\ +\infty .
\]

\textbf{(vii) Conclusion.} Fix $n$. For \textbf{every} $\delta\in(0,\Delta_n]$ we have
$\Delta(\alpha_n),\Delta(\beta_n)\ge\delta$, so this very pair lies in
$\mathcal K^5_{\delta,M_0}(L_0)$ and is a legitimate witness at that $\delta$; hence
$C_*(\delta,M_0,L_0)\ge R_n$ for all $0<\delta\le\Delta_n$. The limit exists by monotonicity, so
$\lim_{\delta\downarrow0}C_*(\delta,M_0,L_0)\ge R_n$ for every $n$, and $R_n\to\infty$ makes that
limit $+\infty$. Monotonicity in $M$ extends the conclusion to every $M\ge M_0$.
\end{proof}

The next statement is the converse half, and it is proved for \emph{arbitrary} uniformly bounded
near-collision sequences rather than only for the ones constructed in \S\ref{sec:8.2}.

\begin{proposition}[near-collisions force torsion degeneracy]
\label{prop:8.23}
Fix $L_0>0$ and $M\ge1$. Let $\alpha_n,\beta_n\in\mathcal G^5_{L_0}$ satisfy
\[
\|\alpha_n\|_{C^5},\ \|\beta_n\|_{C^5}\le M,\qquad \kappa_{\alpha_n},\kappa_{\beta_n}\ge M^{-1},
\qquad \Theta_n:=\|\mathcal D(\alpha_n)-\mathcal D(\beta_n)\|_{C^0}\longrightarrow0,
\]
while the orbit distances stay away from $0$, say $D(\alpha_n,\beta_n)\ge\delta_0>0$ for all $n$.
Then
\[
\Delta_n:=\min\bigl\{\Delta(\alpha_n),\Delta(\beta_n)\bigr\}\ \longrightarrow\ 0 .
\]
\end{proposition}

\begin{proof}
If not, then $\Delta_{n_k}\ge\delta_1>0$ along a subsequence, so
$\alpha_{n_k},\beta_{n_k}\in\mathcal K^5_{\delta_1,M}(L_0)$ and Theorem~\ref{thm:8.17} gives
$\delta_0\le D(\alpha_{n_k},\beta_{n_k})\le C_{\mathrm{geo}}(M,\delta_1,L_0)\,\rho(\Theta_{n_k})\to0$,
a contradiction.
\end{proof}

\begin{corollary}
\label{cor:8.24}
The sequences of Theorem~\ref{thm:8.8}, taken with $r=5$ as in Theorem~\ref{thm:8.22}, necessarily
satisfy $\Delta_n\to0$.
\end{corollary}

\begin{proof}
Steps (i)--(ii) of the proof of Theorem~\ref{thm:8.22} supply the uniform $C^5$ bound and the uniform
curvature lower bound, both with the single constant $M_0$; part 3 of Theorem~\ref{thm:8.8} with
$r=5$ gives $\|\mathcal D(\alpha_n)-\mathcal D(\beta_n)\|_{C^2}\to0$ and hence $\Theta_n\to0$, and
part 4 gives $D(\alpha_n,\beta_n)\ge\delta_0>0$. Apply Proposition~\ref{prop:8.23} with $M=M_0$.
\end{proof}

This turns the observation of Remark~\ref{rem:8.12} into a necessity. There we merely checked that the
torsion of $\bar\gamma_\pm$ vanishes on the planar collars, so that the particular sequence escapes
along the torsion-degenerate locus. Proposition~\ref{prop:8.23} shows that \textbf{no near-collision
sequence obeying a uniform $C^5$ bound and a uniform lower bound on the curvature can avoid
degenerating in $\Delta$}. It says nothing about sequences violating those bounds, and it is not a
classification of instability mechanisms; it is a statement about the class in which
Theorem~\ref{thm:8.17} holds.

\begin{remark}[what the sharpness of \S\ref{sec:engine} does and does not transport]
\label{rem:8.25}
Theorem~\ref{thm:7.2} is a statement
about a pair of one-dimensional functions. Turning it into a \emph{geometric} lower bound, and thereby
proving that the logarithmic order of Theorem~\ref{thm:8.17} is optimal, would require two further ingredients,
each of independent interest and neither supplied here. First, every step of Theorem~\ref{thm:8.17} --- the
triangle inequality inside Duhamel, the passage $\int_0^s\to\int_0^{L_0}$, and
$D\le d_{\mathrm{lab}}\le\|\alpha-\widehat\beta\|_{C^0}$ --- is one-directional, so the chain cannot
simply be reversed. Second,
realising $(f,g_\epsilon)$ as the torsions of two \textbf{closed, embedded, unit-speed} $C^5$ curves of the
same length $L_0$ and with the \textbf{same} curvature requires the closing condition
$\int_0^{L_0}T=0$, three nonlinear constraints, together with embeddedness. There is also a genuine
mechanism working against such a transport: a large $\|\tau_\alpha-\varepsilon\tau_\beta\|_{L^1}$
does not force a large $D$, since a torsion difference concentrated on a short arc can have its
accumulated frame rotation cancelled later. Theorem~\ref{thm:8.17} is therefore an upper bound, and the
question of geometric sharpness is posed in \S\ref{sec:questions}.
\end{remark}

The order of quantifiers matters and may not be rearranged: fix the pair $\to$ obtain
$M_0(\gamma_+,\gamma_-,L_0)$ $\to$ fix $M\ge M_0$ $\to$ let $\delta\downarrow0$. Step (iv) shows how
$M_0$ is built, namely from the $C^5$ bound and the curvature lower bound of the sequence generated by
that pair; whether a bound determined by $L_0$ alone and valid simultaneously for all pairs satisfying
(H) exists is raised in \S\ref{sec:questions}.

%% file: sections/09-relation.tex
\section{Relation to the lifting and root-regularity literature}
\label{sec:relation}

Viewing $\tau^2$ as the coefficient datum of the monic hyperbolic polynomial $P(t)=X^2-\tau(t)^2$ of
degree $2$, the search for a smooth $\tau$ is a lifting problem over the orbit map of the basic
invariant $\sigma(x)=x^2$ for $G=\mathbb Z_2$ acting on $\mathbb R$. That problem has a substantial
literature, whose main goal is the \textbf{existence and regularity of lifts}:
Alekseevsky--Kriegl--Losik--Michor \cite{ref7}; Kriegl--Losik--Michor--Rainer \cite{ref16}, \cite{ref17};
Losik--Michor--Rainer \cite{ref18};
Losik--Rainer \cite{ref19}; Rainer \cite{ref9}; see the survey Parusiński--Rainer \cite{ref8}. This section places the present
results against that background, item by item.

\subsection{The rigidity criterion in the lifting literature, and how it compares}
\label{sec:9.1}

AKLM \cite[\S4.2]{ref7} prove, under a \textbf{normal non-flatness} condition, that any two smooth parametrisations
of the roots differ by a constant permutation. For $d=2$ with $\tau\not\equiv0$ that conclusion
coincides, over an interval base, with the rigidity criterion of Theorem~\ref{thm:3.9} --- but the hypotheses do
not coincide, and the comparison is exact.

\begin{proposition}[normal non-flatness for $d=2$]
\label{prop:9.1}
Let $J$ be a \textbf{connected} one-dimensional base
(an interval or $S^1$) and $\tau\in C^\infty(J)$. Then $P=X^2-\tau^2$ is normally non-flat \textbf{iff}
$Z_\infty(\tau)=\varnothing$ or $\tau\equiv0$.
\end{proposition}

\begin{proof}
For $d=2$ the condition is equivalent to: every $t_0\in Z_\infty(\tau)$ has a neighbourhood
on which $\tau\equiv0$. That makes $Z_\infty$ \textbf{open}, since $t_0\in Z_\infty$ gives $\tau\equiv0$
on a neighbourhood $N$ and hence all derivatives vanish at every point of $N$, so
$N\subseteq Z_\infty$; whereas $Z_\infty=\bigcap_j\{\tau^{(j)}=0\}$ is always \textbf{closed}. As $J$ is
connected and $Z_\infty$ is both open and closed, $Z_\infty=\varnothing$ or $Z_\infty=J$, the latter
meaning $\tau\equiv0$.
\end{proof}

\begin{corollary}[the criterion $c\le1$ is strictly weaker]
\label{cor:9.2}
In the connected case with
$\tau\not\equiv0$, normal non-flatness $\Rightarrow Z_\infty=\varnothing\Rightarrow c=1\Rightarrow$
the lift rigidity of Theorem~\ref{thm:3.9}. The \textbf{converse fails}: on $S^1$ the function
$g(\theta)=\exp\bigl(-1/\sin^2(\theta/2)\bigr)$, with $g(0)=0$, has $Z_\infty=\{0\}\neq\varnothing$
and $g\not\equiv0$, hence is not normally non-flat; yet $S^1\setminus\{0\}$ is connected, so $c=1$
and rigidity holds. \hfill$\square$
\end{corollary}

In particular there is \textbf{no} normally non-flat case consisting of ``several proper zero intervals'':
the endpoint of such an interval lies in $Z_\infty$ and has no neighbourhood on which $\tau\equiv0$.
It is exactly those endpoints that \S\ref{sec:flexibility} manufactures.

Two further points of comparison. Losik--Rainer \cite{ref19} explicitly restate the conclusion of \cite{ref7}, and
Rainer \cite[Lem.~5.1]{ref9} obtains the same conclusion when $E^{(\infty)}(P)=\varnothing$, citing its
source; the three references thus form a \textbf{single chain of source and restatement} rather than three
independent results. Separately, AKLM \cite[\S5.2]{ref7} observe that at a point of infinite flatness of
$x^2=f(t)^2$ the roots may change sign and that a constant permutation cannot absorb this freedom;
this is a \textbf{qualitative} precedent for the local sign freedom in $d=2$, but contains no statement
about independence across components and no count.

\begin{remark}[interval versus circle]
\label{rem:9.3}
The branch mechanism also exists on an interval, and is in
fact \textbf{easier} to trigger there: an isolated infinite-order flat point $p$ in the interior of an
interval $I$ makes $I\setminus\{p\}$ have two components, so $c=2$ and there are four lifts, whereas
$S^1\setminus\{p\}$ is connected and $c=1$. What is specific to the circle is only the purely
topological fact that a single point does not disconnect it. In particular the classification
\textbf{does not use monodromy}; it uses only the number of connected components.
\end{remark}

\subsection{Classification and counting}
\label{sec:9.2}

For $d=2$ over a connected one-dimensional base, the classification of \textbf{all} smooth lifts of a
given smooth non-negative coefficient datum is due to Bony--Colombini--Pernazza \cite[\S1]{ref20}, and $2^{c}$
is the immediate cardinality consequence of that classification. Section~\ref{sec:branch} is a self-contained
restatement of that result on $S^1$, and the attribution note opening \S\ref{sec:branch} maps its
statements to the source clause by clause.
The contribution of the present paper on this side lies in \S\S\ref{sec:rigidity}--\ref{sec:flexibility}: the translation of the sign-level
classification into rigidity theorems, fibre bounds and an exact fibre for space curves.

\subsection{Regularity of roots versus two-datum inverse stability}
\label{sec:9.3}

The branch of the literature closest to \S\ref{sec:engine} is the \textbf{regularity theory of roots and lifts}. The
higher-order Glaeser inequality of Ghisi--Gobbino \cite{ref23} controls the derivative of a \textbf{single radical}:
if $g\in C^{k,\alpha}(\overline I)$ and $|f|^{k+\alpha}=|g|$, then $f'\in L^p_w(I)$ with
$\frac1p+\frac1{k+\alpha}=1$. Bronshtein's theorem, in the proof of Parusiński--Rainer \cite{ref24}, states
that for a $C^{p-1,1}$ curve of hyperbolic polynomials any continuous root \textbf{of that curve} is
locally Lipschitz, with constant bounded by the $C^{n-1,1}$ norms of \textbf{that one} coefficient datum.
In \cite{ref25} a single continuous root is shown to lie in $W^{1,p}$ for every $1\le p<n/(n-1)$, and this
range is optimal; \cite{ref26} gives SBV selections and BV bounds in the multi-parameter case. The works
\cite{ref16}, \cite{ref17}, \cite{ref18}, \cite{ref19}, \cite{ref9} are of the same type, treating Lipschitz, Sobolev and SBV regularity of
roots and local uniqueness up to a group action over the regular orbit locus.

\textbf{The common shape of these conclusions is the regularity of $a\mapsto\operatorname{lift}(a)$, with
bounds using the norms of that single datum $a$.} Theorem~\ref{thm:7.1} has a different shape. It compares
\textbf{two} nearby data $f^2,g^2$ through \textbf{two given smooth signed lifts} of theirs and, after
quotienting by the \textbf{global} $\mathbb Z_2$ sign, provides an \textbf{inverse} stability modulus from
$C^0$ to $L^1$,

\[
\inf_{\varepsilon=\pm1}\|f-\varepsilon g\|_{L^1}\le C_A(B,\delta,L_0)\,
\rho\bigl(\|f^2-g^2\|_{C^0}\bigr),
\]

under the uniform non-degeneracy $f^2+(f')^2\ge\delta^2$ and likewise for $g$.

\textbf{The current state of the two-datum direction.} The survey \cite{ref8} listed the
case $d\ge3$ as \cite[Open Problem 3.8]{ref8}, its $d=2$ qualitative version being \cite[Remark 3.7]{ref8}, and the
continuity of the push-forward in the general complex-coefficient case as \cite[Open Problem 4.8]{ref8}.
Since then Parusiński--Rainer \cite{ref27} proved that in the \textbf{$C^d$ coefficient topology} the \textbf{ordered}
root solution map $\mathsf a\mapsto\lambda^\uparrow\circ\mathsf a$ for hyperbolic polynomials is
continuous into $W^{1,q}_{\mathrm{loc}}$ for every $1\le q<\infty$, failing for $q=\infty$, and state
explicitly that this \textbf{solves} \cite[Open Problem 3.8]{ref8}; \cite[Thm.~1.6]{ref27} gives in addition local
Lipschitz continuity on the interior region $\mathrm{Hyp}^\circ(d)$ where all roots are simple.
Parusiński--Rainer \cite{ref28} proved that the \textbf{unordered} $d$-valued root map of a general, non-hyperbolic
polynomial is continuous from the $C^d$ coefficient topology into $W^{1,q}_{\mathrm{loc}}$ for
$1\le q<\frac d{d-1}$, that exponent being optimal, and state explicitly that this \textbf{solves}
\cite[Open Problem 4.8]{ref8}. Both of those open problems are therefore solved. On the other hand,
\cite[Question 1.5]{ref27} explicitly asks whether these continuity statements are \textbf{uniform} and whether an
\textbf{effective modulus of continuity} exists; that remains unknown.

\textbf{How Theorem~\ref{thm:7.1} is situated.} It differs from those results in four respects, each
checkable against the sources.

\begin{enumerate}
\item \textbf{A different coefficient topology.} The continuity in \cite{ref27}, \cite{ref28} presupposes convergence of the
   coefficients in $C^d$, i.e. $C^2$ when $d=2$, whereas Theorem~\ref{thm:7.1} uses only the $C^0$ norm of the
   coefficient difference, $\eta=\|f^2-g^2\|_{C^0}$, with $C^2$ entering solely as a uniform bound
   $B$.
\item \textbf{Quantitative versus qualitative.} \cite{ref27}, \cite{ref28} give qualitative continuity, and an effective
   modulus is exactly the open \cite[Question 1.5]{ref27}; Theorem~\ref{thm:7.1} supplies an explicit log-Lipschitz
   modulus with an explicit $C_A(B,\delta,L_0)$.
\item \textbf{A different quotient.} \cite{ref27} uses the pointwise increasing rearrangement $\lambda^\uparrow$ and
   \cite{ref28} an unordered $d$-valued target, whereas Theorem~\ref{thm:7.1} quotients by the \textbf{single global} sign
   $\varepsilon\in\{\pm1\}$ relating two \textbf{given} smooth signed lifts.
\item \textbf{A different degenerate locus.} The condition $f^2+(f')^2\ge\delta^2$ \textbf{allows} $f$ to cross
   zero transversally at finitely many points, where $f^2$ has a double root and thus lies on the
   \textbf{boundary} of $\mathrm{Hyp}(2)$; the simple-root Lipschitz result \cite[Thm.~1.6]{ref27} does not apply
   there.
\end{enumerate}

Two further items are close without being priors, and the boundary is worth marking. \cite[Remark 3.7]{ref8}
gives the trivial estimate $\|\lambda-\lambda_n\|_{L^\infty}\le\|f-f_n\|_{C^0}^{1/2}$ for $d=2$ and
\textbf{non-negative} square roots, with no non-degeneracy hypothesis, no $L^1$ bound and no statement
modulo $\mathbb Z_2$; and \cite[Example 3.6]{ref8} shows that in the \textbf{Lipschitz norm} such an estimate
fails, taking $f(t)=t^2$ and $f_n(t)=t^2+n^{-2}$, while \cite[Example 1.12]{ref27}, with $f=|x|$ and
$f_n=\sqrt{x^2+n^{-2}}$, is likewise a counterexample to continuity into the $C^{0,1}$ target
topology. Neither supplies an upper bound theorem. What makes Theorem~\ref{thm:7.1} possible is precisely the
passage to $L^1$ together with a non-degeneracy condition of type $\Delta\ge\delta$.

\subsection{Neighbouring appearances of unsigned data, torsion and knotting}
\label{sec:9.4}

The information operator of this paper reads the \textbf{unsigned} pair $(\kappa,|\tau|)$. Several
works touch the same unsigned quantity, or the same combination of torsion with knot theory, from
other directions; all are complementary to rather than overlapping with the present results.

\begin{itemize}
\item \textbf{The geometric cost of a sign convention.} The crudest way to remove the sign freedom is to
  prescribe a global sign for the torsion. Bray--Jauregui \cite{ref10} show that this is a strong restriction
  in a natural class: if $\gamma\subset\mathbb R^3$ is a graph over a simple closed plane curve of
  positive curvature and has everywhere non-negative, or everywhere non-positive, torsion, then its
  torsion vanishes identically and $\gamma$ is planar; they also prove that a simple closed plane
  curve cannot be perturbed into a closed space curve of constant non-zero torsion. Adding a global
  sign convention therefore does not cheaply remove the branch freedom of \S\ref{sec:flexibility}.
\item \textbf{Naturality at low regularity.} Mucci--Saracco \cite{ref11} construct a weak binormal and a weak normal
  for non-smooth curves of finite total curvature and finite total absolute torsion, and prove that
  the length of the weak binormal equals the \textbf{total absolute torsion}. What survives naturally in
  the weak framework is thus $|\tau|$, not $\tau$.
\item \textbf{The integral-invariant side.} Milnor's curvature--torsion invariant, extended to links by
  Honma--Saeki \cite{ref12}, takes the infimum over representatives of the sum of the total curvature and the
  \textbf{total absolute torsion}, producing a knot and link invariant.
\item \textbf{Torsion and knotting on a surface.} Ghomi--Raffaelli \cite{ref29}, motivated by Nirenberg's
  problem on the isometric rigidity of tight surfaces, study closed asymptotic curves on negatively
  curved surfaces in $\mathbb R^3$ and compute the linking number of such a curve with the surface
  normal by Călugăreanu's theorem, obtaining restrictions on the planar projections of the curve;
  their results apply equally to closed curves of nowhere vanishing torsion and their binormal fields.
  That paper constrains the \textbf{topology} of a closed space curve by a differential-geometric
  hypothesis on its torsion; the present paper asks, in the opposite direction, how much of the curve
  --- and of its knot type --- is recovered from the \textbf{unsigned} torsion function itself, and
  how stably.
\end{itemize}

What separates the present paper from all four is the object being compared: here the unsigned datum
is read \textbf{pointwise along a common arclength label}. The first item above concerns the cost of a
sign convention, the second the naturality of that datum in a weak regularity class, the third
compresses the same unsigned ingredients into a \textbf{single number}, and the fourth uses a torsion
hypothesis as input to a topological conclusion rather than reading the torsion as data.

\subsection{Attribution}
\label{sec:9.5}

The classification of \S\ref{sec:branch} and the count $2^{c(\tau)}$ are known background, due to
Bony--Colombini--Pernazza \cite[\S1]{ref20}. What the present paper adds is listed in \S\ref{sec:1.4}
and proved in \S\S\ref{sec:rigidity}--\ref{sec:dichotomy}.

%% file: sections/10-questions.tex
\section{Questions raised}
\label{sec:questions}

The branch invariant $c(\tau)$ settles the structure of the signed lifts and bounds the curve fibre
sharply, and $\Delta$ locates the quantitative degeneration. Five questions are opened rather than
closed by this, and we list them in the order of the sections that raise them.

\textbf{1. Flexible representatives of prime knots.} Proposition~\ref{prop:5.20} shows that a composite
knot type satisfying (H) with $m=2$ has both a rigid and a flexible representative. The construction
of \S\ref{sec:flexibility} produces its flexibility from a connected sum, so it says nothing about
prime types. \emph{Does a prime knot type admit a flexible representative?} A related and apparently
easier question is which cardinals $c\in\mathbb N_0\cup\{\aleph_0\}$ arise as $c(\tau)$ for a curve of
a prescribed knot type, and how $c(\tau)$ may be computed for a curve given by other means; the bound
of Corollary~\ref{cor:4.5} is vacuous when $c=\aleph_0$.

\textbf{2. Geometric sharpness of the logarithm.} Theorem~\ref{thm:8.17} is an upper bound with modulus
$\rho(t)=t(1+\log_+\frac1t)$, and Theorem~\ref{thm:7.2} shows that this order is optimal
\emph{in one dimension}. \emph{Is it optimal geometrically?} Remark~\ref{rem:8.25} identifies what a
proof would need: reversing the one-directional steps of Theorem~\ref{thm:8.17}, and realising the
family of Theorem~\ref{thm:7.2} as the torsions of two closed embedded unit-speed $C^5$ curves of
equal length with equal curvature --- three nonlinear closing constraints plus embeddedness. The
companion question is whether \emph{local Lipschitz} stability holds: Remark~\ref{rem:7.3} shows every
Hölder exponent $\vartheta<1$ is admissible in one dimension, and the geometric statement at exponent
$1$ is undecided here. Nothing in this paper decides the true order of $C_*(\delta,M,L_0)$ in
$\delta$ either; the upper bound \eqref{eq:8.1} gives $O_{M,L_0}(\delta^{-4})$ as $\delta\downarrow0$, and no lower
bound on the rate is proved.

\textbf{3. Stronger a priori bounds, and uniformity in the pair.} The logarithm in
Theorem~\ref{thm:7.1} is not removed by the $C^2$ bound available here, and Theorem~\ref{thm:7.2}
exhibits a family that saturates it within that class --- the cut-off $\chi$ of that family is itself
only $C^2$. \emph{Does a uniform $C^k$ bound with $k>2$, or an analytic bound, improve the modulus?}
The heuristic after Theorem~\ref{thm:7.1} suggests only that each extra derivative lowers the
logarithmic coefficient rather than removing the logarithm. In the same direction: the constant $M_0$
of Theorem~\ref{thm:8.22} is built from the $C^5$ bound and the curvature lower bound of the sequence
generated by \textbf{one fixed} exact ambiguous pair, and the quantifier order matters; whether a bound determined by $L_0$ alone, valid simultaneously for all
pairs satisfying (H), exists is not addressed here.

\textbf{4. Weakening the comparison, the regularity, and the degree.} Three hypotheses invite
relaxation. The data are compared in a \emph{given} common arclength label (S2); taking an infimum
over labels on the left of Theorem~\ref{thm:8.17}, so that the conclusion becomes intrinsic to the
unparametrised curves, would be a stronger statement. Section~\ref{sec:8.3}
onwards fixes $r=5$ because $\tau\in C^2$ is what Theorem~\ref{thm:7.1} consumes; quantitative
stability at $r=4$, where $\tau$ is only $C^1$, would require a different one-dimensional input. And
Theorem~\ref{thm:7.1} is a $d=2$ statement whose proof uses that the sign group is $\mathbb Z_2$ and
that non-degeneracy makes zeros simple and matchable in pairs; whether an effective two-datum modulus
of the same shape holds for $d\ge3$ --- the effectiveness question of \cite[Question 1.5]{ref27} ---
is open, and the matching-and-parity argument of (P3)--(P4) is the part that would have to be
replaced.

\textbf{5. Zero curvature.} Everything here assumes $\kappa>0$ (S1), which is where the Frenet frame
lives. Points of vanishing curvature are a genuinely different problem, requiring a substitute for
the frame before the question can even be posed.

%% file: sections/A-construction-chain.tex
\section{The construction chain of \S\ref{sec:flexibility}}
\label{app:A}

This appendix contains the bookkeeping of \S\ref{sec:5.1}: the proofs of Lemmas~\ref{lem:5.1}
and \ref{lem:5.3}, and the two auxiliary lemmas that guarantee that every perturbation used in
\S\ref{sec:flexibility} preserves embeddedness, the tangle types relative to their boundary spheres,
the planar collars and the exteriors of the balls.

\subsection{Proof of Lemma~\ref{lem:5.1} (relative local knot insertion)}
\label{sec:A.1}

The construction inserts each factor inside a ball strictly smaller than its decomposing ball and
lets the curve \textbf{equal} the base circle --- as a map, not merely as a set --- outside finitely
many small open parameter arcs, so the collars are exactly circular by construction and no
normalisation of the tangle ends is needed. One standard input is used.

\begin{quote}
\textbf{Long knots.} Every knot type $K$ is realised by a \textbf{long knot}: a smooth proper embedding
$u:\mathbb R\to\mathbb R^3$ with $u(s)=(s,0,0)$ for $|s|\ge1$, whose closure in
$S^3=\mathbb R^3\cup\{\infty\}$ is a smooth knot of type $K$. This is the standard correspondence
between knots and local knots \cite[\S2.7, \S7]{ref30}: represent $K$ by a smooth knot $k\subset S^3$,
choose a point $p\in k$ and an open ball $V\ni p$ meeting $k$ in a single unknotted subarc with
$\partial V\pitchfork k$, and read the properly embedded arc $k\setminus V$ in the ball
$S^3\setminus V$ through an orientation-preserving diffeomorphism of $\operatorname{int}
(S^3\setminus V)$ with $\mathbb R^3$ that carries the arc to a properly embedded copy of $\mathbb R$
agreeing with the $x$-axis outside a compact set; a final diffeomorphism supported near that compact
set normalises it to agree with the $x$-axis for $|s|\ge1$. Here $u$ is \emph{not} assumed to be a graph over the $x$-axis, and for a non-trivial $K$ it cannot
be: a properly embedded graph $s\mapsto(s,f(s),g(s))$ with $f,g$ compactly supported is carried to the
$x$-axis by the isotopy $(1-\theta)(f,g)$ through embeddings, hence unknotted. Put
\[
\varrho_u:=\max\bigl\{\,|u(s)|:\ |s|\le1\,\bigr\}\ \ge\ 1 ,
\]
and rescale: $u_\lambda(s):=\lambda\,u(s/\lambda)$ is again a long knot of type $K$, with
$u_\lambda(s)=(s,0,0)$ for $|s|\ge\lambda$ and $|u_\lambda(s)|\le\lambda\varrho_u$ for $|s|\le\lambda$.
Thus the knotted portion may be confined to an arbitrarily small neighbourhood of the origin.
\end{quote}

\begin{proof}[Proof of Lemma~\ref{lem:5.1}]
\textbf{(1) The base circle, the balls and the tube.} Let $C_0\subset\Pi$ be the circle of curvature
$\kappa_c>0$, parametrised by arclength as $c:\mathbb R/L_c\mathbb Z\to\Pi$, and let
\[
n(t):=\kappa_c^{-1}c''(t)
\]
be its \textbf{principal} normal, which points to the centre. Then $(c'(t),n(t),e_3)$ is an orthonormal
frame with $\det\bigl(c',n,e_3\bigr)=+1$: for the counterclockwise circle
$c(t)=\kappa_c^{-1}(\cos\kappa_ct,\sin\kappa_ct,0)$ one has $c'=(-\sin,\cos,0)$,
$n=(-\cos,-\sin,0)$ and $c'\times n=e_3$. The orientation matters and is checked here because an
orientation-reversing chart would insert $\overline{K_i}$ in place of $K_i$. For
$0<\varrho<\kappa_c^{-1}$ the map
\[
\Psi:\ (\mathbb R/L_c\mathbb Z)\times \overline D^2_\varrho\longrightarrow\mathbb R^3,\qquad
\Psi\bigl(t,(y,z)\bigr):=c(t)+y\,n(t)+z\,e_3
\]
is an orientation-preserving diffeomorphism onto a closed solid-torus neighbourhood of $C_0$, with
$\Psi(t,0)=c(t)$; this is the tubular neighbourhood theorem for the embedded circle $C_0$, and here
$\Psi$ is explicit.

Fix $m$ pairwise disjoint closed parameter arcs $I_1,\dots,I_m$ and let $p_i:=c(t_i)$ be the image of
the midpoint $t_i$ of $I_i$. Choose $r_i>0$ so small that
\begin{enumerate}
\item[(a)] the closed balls $B_i:=\overline B(p_i,r_i)$ are pairwise disjoint;
\item[(b)] $c^{-1}(B_i)\subset\operatorname{int}I_i$;
\item[(c)] $\partial B_i$ meets $C_0$ transversally in exactly two points.
\end{enumerate}
Each of the three is an open condition satisfied for all small $r_i$: (a) because the $p_i$ are
distinct; (b) because $c^{-1}\bigl(\overline B(p_i,r)\bigr)$ shrinks to $\{t_i\}$ as $r\downarrow0$
and $\operatorname{int}I_i$ is a neighbourhood of $t_i$; (c) because for $0<r<2\kappa_c^{-1}$ a
sphere of radius $r$ centred at a point of a circle of radius $\kappa_c^{-1}$ meets that circle in
exactly two points, and the intersection is transversal since the two points are not antipodal on
$\partial B_i$ in the direction of $c'$. Each $B_i$ is round with centre on $\Pi$, so $R(B_i)=B_i$.
Put $B_i':=\overline B(p_i,r_i/2)$, so that $B_i'\Subset\operatorname{int}B_i$.

\textbf{(2) Inserting the factors.} Fix $i$, a long knot $u^{(i)}$ of type $K_i$ with constant
$\varrho_{u^{(i)}}$ as above, and read the local model in the tube coordinates $(t,y,z)$ of $\Psi$, the
$x$-axis of the model corresponding to the tube axis. Choose $\lambda_i>0$ so small that
\begin{equation}\label{eq:A.1}\tag{A.1}
\begin{gathered}
2\lambda_i\varrho_{u^{(i)}}<\varrho,\qquad
2\lambda_i<\tfrac12r_i,\qquad
[t_i-2\lambda_i,\ t_i+2\lambda_i]\subset\operatorname{int}I_i,\\[2pt]
\Psi\Bigl(\{|t-t_i|\le\lambda_i\varrho_{u^{(i)}}\}
\times\overline D^2_{\lambda_i\varrho_{u^{(i)}}}\Bigr)\subset\operatorname{int}B_i' ;
\end{gathered}
\end{equation}
the last inclusion holds for all small $\lambda_i$ because $\Psi$ is continuous and
$\Psi(t_i,0)=p_i\in\operatorname{int}B_i'$. Write $u^{(i)}_{\lambda_i}=(u_1,u_2,u_3)$ for the
rescaled long knot --- \textbf{no assumption is made on $u_1$ beyond $u_1(s)=s$ for
$|s|\ge\lambda_i$} --- and define $\gamma_0:\mathbb R/L_c\mathbb Z\to\mathbb R^3$ by
\[
\gamma_0(t):=\begin{cases}
\Psi\bigl(t_i+u_1(s),\ u_2(s),\ u_3(s)\bigr), & t=t_i+s\ \text{with}\ |s|\le2\lambda_i,\\[2pt]
c(t), & |t-t_i|\ge2\lambda_i\ \text{for every }i .
\end{cases}
\]
Put $J_i:=(t_i-\lambda_i,\ t_i+\lambda_i)$.

\textbf{(3) $\gamma_0$ is a $C^\infty$ embedding, and (T1).} On the \textbf{open} set
$\lambda_i<|s|<2\lambda_i$ one has $u^{(i)}_{\lambda_i}(s)=(s,0,0)$, so the first formula reads
$\Psi(t_i+s,0)=c(t_i+s)$ and coincides with the second together with all derivatives; the two
branches therefore glue to a $C^\infty$ map, and $\gamma_0=c$ \textbf{as maps} on
$(\mathbb R/L_c\mathbb Z)\setminus\bigcup_iJ_i$. By \eqref{eq:A.1}, $\gamma_0(J_i)$ and indeed the
whole block $\gamma_0([t_i-2\lambda_i,t_i+2\lambda_i])$ lie in $\operatorname{int}B_i'$: the
non-standard part by the last inclusion of \eqref{eq:A.1}, the two standard ends because
$|c(t_i\pm s)-p_i|\le|s|\le2\lambda_i<r_i/2$. This is (T1), in the stronger map-level form.

Regularity and injectivity. Write $P_i:=\{t:|t-t_i|\le2\lambda_i\}$ for the $i$-th block of
parameters and $P_0$ for the complement of $\bigcup_iP_i$; the $P_i$ are pairwise disjoint because
$P_i\subset\operatorname{int}I_i$ by \eqref{eq:A.1} and the $I_i$ are, and $\gamma_0=c$ on $P_0$.
Since $\varrho_{u^{(i)}}\ge1$, for $|s|\le2\lambda_i$ one has
$|u_1(s)|\le\max\{\lambda_i\varrho_{u^{(i)}},2\lambda_i\}\le2\lambda_i\varrho_{u^{(i)}}<\varrho
<\kappa_c^{-1}=\tfrac{L_c}{2\pi}<\tfrac{L_c}2$ by \eqref{eq:A.1}, so the first argument of $\Psi$ in
the block formula stays in an arc of length $<L_c$ about $t_i$ and the \emph{value} of $u_1(s)$, not
merely its residue modulo $L_c$, is determined by the block point. Hence on $P_i$ the map $\gamma_0$
is the diffeomorphism $\Psi$ composed with a translate of the embedding $u^{(i)}_{\lambda_i}$, so it
is injective there, while on $P_0$ it is the embedding $c$.

Now suppose $\gamma_0(t)=\gamma_0(t^\ast)$. If $t,t^\ast$ lie in one and the same $P_i$, or both in
$P_0$, injectivity of that single branch gives $t=t^\ast$. If they lie in two distinct blocks, their
images lie in the pairwise disjoint sets $\operatorname{int}B_i'$ by (T1), which is impossible.
There remains the mixed case $t=t_i+s\in P_i$ and $t^\ast\in P_0$. Then $\gamma_0(t)=c(t^\ast)$ lies
on $C_0$, which forces $u_2(s)=u_3(s)=0$, so
\[
\gamma_0(t)=\Psi\bigl(t_i+u_1(s),0\bigr)=c\bigl(t_i+u_1(s)\bigr),
\]
and $c$ being injective, $t^\ast=t_i+u_1(s)$.
Two cases, according to whether $s$ is carried into the knotted part of $u^{(i)}_{\lambda_i}$.
\begin{itemize}
\item $|u_1(s)|<\lambda_i$. Then $|t^\ast-t_i|<\lambda_i<2\lambda_i$, so $t^\ast\in P_i$, contrary to
  $t^\ast\in P_0$.
\item $|u_1(s)|\ge\lambda_i$. There $u^{(i)}_{\lambda_i}$ is already in standard form at the
  parameter $u_1(s)$, so
  $u^{(i)}_{\lambda_i}\bigl(u_1(s)\bigr)=(u_1(s),0,0)=u^{(i)}_{\lambda_i}(s)$, and injectivity of
  $u^{(i)}_{\lambda_i}$ gives $s=u_1(s)$; hence $t^\ast=t_i+u_1(s)=t_i+s=t\in P_i$, again contrary to
  $t^\ast\in P_0$.
\end{itemize}
It should be stressed that injectivity of $u^{(i)}_{\lambda_i}$ does \textbf{not} by itself rule out
$|u_1(s)|\ge\lambda_i$: what it yields there is the identity $s=u_1(s)$, which is not absurd and
merely says that the point sits at a standard end of the block. It is the position of $t^\ast$, not
the injectivity of the long knot, that discards the second case. With both cases excluded, the mixed
case does not occur, $\gamma_0$ is injective, and being an injective immersion of a compact manifold
it is a $C^\infty$ embedding.

\textbf{(4) (T2).} By (T1) the map $\gamma_0$ equals $c$ outside $\bigcup_iJ_i$, and
$\gamma_0(J_i)\subset\operatorname{int}B_i'\subset\operatorname{int}B_i$. Hence
$\gamma_0^{-1}(\partial B_i)=c^{-1}(\partial B_i)$, which by (c) of step (1) consists of exactly two
points at which the intersection is transversal, and $\gamma_0^{-1}(B_i)=c^{-1}(B_i)$ is a single
closed arc. On a neighbourhood of each of the two points $\gamma_0$ equals $c$, an exact
unit-speed circular arc of $\Pi$ of curvature $\kappa_c$: these are the collars.

\textbf{(5) (T3) and (T4).} By (T2) the pair $(B_i,\operatorname{Im}\gamma_0\cap B_i)$ is a
$1$-string tangle. Inside $B_i$ the arc coincides with an arc of $C_0$ except on $J_i$, whose image
lies in $\operatorname{int}B_i'$ and is, through the orientation-preserving chart $\Psi$, a rescaled
copy of the long knot $u^{(i)}$. Hence the tangle is obtained from the trivial $1$-string tangle by a
single local knot insertion of type $K_i$, and its type rel $\partial B_i$ is $K_i$; the orientation
check of step (1) is what makes the type $K_i$ and not $\overline{K_i}$. Since the $B_i$ are pairwise
disjoint and the curve enters and leaves each of them exactly once, iterating the standard
description of the connected sum by local knot insertion \cite[\S7]{ref30} gives
$[\gamma_0]=\#_{i=1}^mK_i$, each $\partial B_i$ exhibiting the corresponding factorisation. (When
some $K_i$ is trivial the sphere $\partial B_i$ is of course not a decomposing sphere in the strict
sense; the displayed connected sum is unaffected, and Remark~\ref{rem:5.15} records that the
construction tolerates trivial factors.) Uniqueness of the prime decomposition \cite{ref4} is not
used here.

\textbf{(6) (T5).} $R$ is the reflection in $\Pi$, $R(B_i)=B_i$ because $B_i$ is round with centre on
$\Pi$, and $R$ fixes $\Pi$ pointwise. On the two collars of $\partial B_i$ the curve lies in $\Pi$, so
$R\circ\gamma_0=\gamma_0$ \textbf{pointwise} there; the modified map therefore agrees with $\gamma_0$
on an open neighbourhood of the endpoints of $\gamma_0^{-1}(B_i)$ and is $C^\infty$. It is an
embedding: the modification is supported in $\gamma_0^{-1}(B_i)$, so nothing outside $B_i$ is touched
--- $R$ fixes $\Pi$ pointwise but not that exterior --- while $R(B_i)=B_i$ keeps the reflected arc
inside $B_i$, where it is the image of an injective map under the diffeomorphism $R$; the $B_i$ being
pairwise disjoint, modifications for distinct $i\in S$ cannot interfere. Since
$R$ reverses the orientation of $\mathbb R^3$, the tangle inside $B_i$ is replaced by its mirror
image, whose type rel $\partial B_i$ is $\overline{K_i}$; the balls being disjoint, the reflections
for distinct $i\in S$ are independent, and the resulting knot type is $\#_iK_i^{\varepsilon_i}$ with
$\varepsilon_i=-1$ exactly for $i\in S$.
\end{proof}

Only three inputs are used: the long-knot representation, the tubular neighbourhood theorem for an
embedded circle, and the description of the connected sum by local knot insertion.

\subsection{Proof of Lemma~\ref{lem:5.3} (relative removal of zero curvature)}
\label{sec:A.2}

We first record the isotopy-extension lemma used here and again in Lemma~\ref{lem:A.2}(4).

\begin{lemma}[embedding and factor types via isotopy extension]
\label{lem:A.1}
The curve $\widetilde\gamma$ of Lemma~\ref{lem:5.3} is ambient isotopic to $\gamma_0$, and inside
each $B_i$ the tangle keeps its type rel $\partial B_i$.
\end{lemma}

\begin{proof}
The straight-line homotopy $\gamma_t=(1-t)\gamma_0+t\widetilde\gamma$ is supported in
$\bigcup\operatorname{int}(B_i)$ and fixes $A$; for small $\epsilon$ every $\gamma_t$ is an
embedding, embeddings being $C^1$-open on a compact domain, so $t\mapsto\gamma_t$ is an isotopy of
embeddings. By the Isotopy Extension Theorem \cite[Thm.~8.1.3]{ref5} it extends to a compactly
supported diffeotopy $H_t$ of $\mathbb R^3$; since the velocity field vanishes on the fixed set $A$
and the trajectory is compactly contained in $\bigcup\operatorname{int}(B_i)$, a cutoff makes $H_t$
supported in $\bigcup\operatorname{int}(B_i)$, fixing the collars, each $\partial B_i$ and the
exterior.
\end{proof}

\begin{proof}[Proof of Lemma~\ref{lem:5.3}]
The set $\Sigma_2=\{(x,v,a):v\neq0,\ v\times a=0\}$ is a smooth codimension-$2$ submanifold of the
regular $2$-jet space, since the map $a\mapsto v\times a$ has rank $2$. Choose
$Z\subset\operatorname{int}K\subset K\subset U$ with $U$ precompact,
$\overline U\Subset\bigcup_i\operatorname{int}(B_i)$ and $U$ disjoint from $A$; since $Z$ is disjoint
from $A$ and contained in $\bigcup_i\operatorname{int}(B_i)$, Convention~\ref{conv:5.2} gives
$Z\subset\coprod_i[\alpha_i+2\rho,\ \beta_i-2\rho]$, so $U$ may be chosen to satisfy \eqref{eq:5.1}.

Perturb only on the compact set $K$: take a \textbf{finite} coordinate cover of $K$ and on each chart
three bumps, of position, first and second order, generating a finite family
$\{\psi_k\}\subset C_c^\infty(U,\mathbb R^3)$ such that $(s,w)\mapsto j^2\gamma_w(s)$, with
$\gamma_w=\gamma_0+\sum_kw_k\psi_k$, is submersive onto the normal directions of $\Sigma_2$ at each
point of $K$; a finite cover of the compact $K$ suffices. By the parametric form of the Jet
Transversality Theorem \cite[Thm.~3.2.8]{ref5}, for almost every small $w$ the holonomic $2$-jet
extension $j^2\gamma_w$ is transverse to $\Sigma_2$ on $K$. Since
$\dim S^1=1<\operatorname{codim}\Sigma_2=2$, transversality means empty preimage, so $\kappa>0$ on
$K$. Take the perturbation $C^2$-small: on the compact set $S^1\setminus\operatorname{int}K$ the
curve $\gamma_0$ has no zero-curvature point, so $\kappa_{\gamma_0}\ge c_0>0$, and $\kappa$ is a
continuous function of the $2$-jet, so $C^2$-smallness gives $\kappa_{\gamma_w}\ge c_0/2>0$ there.
Since $\operatorname{supp}\psi_k\subset U$, the collars and the exteriors of the balls are unchanged;
taking the perturbation also $C^1$-small preserves the embedding by Lemma~\ref{lem:A.1}.
\end{proof}

\subsection{The low-order gates}
\label{sec:A.3}

\begin{lemma}[low-order gates and arclength reparametrisation]
\label{lem:A.2}
There is $\epsilon_0>0$ such that for every $u\in C^\infty(S^1,\mathbb R^3)$ with $u\equiv0$ on
$S^1\setminus\Gamma$ and $\|u\|_{C^4}<\epsilon_0$, the curve $\eta:=\gamma_\flat+u$ satisfies:

\begin{enumerate}
\item \textbf{$\kappa>0$.} On the compact $S^1$ one has $\kappa\ge c_0>0$ for $\gamma_\flat$;
   $\kappa$ is a continuous function of the $2$-jet, so $\|u\|_{C^2}\le\|u\|_{C^4}$ small gives
   $\kappa_\eta\ge c_0/2>0$, and $\eta$ is still regular.
\item \textbf{The core image stays inside the private balls.} By Convention~\ref{conv:5.2}(P2),
   $\gamma_\flat(\overline{\Gamma_i})$ is a compact subset of $\operatorname{int}B_i$, so
   $\varrho_i:=\operatorname{dist}\bigl(\gamma_\flat(\overline{\Gamma_i}),\ \partial B_i\bigr)>0$;
   taking $\epsilon_0<\min_i\varrho_i$, the bound $\|u\|_{C^0}<\epsilon_0$ guarantees
   $\eta(\overline{\Gamma_i})\subset\operatorname{int}B_i$. Hence $\eta^{-1}(B_i)$ is still
   $[\alpha_i,\beta_i]$ and $\eta$ still meets $\partial B_i$ transversally in exactly two points.
\item \textbf{Embedding.} The set of embeddings of a compact domain is open in the $C^1$ topology,
   and $\|u\|_{C^1}\le\|u\|_{C^4}$ small preserves embeddedness.
\item \textbf{Relative tangle types.} The straight-line homotopy $\eta_t=\gamma_\flat+tu$ is
   supported in $\Gamma$ and fixes $S^1\setminus\Gamma$; by 2 and 3 each $\eta_t$ is an embedding
   with track compactly contained in $\bigcup\operatorname{int}B_i$. As in Lemma~\ref{lem:A.1} this
   extends to a compactly supported diffeotopy supported in $\bigcup\operatorname{int}B_i$, fixing
   each $\partial B_i$, the collars and the exterior. Hence $[\eta]=[\gamma_\flat]=\#_iK_i$ and the
   type of each tangle relative to $\partial B_i$ is unchanged.
\item \textbf{Reflection structure.} $\eta$ is still a $\Pi$-circular arc on a left neighbourhood of
   $c_i$ and a right neighbourhood of $d_i$, in particular near $\alpha_i,\beta_i$, and $R(B_i)=B_i$
   is unchanged, so Lemmas~\ref{lem:5.9} and \ref{lem:5.10} apply verbatim to $\eta$.
\item \textbf{Arclength reparametrisation.} Let $\hat\eta=\eta\circ\varphi^{-1}$ be the arclength
   reparametrisation of $\eta$. Then the image, the knot type, embeddedness, $\kappa>0$ and the
   simplicity of the torsion zeros are preserved (Lemma~\ref{lem:6.9}), and
   $\hat\tau=\tau\circ\varphi^{-1}$ gives $Z_\infty(\hat\tau)=\varphi\bigl(Z_\infty(\tau)\bigr)$:
   infinite-order flatness is invariant under composition with a $C^\infty$ diffeomorphism, because
   by the Faà di Bruno formula $\hat\tau^{(k)}(\varphi(t))$ is a polynomial in
   $\tau^{(1)}(t),\dots,\tau^{(k)}(t)$ and the derivatives of $\varphi^{-1}$, and symmetrically in
   the other direction. Hence the number $c$ of components and the correspondence between components
   are unchanged. \hfill$\square$
\end{enumerate}
\end{lemma}

\textbf{Why the low-order and high-order requirements do not conflict.} Items 1--5 are \textbf{open}
conditions at the $C^0$, $C^1$ or $C^2$ level, and
$\|\cdot\|_{C^0}\le\|\cdot\|_{C^1}\le\|\cdot\|_{C^2}\le\|\cdot\|_{C^4}$; meanwhile both the amplitude
$\lambda$ of Lemma~\ref{lem:5.6} and the parameter $w$ of Proposition~\ref{prop:5.7} may be taken
arbitrarily small --- the first because the identity of Lemma~\ref{lem:5.5} is linear in $\phi$, the
second because the transversality is a \textbf{full-measure} statement and therefore does not require
a large perturbation. One therefore fixes $\epsilon_0$ according to 1--5 first, and chooses $\lambda$
and $w$ inside the $\epsilon_0$-ball of $C^4$ afterwards.

%% file: sections/B-proposition-57.tex
\section{Proof of Proposition~\ref{prop:5.7}(ii)}
\label{app:B}

Part (i) was proved in the body. For (ii) the order in which objects are produced is the whole point:
the local candidates are constructed first, a finite subcover is extracted, the complete family
$\{\psi_k\}_{k=1}^N$ --- and with it the dimension $N$ --- is fixed, and only afterwards is each local
$2\times2$ block treated as a continuous function of the complete variable
$(s,w)\in V\times\mathbb R^N$ and a single parameter ball chosen. No radius $\delta$ in the
$w$-variable is available before the family, and hence $N$, has been fixed.

\textbf{Step 1 (local candidates; only the information at $w=0$ is used).} Fix $s_0\in K$. By
Lemma~\ref{lem:A.2}(1), whose proof does not use the present proposition, so that there is no circularity, we
have $\kappa_{\gamma_1}>0$ and hence $w_0:=(v\times a)(s_0)\neq0$, where
$(v,a)=(\gamma_1'(s_0),\gamma_1''(s_0))$. Choose an \textbf{open} neighbourhood $O_{s_0}\subset V$ of
$s_0$ and $\chi\in C_c^\infty(O_{s_0})$ with $\chi\equiv1$ near $s_0$, and put
\[
\psi^{(s_0)}_1:=\chi\cdot\frac{(s-s_0)^3}{6}\,w_0,\qquad
\psi^{(s_0)}_2:=\chi\cdot\frac{(s-s_0)^4}{24}\,w_0 .
\]
At $s_0$ these satisfy $\psi'=\psi''=0$, $(\psi_1''',\psi_1'''')=(w_0,0)$ and
$(\psi_2''',\psi_2'''')=(0,w_0)$. Since
$\partial_wD=\det(\psi',a,j)+\det(v,\psi'',j)+\det(v,a,\psi''')$ and the first two terms vanish at
$s_0$, similarly for $D'$, along these two directions and at the base point $w=0$ --- where
$\gamma_w=\gamma_1$, so that the block is already well defined and independent of any perturbations
added later --- one has
\[
\frac{\partial(D,D')}{\partial\bigl(w^{(s_0)}_1,w^{(s_0)}_2\bigr)}\Bigl|_{(s_0,\,w=0)}=|w_0|^2\,I_2,
\]
which is invertible. At $w=0$ this block depends on $s$ only, and continuously, so there is an
\textbf{open} neighbourhood $O'_{s_0}\subset O_{s_0}$ of $s_0$ on which it is invertible at $w=0$; shrink
once more to an open $O''_{s_0}\ni s_0$ with $\overline{O''_{s_0}}\subset O'_{s_0}$ and
$\overline{O''_{s_0}}$ compact. This step involves no radius in the $w$-variable.

\textbf{Step 2 (fixing the complete finite family).} Since $K$ is compact, extract from
$\{O''_{s_0}\}_{s_0\in K}$ a finite subcover $K\subset\bigcup_{l=1}^{L}O''_{s_l}$, put $N:=2L$ and
\[
\{\psi_k\}_{k=1}^{N}:=\bigl\{\psi^{(s_l)}_1,\psi^{(s_l)}_2\bigr\}_{l=1}^{L}
\subset C_c^\infty(V,\mathbb R^3),
\]
with $w_{2l-1},w_{2l}$ the coefficients of $\psi^{(s_l)}_1,\psi^{(s_l)}_2$. The complete family, and
with it the dimension $N$, is now fixed, so that $w\in\mathbb R^N$,
$\gamma_w:=\gamma_1+\sum_{k=1}^Nw_k\psi_k$ and $(D_{\gamma_w},D'_{\gamma_w})$ are all well defined.

\textbf{Step 3 (one $\delta$, chosen on the complete variable $(s,w)$).} For $1\le l\le L$ put
\[
B_l(s,w):=\frac{\partial(D_{\gamma_w},D'_{\gamma_w})}{\partial(w_{2l-1},w_{2l})}(s)\in\mathbb R^{2\times2}.
\]
Since $\gamma_w$ is affine in $w$ with smooth coefficients, $B_l$ --- and hence $\det B_l$ --- is a
continuous function of $(s,w)\in V\times\mathbb R^N$, jointly in all $N$ coordinates and not merely
in $s$. By Step 1, $\det B_l(\cdot,0)$ is nowhere zero on the compact set $\overline{O''_{s_l}}$, so
$\min_{\overline{O''_{s_l}}}\bigl|\det B_l(\cdot,0)\bigr|>0$; and $\det B_l$ is uniformly continuous
on the compact set $\overline{O''_{s_l}}\times\{|w|\le1\}$, so there is $\delta_l\in(0,1)$ with
$\det B_l\neq0$ on $\overline{O''_{s_l}}\times\{|w|<\delta_l\}$. Taking the minimum over the
\textbf{finitely many} $l$, set $\delta:=\min_{1\le l\le L}\delta_l>0$ and
$W:=\{w\in\mathbb R^N:|w|<\delta\}$.

\textbf{Step 4 (parametric transversality and Sard).} Fix $l$ and let $F_l:O''_{s_l}\times W\to\mathbb R^2$,
$F_l(s,w):=(D_{\gamma_w},D'_{\gamma_w})(s)$; its domain is an open manifold \textbf{without boundary}. By
the invertibility of Step 3, $\partial_wF_l$ is surjective on all of $O''_{s_l}\times W$, hence
$dF_l$ is surjective, $(0,0)$ is a regular value of $F_l$, and $F_l^{-1}(0)$ is a smooth manifold of
dimension $(1+N)-2$. Apply Sard's theorem to the projection $\pi:F_l^{-1}(0)\to W$: its regular
values form a set of full measure in $W$. If $w$ is a regular value of $\pi$ then $(0,0)$ is a
regular value of $F_l(\cdot,w):O''_{s_l}\to\mathbb R^2$; but the image of the tangent space of a
$1$-manifold is at most $1$-dimensional and cannot span $\mathbb R^2$, so
$F_l(\cdot,w)^{-1}(0)=\varnothing$. Intersecting the full-measure sets over all $l$ --- still of full
measure, hence meeting every neighbourhood of $0$ --- and using $K\subset\bigcup_lO''_{s_l}$ gives the
required $w$.

Finally, since $\dim S^1=1<\operatorname{codim}\Xi=2$, transversality can hold only as empty
intersection, so $j^4\gamma_w\pitchfork\Xi$ on $K$ is equivalent to $(D,D')\neq(0,0)$ on $K$; and
since $D'_{\gamma_w}=(D_{\gamma_w})'$ --- of the three terms in the derivative of the determinant the
first two have a repeated column and vanish --- this is in turn equivalent to all zeros of
$D_{\gamma_w}$ on $K$ being simple. \hfill$\blacksquare$

%% file: sections/C-constants.tex
\section{Explicit constants and elementary verifications}
\label{app:C}

\subsection{Proof of Lemma~\ref{lem:8.14}}
\label{sec:C.1}

\begin{proof}
Unit speed gives $\alpha'\times\alpha''=\kappa B$ and $|\alpha'\times\alpha''|=\kappa=|\alpha''|$;
inserting this into the parametrisation-free formula of Lemma~\ref{lem:2.8}, Proof 2, yields
\[
\tau_\alpha=\frac Du,\qquad D:=\det(\alpha',\alpha'',\alpha'''),\qquad u:=|\alpha''|^2=\kappa^2 .
\]
Multilinearity of the determinant, together with the vanishing of a determinant with two equal
columns, gives
\[
D'=\det(\alpha',\alpha'',\alpha^{(4)}),\qquad
D''=\det(\alpha',\alpha''',\alpha^{(4)})+\det(\alpha',\alpha'',\alpha^{(5)}).
\]
So the computation involves no derivative beyond $\alpha^{(5)}$, and $C^5$ is exactly what is needed
--- which is why \S\ref{sec:8.3} onwards fixes $r=5$, giving $\tau\in C^{r-3}=C^2$. Write
$m_j=\|\alpha^{(j)}\|_{C^0}$; the summed norm gives $m_j\le M$, $m_1=1$ and $M\ge1$. Hadamard's
inequality gives $|D|,|D'|\le M^2$ and $|D''|\le2M^2$; from $u'=2\langle\alpha'',\alpha'''\rangle$
and $u''=2|\alpha'''|^2+2\langle\alpha'',\alpha^{(4)}\rangle$ we get $M^{-2}\le u\le M^2$,
$|u'|\le2M^2$ and $|u''|\le4M^2$. Hence $w:=1/u$ satisfies $|w|\le M^2$, $|w'|=|u'|/u^2\le2M^6$ and
$|w''|\le|u''|/u^2+2|u'|^2/u^3\le4M^6+8M^{10}$. Applying Leibniz to $\tau=Dw$ and summing gives the
stated bound.
\end{proof}

\subsection{Verification of the family of Theorem~\ref{thm:7.2}}
\label{sec:C.2}

Throughout $\epsilon_0=10^{-3}$ and $0<\epsilon\le\epsilon_0$, so $T=\epsilon^{1/3}\le0.1$. Write
$p(s):=(s^2+T^2)^{-1/2}$. We first record the elementary extrema
\[
q'(x)=30x^2(1-x)^2,\qquad q''(x)=60x(1-x)(1-2x),
\]
\[
\max_{[0,1]}|q'|=q'(\tfrac12)=\tfrac{15}8,\qquad\max_{[0,1]}|q''|=\tfrac{10}{\sqrt3}<5.774,
\]
\[
\|p\|_\infty=T^{-1},\qquad\|p'\|_\infty=\tfrac{2}{3\sqrt3}T^{-2}<0.3850\,T^{-2},\qquad
\|p''\|_\infty=T^{-3},
\]
since $|p'(s)|=|s|(s^2+T^2)^{-3/2}$ is maximal at $s=T/\sqrt2$, while putting $s=xT$ gives
$p''=T^{-3}\phi(x)$ with $\phi(x)=\frac{2x^2-1}{(x^2+1)^{5/2}}$, whose modulus is maximal at $x=0$
with value $1$.

\textbf{(a) The $C^2$ splice and the periodic extension.} One has $q(0)=q'(0)=q''(0)=0$ and
$q(1)=10-15+6=1$, $q'(1)=30-60+30=0$, $q''(1)=60-180+120=0$, so at both splice points $x=0,1$ the
value and the first two derivatives of $\psi$ match those of the constants $0$ and $1$ on either
side; hence $\psi\in C^2(\mathbb R)$, the third derivative jumping there, which is immaterial. For
$\chi(s)=\psi(2(1-|s|))$: on a neighbourhood of $|s|<\frac12$ one has $\chi\equiv1$, so the failure
of $|s|$ to be differentiable at $s=0$ is harmless, while for $s\ne0$ the map $|s|$ is smooth and
composition preserves $C^2$. Thus $\chi\in C^2(\mathbb R)$ with
\[
\chi'(s)=-2\operatorname{sgn}(s)\,\psi'\bigl(2(1-|s|)\bigr),\qquad
\chi''(s)=4\,\psi''\bigl(2(1-|s|)\bigr),
\]
\[
\|\chi'\|_\infty\le2\cdot\tfrac{15}8=3.75,\qquad\|\chi''\|_\infty\le4\cdot\tfrac{10}{\sqrt3}<23.10,
\]
\[
\operatorname{supp}\chi',\ \operatorname{supp}\chi''\subset\{\tfrac12\le|s|\le1\}.
\]
Since $\operatorname{supp}\chi\subset[-1,1]\subset(-\pi,\pi)$ and $\chi$ vanishes together with its
first two derivatives for $|s|\ge1$, the $2\pi$-periodic extension of
$h_\epsilon=\lambda\epsilon\,\chi\,p$ has vanishing value and first two derivatives at the splice
points $s=\pm\pi$. Hence $h_\epsilon\in C^2(\mathbb R/2\pi\mathbb Z)$ and therefore
$g_\epsilon=f-h_\epsilon\in C^2(\mathbb R/2\pi\mathbb Z)$.

\textbf{(b) A uniform $C^2$ bound, independent of $\epsilon$.} On $\{\frac12\le|s|\le1\}$ one has
$p\le(1/4)^{-1/2}=2$ and $|p'|\le1\cdot(1/4)^{-3/2}=8$. From $T=\epsilon^{1/3}$ we get
$\epsilon T^{-1}=\epsilon^{2/3}$, $\epsilon T^{-2}=\epsilon^{1/3}$ and $\epsilon T^{-3}=1$ --- the
choice of exponent that keeps the $C^2$ norm bounded --- so Leibniz gives
\[
\|h_\epsilon\|_{C^0}\le\lambda\epsilon T^{-1}=\lambda\epsilon^{2/3},\qquad
\|h_\epsilon'\|_{C^0}\le\lambda\epsilon\bigl(\|\chi'\|_\infty\cdot2+\|p'\|_\infty\bigr)
\le\lambda\bigl(7.5\,\epsilon+0.3850\,\epsilon^{1/3}\bigr),
\]
\[
\|h_\epsilon''\|_{C^0}\le\lambda\epsilon\bigl(\|\chi''\|_\infty\cdot2+2\|\chi'\|_\infty\cdot8+\|p''\|_\infty\bigr)
\le\lambda\bigl(46.2\,\epsilon+60\,\epsilon+1\bigr),
\]
the terms containing $\chi'$ or $\chi''$ being supported in $\operatorname{supp}\chi'$, where the
bounds $p\le2$ and $|p'|\le8$ apply. Summing, in the sum-type $C^2$ norm,
\[
\|h_\epsilon\|_{C^2}\le\lambda\bigl(1+0.385\,\epsilon^{1/3}+\epsilon^{2/3}+113.7\,\epsilon\bigr)
\le0.25\times1.162<0.291\qquad(\epsilon\le10^{-3}),
\]
so that $\|f\|_{C^2}+\|g_\epsilon\|_{C^2}\le3+(3+0.291)=6.291\le7=B$, independently of $\epsilon$.

\textbf{(c) The non-degeneracy bound.} Here $f^2+(f')^2=\sin^2+\cos^2\equiv1$, and for $g=f-h$,
\[
g^2+(g')^2=1-2\bigl(fh+f'h'\bigr)+h^2+(h')^2\ \ge\ 1-2\|h\|_{C^0}-2\|h'\|_{C^0}.
\]
By (b), for $\epsilon\le10^{-3}$ we have $\|h\|_{C^0}\le0.25\times10^{-2}=2.5\times10^{-3}$ and
$\|h'\|_{C^0}\le0.25(7.5\times10^{-3}+0.0385)\le0.0115$, whence
\[
g_\epsilon^2+(g_\epsilon')^2\ \ge\ 1-0.005-0.023=0.972\ >\ 0.97\ >\ 0.81=\delta^2,
\]
and $\delta=0.9$ is likewise independent of $\epsilon$.

\textbf{(d) Two-sided bounds for $\eta_\epsilon$.} Write $f^2-g^2=h(2f-h)$. \emph{Upper bound}: using
$|\sin s|\le|s|$ and $\chi\le1$,
\[
|f(s)h(s)|=|\sin s|\,\lambda\epsilon\,\chi(s)p(s)\le\lambda\epsilon\frac{|s|}{\sqrt{s^2+T^2}}\le\lambda\epsilon,
\]
so $\eta_\epsilon\le2\|fh\|_{C^0}+\|h\|_{C^0}^2\le2\lambda\epsilon+\lambda^2\epsilon^{4/3}
=2\lambda\epsilon\bigl(1+\frac\lambda2\epsilon^{1/3}\bigr)\le2.1\,\lambda\epsilon$.
\emph{Lower bound}: it suffices to evaluate at the \textbf{fixed} point $s_*=\frac12$, where $\chi(s_*)=1$;
since $T\le0.1$,
\[
h(s_*)=\frac{\lambda\epsilon}{\sqrt{1/4+T^2}}\ \ge\ \frac{\lambda\epsilon}{\sqrt{0.26}}
\ \ge\ 1.9611\,\lambda\epsilon,\qquad h(s_*)\le\frac{\lambda\epsilon}{1/2}=2\lambda\epsilon\le5\times10^{-4},
\]
\[
2f(s_*)-h(s_*)=2\sin\tfrac12-h(s_*)\ \ge\ 0.95885-0.0005=0.95835,
\]
\[
\eta_\epsilon\ \ge\ |f^2-g^2|(s_*)=h(s_*)\bigl(2f(s_*)-h(s_*)\bigr)
\ \ge\ 1.9611\times0.95835\,\lambda\epsilon\ \ge\ 1.87\,\lambda\epsilon .
\]
Both bounds are used in step (f) of the proof of Theorem~\ref{thm:7.2}, and in opposite directions.

\subsection{The explicit bounds \eqref{eq:7.3} and \eqref{eq:8.1}}
\label{sec:C.3}

Both are stated on an explicit range of the parameters rather than as unquantified $O(\cdot)$
symbols, and both are obtained by enlarging each term of the constant separately.

\textbf{(a) Proof of \eqref{eq:7.3}.} Assume
\[
B\ge1,\qquad 0<\delta\le\min\{1,B\},
\]
so that $\delta^{-1}\le\delta^{-4}$, $B\le B^3$ and $\delta^{-k}\le\delta^{-4}$ for $k\le4$. Recall
$r=\frac{\delta}{4B}$, $a_0=\frac{\delta^2}{16B}$, $N_*=\frac{2BL_0}{\delta}$ and
$\eta_0=\frac{\delta^4}{512B^2}$. First,
$c_1=\max\{\frac4\delta,\frac1{\sqrt{2B\delta}}\}=\frac4\delta$, because
$\frac{16}{\delta^2}\ge\frac1{2B\delta}$ is equivalent to $32B\ge\delta$, which holds. The six
non-logarithmic groups of terms of $C_A$ are then bounded by
\[
\frac{2BL_0}{\eta_0}=1024\,B^3L_0\delta^{-4},\qquad
\frac{L_0}{2\delta}\le\tfrac12B^3L_0\delta^{-4},\qquad
\frac{2BL_0}{\delta^2}\le2B^3L_0\delta^{-4},
\]
\[
\frac{L_0}{a_0}=\frac{16BL_0}{\delta^2}\le16\,B^3L_0\delta^{-4},\qquad
N_*\cdot\frac{32B}{\delta^2}=64\,B^2L_0\delta^{-3}\le64\,B^3L_0\delta^{-4},
\]
\[
N_*\cdot4Bc_1^2=\frac{2BL_0}{\delta}\cdot\frac{64B}{\delta^2}=128\,B^2L_0\delta^{-3}
\le128\,B^3L_0\delta^{-4}.
\]
For the logarithmic term, $r\sqrt{2B\delta}=\frac{\sqrt2}{4}\,\delta^{3/2}B^{-1/2}\le\frac{\sqrt2}4<1$,
so
\[
\bigl|\ln\bigl(r\sqrt{2B\delta}\bigr)\bigr|=\ln\tfrac{4}{\sqrt2}+\tfrac32\ln\tfrac1\delta+\tfrac12\ln B
\le1.04+\tfrac32\,\delta^{-1}+\tfrac12B,
\]
using $\ln x\le x$ twice. Hence
\[
N_*\cdot\frac2\delta\Bigl(1+\bigl|\ln(r\sqrt{2B\delta})\bigr|\Bigr)
\le\frac{4BL_0}{\delta^2}\Bigl(2.04+\tfrac32\delta^{-1}+\tfrac12B\Bigr)
\]
\[
\le\bigl(8.16+6+2\bigr)B^3L_0\delta^{-4}\le16.2\,B^3L_0\delta^{-4}.
\]
Adding, $C_A\le(1024+0.5+2+16+64+128+16.2)\,B^3L_0\delta^{-4}\le1300\,B^3L_0\delta^{-4}$, which is
\eqref{eq:7.3}. The dominant term is the first, coming from the trivial branch (P2)(1) at the
threshold $\eta_0$, and no attempt has been made to balance it against the others.

\textbf{(b) Proof of \eqref{eq:8.1}.} Let $M\ge1$ and $0<\delta\le1$, and put $\mathsf B=2B(M)=44M^{12}$,
so $\mathsf B\ge44\ge1$ and $\delta\le1\le\mathsf B$: the range of (a) applies. Since
$44^3=85184$,
\[
C_A(\mathsf B,\delta,L_0)\le1300\,\mathsf B^3L_0\delta^{-4}=1300\cdot85184\;M^{36}L_0\delta^{-4}
\le1.108\times10^{8}\,M^{36}L_0\delta^{-4}.
\]
Therefore, using $M^{36}\delta^{-4}\ge1$,
\[
C_{\mathrm{geo}}=\sqrt2\,L_0\bigl(L_0+C_A(\mathsf B,\delta,L_0)\bigr)
\le\sqrt2\,L_0^2\bigl(1+1.108\times10^{8}M^{36}\delta^{-4}\bigr)
\]
\[
\le1.6\times10^{8}\,L_0^2M^{36}\delta^{-4},
\]
which is \eqref{eq:8.1}. Both bounds are sufficient, not sharp; in particular the exponent $-4$ in
$\delta$ is the one that this proof delivers, and \S\ref{sec:questions} records that its true order is
not determined here.

%% file: sections/D-notation.tex
\section{Notation}
\label{app:D}

\begin{center}
\tabtext{%
\begin{tabular}{@{}>{\raggedright\arraybackslash}p{0.26\textwidth}
                  >{\raggedright\arraybackslash}p{0.46\textwidth}
                  >{\raggedright\arraybackslash}p{0.20\textwidth}@{}}
\toprule
Symbol & Meaning & Introduced \\
\midrule
$\kappa,\tau$ & curvature and signed torsion, sign fixed by $B'=-\tau N$ & Def.~\ref{def:2.3} \\
\addlinespace
$F_{(\kappa,|\tau|)}$ & the unsigned Frenet datum (information operator) & Def.~\ref{def:2.3} \\
\addlinespace
$\mathrm{Fib}(d)$, $\mathrm{Fib}_{SE}(d)$, $\mathrm{Fib}_{E}(d)$ & the fibre over an unsigned datum, and its quotients by $SE(3)$ and by $E(3)$ & Def.~\ref{def:2.4} \\
\addlinespace
$\mathcal F(d)=\mathcal F(\kappa,|\tau|)$ & knot types realised in the fibre, in a common label & Def.~\ref{def:2.4} \\
\addlinespace
$\operatorname{ord}_pf$, $Z$, $Z_{\mathrm{fin}}$, $Z_\infty$ & order of a zero; zero set; finite- and infinite-order parts & Def.~\ref{def:2.6} \\
\addlinespace
$SE(3)$, $E(3)$, $\overline K$, $K^{\pm1}$ & rigid motions; with reflections; mirror knot & Def.~\ref{def:2.2} \\
\addlinespace
$\Omega$, $F=(T|N|B)$ & Frenet coefficient matrix and frame & Lem.~\ref{lem:2.10} \\
\addlinespace
$\mathcal L(f)$ & smooth signed lifts: $\{g\in C^\infty:|g|=|f|\}$ & \S\ref{sec:branch} \\
\addlinespace
$\mathcal E(f)$, $\Phi$ & locally constant signs on $S^1\setminus Z_\infty$; the bijection & Thm.~\ref{thm:3.7} \\
\addlinespace
$c(f)$ & \textbf{branch invariant}: $\#$ components of $S^1\setminus Z_\infty(f)$ & Cor.~\ref{cor:3.8} \\
\addlinespace
$\Pi$, $R$, $\kappa_c$ & mirror plane $\{z=0\}$; reflection in it; collar curvature & \S\ref{sec:flexibility} \\
\addlinespace
$B_i$, $\Gamma_i=(c_i,d_i)$, $\Gamma$ & private balls; core open arcs; their union & Lem.~\ref{lem:5.1}, Conv.~\ref{conv:5.2} \\
\addlinespace
$D_\gamma$, $D'_\gamma$ & $\det(\gamma',\gamma'',\gamma''')$, $\det(\gamma',\gamma'',\gamma'''')$ & \S\ref{sec:5.2}, Def.~\ref{def:6.1} \\
\addlinespace
(H) & prime, chiral, pairwise non-mirror factors & \S\ref{sec:5.3} \\
\addlinespace
$\gamma_\varepsilon$, $\varepsilon\in\{\pm1\}^m$ & the mirrored family & Thm.~\ref{thm:5.11} \\
\addlinespace
$\mathcal E^r$, $\mathcal G^r$, $\mathcal G^r_{L_0}$ & $C^r$ embeddings; simple torsion zeros; normalised slice & Def.~\ref{def:6.1}, Def.~\ref{def:8.1} \\
\addlinespace
$\rho(t)=t(1+\log_+\frac1t)$ & the log-Lipschitz modulus & \S\ref{sec:engine} \\
\addlinespace
$\eta$, $\Theta$ & $\Vert f^2-g^2\Vert_{C^0}$; $\Vert(\kappa_\alpha,\tau^2_\alpha)-(\kappa_\beta,\tau^2_\beta)\Vert_{C^0}$ & Thm.~\ref{thm:7.1}, Thm.~\ref{thm:8.17} \\
\addlinespace
$C_A$, $B(M)$, $C_{\mathrm{geo}}$, $C_*$ & one-dimensional, torsion, geometric and optimal constants & Thm.~\ref{thm:7.1}, Lem.~\ref{lem:8.14}, Thm.~\ref{thm:8.17}, Def.~\ref{def:8.21} \\
\addlinespace
$N_{L_0}$, $\mathcal D$ & normalisation to length $L_0$; squared datum $(\kappa,\tau^2)$ of the normalised curve & Def.~\ref{def:8.1} \\
\addlinespace
$D$, $\bar D$ & Hausdorff orbit distance mod $E(3)$; the same after normalising both arguments & Def.~\ref{def:8.1} \\
\addlinespace
$d_{\mathrm{lab}}$ & labelled orbit distance $\inf_{g\in E(3)}\Vert\alpha-g\beta\Vert_{C^0}$; $D\le d_{\mathrm{lab}}$ & Def.~\ref{def:8.1}, Lem.~\ref{lem:8.2} \\
\addlinespace
$\Delta(\alpha)$ & \textbf{quantitative non-degeneracy}: $\inf_s\sqrt{\tau^2+(\tau')^2}$ & Def.~\ref{def:8.13} \\
\addlinespace
$\mathcal K^5_{\delta,M}(L_0)$ & non-degenerate normalised strata & Def.~\ref{def:8.13} \\
\bottomrule
\end{tabular}}
\end{center}

%% file: sections/E-declarations.tex
%

\section*{Acknowledgements}
\addcontentsline{toc}{section}{Acknowledgements}

We thank Professor Song Dai, of the Center for Applied Mathematics, Tianjin University, for reading
the manuscript and for valuable guidance. We also thank Zhongzhen Hengyu Intelligent Technology
(Tianjin) Co., Ltd. for providing computing facilities and research resources in support of this
work.

\section*{Statements and Declarations}
\addcontentsline{toc}{section}{Statements and Declarations}

\paragraph{Funding.}
This work was supported by the Tianjin Municipal Science and Technology Major Program under Grant
No.\ 25ZXZSSS00680.

\paragraph{Competing Interests.}
One of the authors, Q.\ Tian, is affiliated with Zhongzhen Hengyu Intelligent Technology (Tianjin)
Co., Ltd.\ in addition to Tianjin Normal University, and that company provided computing facilities
and research resources in support of the work reported here. This is disclosed in place of a generic
declaration of no competing interests. The authors have no other relevant financial or non-financial
interests to disclose.

\paragraph{Author Contributions.}
All authors contributed to the conception of the problem and to the formulation of the results.
The proofs were developed and written by all authors jointly; all authors read, checked and approved
the final manuscript.

\paragraph{Data Availability.}
No datasets were generated or analysed during the current study. All statements of the paper are
mathematical, and the arguments supporting them are contained in the paper itself.

\paragraph{Ethics Approval.}
Not applicable: this work involves no human participants, no animals, and no personal data.